\documentclass[12pt]{amsart}
\usepackage{latexsym,amsmath, amscd, amssymb, amsthm, euscript, mathrsfs, stmaryrd}
\usepackage[all]{xy}
\usepackage{hyperref}
\usepackage[colorinlistoftodos]{todonotes}
\usepackage[cal=boondoxo, scr=boondoxo]{mathalfa}
\usepackage{graphicx}
\graphicspath{ {./images/} }

\newtheorem{thm}[subsection]{Theorem}

\newtheorem{thm/def}[subsection]{Theorem/Definition}

\newtheorem{cor}[subsection]{Corollary}

\newtheorem{lem}[subsection]{Lemma}

\newtheorem{prop}[subsection]{Proposition}

\theoremstyle{definition}
\newtheorem{notation}[subsection]{Notation}
\newtheorem{defn}[subsection]{Definition}

\theoremstyle{definition}

\theoremstyle{definition}

\newtheorem{rem}[subsection]{Remark}

\newtheorem{example}[subsection]{Example}
\numberwithin{equation}{subsection}
\newtheorem{claim}[subsection]{Claim}
\newtheorem{pg}[subsection]{}

\newcommand{\R}{\mathrm R}
\newcommand{\rH}{\mathrm H}

\newcommand{\mc}{\mathcal }

\newcommand{\Z}{\mathrm{Z}}

\newcommand{\Sp}{\text{\rm Spec}}

\newcommand{\Alg}{\text{\rm Alg}}
\newcommand{\alg}{\text{\rm Alg}}

\newcommand{\dlog}{\text{\rm dlog}}

\newcommand{\mls}{\mathscr}
\newcommand{\ms}{\mathscr}

\newcommand{\id}{\mathrm{id}}

\newcommand{\red}{\text{\rm red}}

\newcommand{\pf}{\textit{pf}}

\newcommand{\ord}{\text{\rm ord}}

\DeclareMathOperator{\Br}{Br}
\DeclareMathOperator{\Spf}{Spf}

\newcommand{\cHom}{\mathscr{H}\!{\it om}}

\newcommand{\cSpec}{\mathscr{S}\!{\it pec}}

\newcommand{\Gb}{\mathscr{G}\!{\it b}}

\DeclareMathOperator{\Pic}{Pic}
\DeclareMathOperator{\Ext}{Ext}

\DeclareMathOperator{\salg}{salg}

\DeclareMathOperator{\Mod}{Mod}
\DeclareMathOperator{\Sch}{Sch}
\DeclareMathOperator{\sm}{sm}

\DeclareMathOperator{\sh}{sh}

\newcommand{\et}{\text{\rm \'et}}

\newcommand{\liset}{\text{\rm lis-\'et}}

\newcommand{\RHom}{\text{R}\mathscr{H}\!{\it om}}

\newcommand{\ch}{\text{\rm ch}}

\newcommand{\ani}{\mathrm{ani}}

\newcommand{\op}{\mathrm{op}}
\newcommand{\coh}{\mathrm{coh}}

\newcommand{\Aff}{\mathrm{Aff}}
\newcommand{\Ab}{\mathrm{Ab}}
\newcommand{\Pro}{\mathrm{Pro}}
\newcommand{\Ind}{\mathrm{Ind}}

\newcommand{\Shv}{\mathscr{S}\!\mathscr{h}\!\mathscr{v}}

\newcommand{\Shvf}{\Shv^{\circ}}

\newcommand{\D}{\mathrm{D}}

\newcommand{\C}{\mathrm{C}}

\newcommand{\I}{\mathrm{I}}

\newcommand{\rL}{\mathrm{L}}

\newcommand{\F}{\mathrm{F}}

\newcommand{\bA}{\mathbf{A}}
\newcommand{\bB}{\mathbf{B}}
\newcommand{\bC}{\mathbf{C}}
\newcommand{\bD}{\mathbf{D}}

\newcommand{\bF}{\mathbf{F}}
\newcommand{\bG}{\mathbf{G}}
\newcommand{\bH}{\mathbf{H}}

\newcommand{\bU}{\mathbf{U}}
\newcommand{\bV}{\mathbf{V}}
\newcommand{\bW}{\mathbf{W}}
\newcommand{\bZ}{\mathbf{Z}}

\newcommand{\bmu}{\boldsymbol{\mu}}
\newcommand{\balpha}{\boldsymbol{\alpha}}

\newcommand{\brho}{\boldsymbol{\rho}}
\newcommand{\lotimes}{\otimes ^{\mathbf{L}}}

\DeclareMathOperator{\fppf}{fppf}

\DeclareMathOperator{\Spec}{Spec}

\DeclareMathOperator{\Sym}{Sym}
\DeclareMathOperator{\Hom}{Hom}

\DeclareMathOperator{\Poly}{Poly}
\DeclareMathOperator{\cocone}{Cocone}
\DeclareMathOperator{\cone}{Cone}

\newcommand{\Bl}{\mathrm{Bl}}

\usepackage{tikz-cd}

\newcommand{\cts}{\mathrm{cts}}

\newcommand{\gap}{\rule{.25cm}{0.4pt}}

\usepackage{rotating}

\AtBeginDocument{%
   \def\MR#1{}
}

\DeclareMathOperator*{\hocolim}{hocolim}

\DeclareMathOperator*{\holim}{holim}

\DeclareMathOperator{\sha}{sha}

\DeclareMathOperator{\res}{res}

\renewcommand{\leq}{\leqslant}
\renewcommand{\geq}{\geqslant}

\newcommand{\uLOmega}{\underline{\rL\Omega}}
\newcommand{\uLZOmega}{\underline{\rL\Z\Omega}}

\newcommand{\iso}{\xrightarrow{
   \,\smash{\raisebox{-0.45ex}{\ensuremath{\scriptstyle\sim}}}\,}}

\begin{document}

\title{Representability of cohomology of finite flat abelian group schemes}

\author{Daniel Bragg and Martin Olsson}


\begin{abstract}
We prove various finiteness and representability results for cohomology of finite flat abelian group schemes.  In particular, we show that if $f\colon X\rightarrow \Spec(k)$ is a projective scheme over a field $k$ and $\bG$ is a finite flat abelian group scheme over $X$ then $\R^nf_*\bG$ is an algebraic space for all $n$.  More generally, we study the derived pushforwards $\R^nf_*\bG$ for $f\colon  X\rightarrow S$ a projective morphism and $\bG$ a finite flat abelian group scheme over $X$.  We also define compactly supported cohomology for finite flat abelian group schemes, describe cohomology in terms of the cotangent complex for group schemes of height $1$, and prove higher categorical versions of our main representability results.
\end{abstract}

\maketitle

\setcounter{tocdepth}{1}
\tableofcontents

\section{Introduction}

\begin{pg}\label{pg:intro 1}
The results of this article have their origins in an attempt to understand the properties of the flat cohomology sheaves $\R^nf_*\bmu_p$ for a prime $p$ and a smooth proper morphism $f\colon X\rightarrow S$ with geometrically connected fibers. Here, $\R^nf_*\bmu_p$ denotes the $n$th higher direct image of the pushforward $f_*$ from the fppf site of $X$ to that of $S$, which may be realized as the fppf sheafification of the functor on $S$-schemes given by
\[
    T\mapsto\rH^n(X\times_ST,\bmu_p).
\]
If $p$ is everywhere invertible in $S$, then fundamental results in \'{e}tale cohomology imply that the flat cohomology groups of $\bmu_p$ are isomorphic to the corresponding \'{e}tale cohomology groups, and for each $n$ the cohomology sheaf $\R^nf_*\bmu_p$ is representable by a finite locally constant group scheme (see \ref{ex:example 0}). Without this assumption (e.g. if $S$ has characteristic $p$) the situation is more subtle. For low degrees we may understand the flat cohomology of $\bmu_p$ using the Kummer sequence
\[
    \begin{tikzcd}
        1\arrow{r}&\bmu_p\arrow{r}&\mathbf{G}_m\arrow{r}{\F}&\mathbf{G}_m\arrow{r}&1.
    \end{tikzcd}
\]
For $n=0$, the fact that the adjunction map $\ms O_S\to f_*\ms O_X$ is universally an isomorphism implies that $f_*\bG_m\simeq\bG_m$, and hence we have $\R^0f_*\bmu _{p}\simeq \bmu _{p}$. For $n=1$ we can observe that $\R^1f_*\mathbf{G}_{m}$ is simply the Picard functor $\mathrm{Pic}_{X/S}$, whose representability was shown by Artin \cite{Artin-algebraization}. From this and the Kummer sequence it follows that 
\[
    \R^1f_*\bmu_{p}\simeq\Pic_{X}[p]:=\ker(\cdot p\colon \mathrm{Pic}_{X/S}\rightarrow \mathrm{Pic}_{X/S}),
\]
and, in particular, Artin's result implies that $\R^1f_*\bmu_p$ is representable by a finite group scheme over $S$. However, in degrees $n\geq 2$, the flat cohomology groups $\R^nf_*\bG_m$ are typically not representable. Thus, the Kummer sequence is of little help, and not much seems to have been known concerning the representability of the $\R^nf_*\bmu_p$ in higher degrees, even in the case when the target $S$ is the spectrum of a field.
\end{pg}

\begin{pg} 
In this article we generalize these representability results for $\R^nf_*\bmu_{p}$ in three directions:
\begin{itemize}
     \item We consider all integers $n\geq 0$.
     \item We replace $\bmu_{p}$ by a general commutative finite locally free  group scheme over $X$.
     \item We relax the assumptions on $f$.
\end{itemize}

\end{pg}
\begin{pg}{\bf Statements of results.}
Taking inspiration from the theory of \'{e}tale cohomology and its presentation by Deligne \cite{MR463174}, we consider two sorts of representability results. The first are \emph{generic representability results}, which require minimal assumptions on the morphism $f$, namely only projectivity or properness, but only conclude representability after passing to a dense open subset of the base.
\end{pg}


\begin{thm}[Generic representability of flat cohomology]
\label{T:1.3}
Let $f\colon X\to S$ be a projective morphism of noetherian schemes of characteristic $p>0$ with $S$ reduced and let $\bG$ be a commutative finite flat group scheme over $X$. For each $n\geq 0$ there exists a dense open subscheme $U\subset S$ such that $(\R^nf_*\bG)|_U$ is representable
by a group scheme that is flat, affine, and of finite type over $U$.
\end{thm}

We remark that, in the situation of Theorem \ref{T:1.3}, the functors $\R^nf_*\bG $ need not be representable over the entire base scheme $S$, even in very nice situations. For example, if $f\colon X\to S$ is a suitably varying family of K3 surfaces, then $\R^3f_*\bmu_p$ need not be representable (see Example \ref{ex:non representable cohomology in families}). Thus, the shrinking in Theorem \ref{T:1.3} is in general necessary for representability to hold. 

\begin{cor}[Constructibility of flat cohomology]
\label{thm:constructibility of flat cohomology}
    With the assumptions of \ref{T:1.3}, for each integer $n\geq 0$, the flat cohomology sheaf $\R^nf_*\bG$ has the property that there exists a finite partition $S=\bigsqcup_i Z_i$ of $S$ into reduced locally closed subschemes $Z_i\subset S$ such that for each $i$ the restriction of $\R^nf_*\bG$ to $Z_i$ is representable by a group scheme that is flat, affine, and of
    finite type over $Z_i$.
\end{cor}
\begin{proof}
This follows from \ref{T:1.3} and noetherian induction.
\end{proof}



\begin{cor}\label{C:1}
    Let $X$ be a projective scheme over a field $k$ of characteristic $p>0$ and let $f\colon X\rightarrow\Spec(k)$ denote the structure morphism. If $\bG$ is a commutative finite flat group scheme over $X$, then for each integer $n\geq 0$ the functor $\R^nf_*\bG$ is representable by
    an affine finite type $k$-group scheme.
\end{cor}
\begin{proof} This follows from specializing Theorem \ref{T:1.3} to the case when the base is the spectrum of a field.
\end{proof}

\begin{cor}\label{C:finiteness over a finite field}
    Let $X$ be a projective scheme over a finite field $k$. If $\bG$ is a commutative finite flat group scheme over $X$, then for all $n\geq 0$ the flat cohomology group $\rH^n(X,\bG)$ is finite.
\end{cor}
\begin{proof}
Let $f\colon X\to\Spec k$ denote the structural morphism and consider the Leray spectral sequence
\[
    \rH^a(\Spec k,\R^bf_*\bG)\implies\rH^{a+b}(X,\bG).
\]
By \ref{C:1} each of the flat cohomology sheaves $\R^bf_*\bG$ is representable by an affine finite type $k$-group scheme. To prove the result it will therefore suffice to show that if $\bA$ is a commutative affine finite type group scheme over $k$ then the flat cohomology groups $\rH^n(\Spec k,\bA)=\rH^n((\Spec k)_{\fppf},\bA)$ are finite. Indeed, if $\epsilon\colon(\Sp (k))_{\fppf}\rightarrow (\Sp (k))_\et $ is the projection from the fppf topos to the small \'etale topos of $\Spec(k)$ then it is well known that $\R^i\epsilon_*\bA=0$ for $i>0$ (this uses that $k$ is perfect, see for example \cite[2.8]{Milne76}).  It follows that in this case we have $\rH ^n(\Sp (k), \bA) = \rH ^n(G_k, \bA(\bar k))$, where $G_k$ denotes the absolute Galois group of $k$ and $\bar k$ is an algebraic closure of $k$.  In particular, if $\F_k\in G_k$ is the Frobenius morphism then we see that $\rH ^n(\Sp (k), \bA)$ is calculated by the complex
\[
    \begin{tikzcd}
        \bA(\overline{k})\arrow{r}{\F_k-1}&\bA(\overline{k}).
    \end{tikzcd}
\]
By Lang's theorem \cite{MR86367} the map $\F _k-1$ is surjective on the points of the connected component of the identity; in particular, the kernel and cokernel of this map is finite and we conclude that $\rH ^n(\Sp (k), \bA)$ is finite for all $n$ (and vanishes for $n>1$).
\end{proof}

\begin{rem}
    In the situation of \ref{C:1}, if $k$ is algebraically closed then the group of $k$-points $(\R^nf_*\bG)(k)$ is identified with the flat cohomology group $\rH^n(X,\bG)$. Thus, $\rH^n(X,\bG)$ is finite if and only if $\R^nf_*\bG$ is represented by a zero-dimensional group scheme, and is infinite if and only if $\R^nf_*\bG$ is represented by a positive dimensional group scheme. For examples where the latter case occurs, see Examples \ref{ex:cohomology of alpha_p} and \ref{ex:k3 surfaces}.
\end{rem}

\begin{pg}
The other sort of representability results we consider are \emph{global representability results}, which hold over the entire base. As indicated above, even in very nice situations we cannot expect algebraicity of the individual cohomology sheaves over the entire base scheme. Instead, we consider representability properties of the complex $\R f_*\bG$, which is better behaved than the individual cohomology sheaves. For an algebraic space $S$ we write $\D(S):=\D(S_{\fppf})$ for the derived category of the category of sheaves of abelian groups on the big fppf site of $S$. For technical reasons, in the rest of this article we will prefer to work with the derived $\infty$-category $\mls D(S):=\mls D(S_{\fppf})$, but the triangulated category will suffice to state our results here. We define in \S\ref{S:section10} the notion of a \emph{stably algebraic complex} over $S$; this is a complex $\bC\in\D(S)$ with the property that for some sufficiently large integer $n$ the shift $\bC[n]$ corresponds under the Dold--Kan correspondence to an algebraic $\infty$-stack that is flat and locally of finite presentation over $S$ (we discuss the theory of algebraic $\infty $-stacks in section \ref{S:section10}). The basic examples of stably algebraic complexes are perfect complexes of vector bundles on $S$ and the complexes $\bG[j]$ for flat locally of finite presentation group schemes $\bG$ over $S$ and integers $j$.
\end{pg}

\begin{thm}[Global representability of flat cohomology]
\label{T:strongform}
Let $f\colon X\rightarrow S$ be a proper finitely presented morphism of algebraic spaces of characteristic $p>0$ and let $\bG$ be a commutative finite locally free group scheme over $X$. Suppose that (at least) one of the following holds.
\begin{enumerate}
    \item\label{item:coho etale ell} $f$ is smooth and $\bG$ is \'{e}tale.
    \item\label{item:coho AM} $f$ is flat and $\bG$ has coheight 1 (that is, the Cartier dual $\bG^{\vee}$ has height 1).
    \item\label{item:coho height 1} $f$ is flat and lci and $\bG$ has height 1.
\end{enumerate}
Then the flat cohomology $\R f_*\bG\in\D(S)$ is stably algebraic and is finitely presented over $S$. Furthermore, there exists a Zariski open cover $S'\to S$ such that the pullback $(\R f_*\bG)|_{S'}\in\D(S')$ is contained in the full triangulated subcategory of $\D(S')$ generated by commutative finite locally free group schemes and vector groups over $S'$.
\end{thm}

The proof of \ref{T:strongform} is given in section \ref{sec:algebraicity of cohomology}. We remark that it is only case~\eqref{item:coho height 1} that is essentially new; the results in cases~\eqref{item:coho etale ell} and~\eqref{item:coho AM} are straightforward consequences of standard results. Applying some general results on the representability of the cohomology sheaves of a stably algebraic complex (see \ref{prop:all coho flat, with fp 2}, \ref{cor:coho sheaves of a stably alg complex}, and \ref{cor:generic representability of stably alg complex}) we deduce the following.

\begin{cor}\label{cor:global rep corollary}
    Assume the notation of \ref{T:strongform}, and suppose that at least one of the given sets of conditions in \ref{T:strongform} hold. If $n$ is an integer such that the sheaves $\R^if_*\bG$ are representable and flat over $S$ for all $i<n$, then $\R^nf_*\bG$ is representable by a group algebraic space that is finitely presented over $S$.
    In particular, $f_*\bG$ is representable.
\end{cor}

\begin{rem}
We expect strengthenings of our results in several directions, though this seems to require new ideas and methods. First, we expect that Theorem \ref{T:1.3} should hold with the projective hypothesis on $f$ relaxed to merely proper and of finite presentation, and similarly for its corollaries \ref{thm:constructibility of flat cohomology}, \ref{C:1}, and \ref{C:finiteness over a finite field} (and further should hold for $X$ an algebraic space). Second, we expect that \ref{T:1.3}, \ref{thm:constructibility of flat cohomology}, and \ref{T:strongform} should hold without the assumption that the base $S$ has characteristic $p$. In particular, they should hold for a base scheme of mixed characteristic. Finally, we expect that the conclusion of \ref{T:strongform} should hold under more general sets of conditions on $\bG$ and $f$ than those given.
\end{rem}

\begin{rem}
    As we work with big sites, our flat cohomology sheaves are tautologically compatible with arbitrary base change.
\end{rem}

\begin{rem}\label{rem:etale case of main theorems}
    The conclusions of the above results \ref{T:1.3}, \ref{thm:constructibility of flat cohomology}, \ref{C:1}, and \ref{T:strongform} also hold when the base scheme has characteristic 0, or more generally when $\bG$ is finite \'{e}tale and has order everywhere invertible in $S$. Indeed, under these assumptions \ref{T:1.3} follows from Deligne's generic constructibility theorem \cite[Finitude, Th\'{e}or\`{e}me 1.9]{MR463174}, while \ref{T:strongform} follows from the smooth and proper base change theorems and the local acyclicity of smooth morphisms (see e.g. \cite[0GKD]{stacks-project}).
\end{rem}

\begin{rem}
    It follows from fundamental theorems on the proper pushforward of coherent sheaves that the conclusion of \ref{C:1} holds also if $\bG$ is a vector group over $X$. In fact, all of the above results, namely \ref{T:1.3}, \ref{thm:constructibility of flat cohomology}, \ref{C:finiteness over a finite field}, \ref{T:strongform}, and \ref{cor:global rep corollary}, hold in this case as well, and more generally the same is true if $\bG$ is an iterated extension of finite locally free group schemes and vector groups. 
    It is natural to ask what is the largest class of group schemes $\bG$ for which we should expect these sorts of results. We note that some assumptions on $\bG$ are necessary: for example, if $f\colon X\to\Spec k$ is a smooth proper variety over a field, then $\R^nf_*\bG_m$ is representable for $n=0,1$, but need not be representable for $n\geq 2$.
\end{rem}

\begin{rem}\label{rem:strenghtening result 2}
    Under some extra assumptions, we can use the global representability result of \ref{T:strongform} (which requires only properness) to weaken the projectivity hypothesis in \ref{T:1.3}, \ref{C:1}, and \ref{C:finiteness over a finite field} to merely properness. For example, if $f\colon X\to S$ is a proper  morphism of finite presentation of algebraic spaces of characteristic $p$ with $S$ a reduced noetherian scheme and $\bG$ is a commutative group scheme over $X$ that is an iterated extension of the group schemes in the various sets of conditions of \ref{T:strongform}, then by \ref{T:strongform} the cohomology $\R f_*\bG$ is stably algebraic. Applying some general results \ref{cor:generic representability of stably alg complex} on the algebraicity of the cohomology sheaves of a stably algebraic complex we obtain that the conclusion of \ref{T:1.3} holds for the cohomology of $\bG$. In particular, if $S=\Spec k$ is the spectrum of a field, then the conclusion of \ref{C:1} holds, and if $k$ is moreover finite, then the conclusion of \ref{C:finiteness over a finite field} holds.
\end{rem}

\begin{pg}{\bf Methods of proof.}\label{pg:methods of proof}
We outline our approach to the above results and the key technical inputs, some of which may be of independent interest. Over an algebraically closed field $k$ of positive characteristic $p$ there is a good understanding of the structure of commutative finite flat group schemes: the category of such group schemes is abelian, and has simple objects
\[
    \bmu_p,\,\balpha_p,\,\bZ/p,\,\text{and}\,\,\mathbf{Z}/\ell\,\,\,\,(\text{for prime }\ell\neq p).
\]
This follows from Dieudonn\'e theory \cite[Th\'eor\`eme 1 (iii)]{Fontaine}. We first explain our approach to understanding flat cohomology in these cases. The cohomology of $\bZ/\ell$ is well understood from the theory of \'{e}tale cohomology, so we focus here on $\bmu_p$, $\balpha_p$, and $\bZ/p$.
\end{pg}

Our representability results rely ultimately on relating finite flat group schemes to coherent sheaves, and hence flat cohomology to coherent cohomology. In some cases this is straightforward. For example, let $k$ be a field of characteristic $p$ and let $f\colon X\to\Spec k$ be a proper morphism. To understand the cohomology of $\balpha_p$, we may consider the exact sequence
\begin{equation}\label{eq:alpha_p sequence, first appearence}
        0\to\balpha_p\to\bG_a\xrightarrow{\F}\bG_a\to 0
\end{equation}
of group schemes over $X$. Applying $\R f_*$ we obtain the long exact sequence
\[
    \ldots\to\R^{n-1}f_*\bG_a\xrightarrow{\F}\R^{n-1}f_*\bG_a\to\R^nf_*\balpha_p\to\R^nf_*\bG_a\xrightarrow{\F}\R^nf_*\bG_a\to\ldots.
\]
 By flat base change \cite[02KH]{stacks-project}, the usual Zariski cohomology groups $\rH^n(X,\mls O_X)$ are compatible with arbitrary base change, in the sense that for any $k$-scheme $T$ the natural map $\rH^n(X,\mls O_X)\otimes_k\mls O_T\to\rH^n(X_T,\mls O_{X_T})$ is an isomorphism. This yields a natural isomorphism
    \[
        \bV(\rH^n(X,\mls O_X)):=\Spec(\Sym^{\ast}\rH^n(X,\mls O_X)^{\vee})\simeq\R^nf_*\bG_a
    \]
of functors, where the left side denotes the geometric vector bundle associated to the $k$-vector space $\rH^n(X,\mls O_X)$. Thus, the cohomology sheaves $\R^nf_*\bG_a$ are representable, and combined with the above long exact sequence this shows that the flat cohomology sheaves $\R^nf_*\balpha_p$ are representable. Similar arguments suffice to understand the cohomology of $\bZ/p$ via the Artin--Schreier sequence
\begin{equation}\label{eq:Z_p sequence, first appearence}
    0\to\bZ/p\to\bG_a\xrightarrow{\F-1}\bG_a\to 0.
\end{equation}
In families an additional complication arises owing to the fact that the Zariski cohomology of $\mls O_X$ may jump. Accordingly, the flat cohomology sheaves of $\bG_a$ and $\balpha_p$ need only be representable when restricted to a locally closed stratification of the base. 


More generally, with the above idea one can understand the flat cohomology of any finite flat group scheme that can be resolved by vector groups. This was observed by Artin and Milne \cite[Corollary 1.4]{ArtinMilne}. This notably does not apply to $\bmu_p$, which does not admit any nontrivial morphism to a vector group. We outline our approach in this more subtle case.
We first consider flat cohomology over affine schemes. Let $P$ be a smooth $\bF_p$-algebra. The cohomology of $\bmu_p$ over $P$ can be understood using a certain distinguished triangle
\begin{equation}\label{eq:intro hoobler for mup}
    \R\Gamma(\Spec P,\bmu_p)[1]\to\F_{P*}\Z\Omega^1_P\xrightarrow{\C-\widetilde\iota}\Omega^1_P\xrightarrow{+1}
\end{equation}
in the derived category of abelian groups $\D(\Ab)$, where $\C$ is the $\mls O_X$-linear Cartier operator and $\widetilde\iota$ is the $p$-linear map given locally by $s\omega\mapsto s^p\omega$ (more generally, there is a similar triangle when $\bmu_p$ is replaced with an arbitrary finite flat height 1 group scheme; see \ref{cor:hoobler for smooth algebras}). In particular, this sequence relates the flat cohomology of $\bmu_p$ over $P$ to certain coherent sheaves on $\Spec P$. However, due to the restriction that $P$ be smooth, this result is not strong enough for our applications to representability results, even in simple cases. For example, consider a smooth projective scheme $X$ over $\bF_p$ with structure morphism $f\colon X\to\Spec\bF_p$. To understand the sheaves $\R^nf_*\bmu_p$ one needs to control the groups $(\R^nf_*\bmu_p)(A)$ of $A$-points for arbitrary $\bF_p$-algebras $A$, including, for instance, infinitesimal thickenings of $\bF_p$. Even though $X$ is smooth, the above sequence gives information on these groups only when $A$ is smooth. To control flat cohomology over arbitrary $\bF_p$-algebras, we use the fact that the functor 
\[
    A\mapsto\R\Gamma(\Spec A,\bmu_p)
\]
on $\bF_p$-algebras has the remarkable property of being determined by its values on polynomial $\bF_p$-algebras by left Kan extension. This result is due to Bhatt and Lurie (private communication), and had also been observed by Scholze. We remark that this property is quite special to this situation; for instance, it does not hold if we replace $\bmu_p$ with $\bmu_{\ell}$ for a prime $\ell\neq p$. As a consequence, for an arbitrary $\bF_p$-algebra $R$ there is a canonical triangle
\[
        \R\Gamma(\Spec R,\bmu_p)[1]\to\F_{R*}\rL\Z\Omega^1_R\xrightarrow{\C-\widetilde \iota}\rL\Omega^1_R\xrightarrow{+1}
\]
in $\D(\Ab)$, where $\rL\Omega^1_R$ is the cotangent complex and $\rL\Z\Omega^1_R$ is a derived form of the sheaf of closed $1$-forms. In particular, this sequence relates the flat cohomology of $\bmu_p$ to complexes of coherent sheaves. In this paper we prove a generalization of these results in which $\bmu_p$ is replaced with a general commutative finite flat height 1 group scheme.

\begin{thm}\label{thm:flat cohomology is Kan extended from polynomial algebras}
    Let $R$ be a ring of characteristic $p$ and let $\bG$ be a commutative finite locally free height 1 group scheme over $R$. The functor $\Alg_R\to\mls D(\Ab), A\mapsto \R\Gamma (\Spec A, \bG_A)$ is left Kan extended from the subcategory $\Poly_R\subset\Alg_R$ of finitely generated polynomial $R$-algebras.
\end{thm}

\begin{thm}[Derived Hoobler sequence]\label{thm:derived hoobler}
    Let $R$ be a ring of characteristic $p$ and let $\bG$ be a commutative finite locally free height 1 group scheme over $R$ with Lie algebra $\mls L$ and $p$th power map $\widetilde{\rho} \colon \mls L\to\mls L$. There is a canonical distinguished triangle
    \[
        \R\Gamma(\Spec R,\bG)[1]\to\mls L\otimes^{\rL}_R\F_{R*}\rL\Z\Omega^1_R\xrightarrow{\C-\widetilde \rho}\mls L\otimes^{\rL}_R\rL\Omega^1_R\xrightarrow{+1}
    \]
    in $\mls D(\Ab)$. This triangle is determined up to canonical isomorphism by the requirements that it is functorial in $R$ and $\bG$, and that it agrees with the triangle~\eqref{eq:triangle in smooth case} when $R$ is smooth over $\bF_p$.
\end{thm}

The proofs of \ref{thm:flat cohomology is Kan extended from polynomial algebras} and \ref{thm:derived hoobler} are given in \S\ref{sec:flat coho via Kan}. We also show that both these results generalize further to the animated setting (see \ref{thm:coho is continuous} and \ref{rem:derived hoobler sequence}). In \S\ref{sec:algebraicity of cohomology} we use Theorem \ref{thm:derived hoobler} and standard results on flat proper pushforwards of perfect complexes to prove the global representability result of Theorem \ref{T:strongform}. As a further application of \ref{thm:derived hoobler}, we prove in \S\ref{sec:flat coho via Kan} the following projective bundle formula for flat cohomology.

\begin{thm}[Projective bundle formula for flat cohomology]\label{thm:proj bundle formula}
Let $X$ be an algebraic space over $\bF_p$ and let $\mls E$ be a locally free sheaf on $X$ with associated projective bundle $\pi \colon Y := \mathbf{P}(\mls E)\rightarrow X$. Let $\bG_X$ be a commutative finite locally free height 1 group scheme on $X$ with Lie algebra $\mls L$ and $p$-linear $p$th power map $\widetilde{\rho}\colon\mls L\to\mls L$. There is a canonical decomposition
    \[
        \R\pi_*\bG_Y\simeq \bG_X\oplus\bH_X[-2]
    \]
in $\D(X)$, where $\bG_Y:=\pi^{-1}\bG_X$ and $\bH_X$ is the group scheme on $X$ defined by
\[
    \bH_X=\ker\left(\boldsymbol{1}-\widetilde{\brho}:\bV(\mls L)\to\bV(\mls L)\right).
\]
Thus, for each $n$ we have a canonical isomorphism
\[
    \rH^n(Y,\bG_Y)\simeq\rH^n(X,\bG_X)\oplus\rH^{n-2}(X,\bH_X).
\]
\end{thm}

Here, $\widetilde{\brho}\colon\bV(\mls L)\to\bV(\mls L)$ denotes the homomorphism of vector groups induced by the $p$-linear map $\widetilde{\rho}$ of coherent sheaves (see \ref{pg:p linear map of vector bundles}).

\begin{example}
    For $\bG_X=\balpha_{p,X}$, we have $\widetilde{\rho}=0$, and so $\boldsymbol{1}-\widetilde{\brho}$ is an isomorphism and $\bH_X$ is trivial. Thus, \ref{thm:proj bundle formula} implies that the pullback map induces an isomorphism
        \[
            \R\pi_*\balpha_{p,Y}\simeq\balpha_{p,X}.
        \]
    For $\bG_X=\bmu _{p,X}$, we have $\bH_X\simeq(\mathbf{Z}/p)_X$, and so \ref{thm:proj bundle formula} gives a decomposition
    \[
        \R\pi_*\bmu_{p,Y}\simeq\bmu_{p,X}\oplus(\mathbf{Z}/p)_X\,[-2].
    \]
\end{example}

We now describe our approach to generic representability, and in particular Theorem \ref{T:1.3}.
We take as input Theorem \ref{T:strongform}, which implies in particular the special case of \ref{T:1.3} when $f\colon X\to S$ is smooth and projective and $\bG$ is equal to one of $\bmu_p$, $\balpha_p$, $\bZ/p$, or $\bZ/\ell$ for a prime $\ell\neq p$. We then reduce the general statement to these cases using alterations together with a certain devissage of the coefficient sheaf. The latter is not straightforward, however, as one cannot directly reduce to these cases when working with families of finite locally free group schemes. For example, one can have a finite flat group scheme $\bG$ over a curve $C$ whose restriction $\bG|_U$ to a dense open $U\subset C$ is isomorphic to $\bmu_p$, but whose fiber over $x\in C\setminus U$ is isomorphic to $\balpha_p$.  Similarly, one can construct degenerations of $\mathbf{Z}/p$ to $\balpha_p$.

To carry out this devissage we make systematic use of the now-standard methods of derived algebraic geometry and animated rings \cite{CS}. We review this theory in \S\ref{S:section4}. For our purposes, the main use of these methods is to allow a comparison of flat cohomology over a scheme and its formal completion along a closed subscheme. This is closely related to a theory of compactly supported flat cohomology, and our work in this regard is inspired by the theory of compactly supported cohomology for coherent sheaves (developed in \cite[Appendix]{RandD} and \cite{Hartshornecompact}). We prove \ref{T:1.3} in \S\ref{S:section8.5} as a consequence of a similar result for animated cohomology (Theorem \ref{T:8.1}). Our argument proving the latter consists of an induction on the relative dimension of the morphism $f\colon X\to S$, with inductive step involving the formation of suitably chosen alterations, derived blow-ups, and devissage of the coefficient sheaf to reduce to the cases covered by Theorem \ref{T:strongform}.

\begin{pg}\label{P:1.12}{\bf Relation with prior work.} We note some special cases in which our representability results were already known. Some of these results require weaker hypothesis than in   our general theorems, for example requiring only properness hypothesis or allowing the base to be of mixed characteristic.

\end{pg}

\begin{example}\label{ex:example 0}
    Let $f\colon X\to S$ be a smooth proper morphism of algebraic spaces and let $\bG$ be a commutative finite locally free group scheme over $X$. If $\bG$ is \'{e}tale and has order everywhere invertible, then it follows from the smooth proper base change theorem in \'{e}tale cohomology that for each $n$ the flat cohomology sheaf $\R^nf_*\bG$ is represented by a finite \'{e}tale group scheme over $S$. If $\bG$ admits a resolution by vector groups, then all of our results on the flat cohomology sheaves $\R^nf_*\bG$ are straightforward consequences of standard results on coherent cohomology
\end{example}

\begin{example}\label{ex:example 1}
    Let $f\colon X\to S$ be a flat, proper, and finitely presented morphism of schemes such that the adjunction map $\ms O_S\to f_*\ms O_X$ is an isomorphism and remains so after arbitrary base change. Let $\bG_S$ be a commutative finite locally free group scheme over the base $S$ and write $\bG_X:=f^{-1}\bG_S$. Using some standard results one can deduce that $\R^nf_*\bG_X$ is representable for $n=0,1$, as we explain. For $n=0$, we claim that in fact the adjunction map $\bG_S\to f_*\bG_X$ is an isomorphism. Indeed, if $T$ is an $S$-scheme then we have $f_{T*}\mls O_{X_T} = \mls O_T$ by our assumptions on the fibers of $f$, and so the map
    \[
        \bG_S(T) = \cHom_S(T,\bG_S)\rightarrow \cHom_S(X_T,\bG_S) = (f_*\bG_X)(T)
    \]
    is an isomorphism, since $\bG_S$ is affine.
    For $n=1$, a result of Raynaud \cite[6.2.1]{MR282993} shows that there is a canonical isomorphism
    \[
        \R^1f_*\bG_X\simeq\cHom_{S\text{-grp}}(\bG_S^{\vee},\Pic_{X/S})
    \]
    where $\bG_S^{\vee}$ is the Cartier dual of $\bG_S$. Combining the representability of the Picard scheme with standard results on the representability of the hom-functor, it follows that the right side and hence $\R^1f_*\bG_X$ is representable.
\end{example}


\begin{example}\label{ex:example 4}
    In the case when $q$ is a nonzero integer not invertible on the base, there is to the best of our knowledge only one nontrivial case besides the above in which the representability of $\R^nf_*\bmu_q$ for any $n\geq 2$ was previously known. This is the result, due to Max Lieblich and the first author of this article, that if $f\colon X\to S$ is a smooth proper morphism of relative dimension 2 with geometrically connected fibers whose degree 1 cohomology satisfies certain vanishing conditions, then for any nonzero integer $q$ the flat cohomology $\R^2f_*\bmu_q$ is representable \cite[Theorem 1.2]{10.1093/imrn/rnab334}. These vanishing conditions hold, for instance, for a family of K3 surfaces. We note that this result allows $S$ to be of mixed characteristic, and so is more general than the results considered in this article. The proof relies on results on moduli spaces of Azumaya algebras on surfaces, and does not generalize to higher dimensions. We remark that, when $f\colon X\to\Spec k$ is a projective surface over a field of characteristic $p$, the representability of $\R^nf_*\bmu_p$ for all $n$ was claimed by Artin \cite[Theorem 3.1]{ArtinSSK3}, with the caveat that the proof would appear elsewhere, although to the best of our knowledge it did not.
\end{example}

\begin{rem}
    Let $f\colon X\to\Spec k$ be a smooth proper morphism where $k$ is a perfect field of characteristic $p$. There is a significant literature studying the \emph{perfection} of the cohomology sheaves $\R^nf_*\bmu_p$, that is, the restriction of the functor $\R^nf_*\bmu_p$ to the category of perfect $k$-schemes. In case $\R^nf_*\bmu_p$ is representable by a group scheme say $\bH$, then the perfection of the functor $\R^nf_*\bmu_p$ is represented by the perfection of $\bH$, in the usual sense. The representability of the perfection of $\R^nf_*\bmu_p$ by a perfect group scheme turns out to be significantly easier, and was shown, for instance, in \cite[2.7]{Milne76}.
\end{rem}

\begin{rem}
 A key role in this article is played by the Hoobler sequence (see \ref{thm:derived hoobler}) relating the cohomology of height $1$ group schemes to coherent cohomology.  More generally, there should be a similar relationship between the cohomology of other group schemes, most notably $\mathbf{Z}/(p^n)(i)$ for various $i$ and $n$, relating flat cohomology to prismatic cohomology, as well as results over mixed-characteristic bases.  Recently Mondal and Madapusi have announced results in this direction.
\end{rem}
\begin{rem}In this article we develop a theory of compactly supported flat cohomology sufficient for the applications discussed above, but a more complete treatment of this theory would be useful;  in particular, establishing independence of compactifications in greater generality, making precise the relationship with Deligne and Hartshorne's theory for coherent sheaves (and more recently the theory of Clausen and Scholze), and considering the theory in mixed characteristic. 
\end{rem}
\begin{rem} We do not discuss duality for compactly supported cohomology as in the classical work of Artin and Milne \cite{ArtinMilne, Milne76}.  This is closely related to developing  a theory of $f^!$ for flat cohomology.
\end{rem}




\begin{pg}{\bf Organization of the article.}  In section \ref{S:section2new} we set up and review basic notation and results about derived categories, sheaves, and group schemes.   Sections \ref{sec:flat coho via Kan} and \ref{sec:proof of kan ext result} are devoted to the key technical result that the cohomology of height 1 group schemes can be described by Kan extension from the cohomology of smooth algebras. Combining this with results for smooth algebras we then obtained the Hoobler sequence relating cohomology of height 1 group schemes to coherent data.  In section \ref{S:section10} we develop a general formalism of algebraic complexes modeled on Simpson's work on algebraic $n$-stacks \cite{Simpson}. Section \ref{sec:algebraicity of cohomology} is devoted to the proofs of \ref{T:1.3} and \ref{T:strongform} for height $1$ group schemes and proper lci morphisms. In section \ref{S:section4} we provide some technical foundations concerning animated rings and sheaves on them. In sections \ref{sec:formal cohomology} and \ref{S:section8.5} we prove our main generic representability results. In \S\ref{S:compact} we give a brief formulation of a theory of compactly supported flat cohomology. Finally, in \S\ref{sec:examples} we record some examples of flat cohomology groups illustrating our main results.
\end{pg}

\begin{pg}{\bf Acknowledgments.}
The authors thank K. \v{C}esnavi\v{c}ius, O. Gabber,  A. J.  de Jong,  and M.  Lieblich for helpful correspondence.  Special thanks are due B. Bhatt for pointing out errors in earlier versions of the article and for explaining a proof of the continuity property for $\bmu_p$, and to the referee for a very large number of helpful suggestions and corrections. The authors gratefully acknowledge the support of the National Science Foundation under NSF Postdoctoral Research Fellowship DMS-1902875 and NSF grant DMS-1840190 (Bragg), and NSF grant DMS-1902251 (Olsson),  NSF FRG grant DMS2151946 (Olsson), and Olsson was also partially supported by the Simons Collaboration on Perfection in Algebra, Geometry, and Topology.
\end{pg}

\begin{pg}{\bf Conventions.} We fix throughout a prime number $p$ and work with schemes (or algebraic spaces, stacks, derived schemes, etc.) over $\mathbf{F}_p$.  All the locally free sheaves we consider will be of finite rank and will be referred to simply as \emph{locally free sheaves}.
\end{pg}
 We will work with algebraic spaces for most of the article, except in section \ref{S:section8.5} where some projectivity assumptions are needed. The added generality of working with algebraic spaces does not substantively change many of the arguments and therefore seems the more appropriate context. For schemes one can in many places work with the Zariski topology instead of the small \'etale topology, which is our preferred ``small'' topology. T


 By a \emph{group scheme $\bG $} over an algebraic space $X$ we mean a group algebraic space $\bG \rightarrow X$ such that for any morphism $T\rightarrow X$ with $T$ a scheme the fiber product $\bG \times _TX$ is a scheme over $T$.  A more proper terminology would be \emph{relative group scheme} but since we do not consider any other notion we simply use the term ``group scheme.''  In fact, most of the group schemes we consider will be affine over $X$ so the schematic nature of the base changes is automatic.  We say that an affine group scheme $\bG \rightarrow X$ is \emph{finite locally free} if the corresponding quasi-coherent sheaf of algebras $\mls O_{\bG }$ on $X$ is a locally free sheaf of finite rank.

\section{Sheaves, derived categories, and group schemes}\label{S:section2new}

In this section we establish some notation concerning the various sorts of sheaves and derived categories that we will use. We also recall some fundamental results concerning group schemes.

\begin{pg}\label{ssec:derived categories}{\bf Derived categories of sheaves.} Let $X$ be an algebraic space.
\end{pg}
\begin{notation}\label{N:2.2}
    We write $\Alg_X$ for the category whose objects are rings $A$ equipped with a map $\Spec A\to X$ and whose morphisms $A\to B$ are ring maps whose induced map on spectra $\Spec B\to\Spec A$ commutes with the maps to $X$. We write $\Aff_X\simeq(\Alg_X)^{\op}$ for the category of affine schemes over $X$ and $X_{\fppf}$ (resp. $X_{\et}$) for the big fppf (resp. small \'{e}tale) site of affine schemes over $X$. We write $\mls O_{\fppf}$ and $\mls O_{\et}$) for the respective structure sheaves.
\end{notation}


\begin{notation}
    We let $\mls D(X_{?})$ denote the derived $\infty$-category (in the sense of \cite[1.3.5.8]{LurieHA}) of the abelian category of sheaves of abelian groups on $X_?$, where $?=\fppf$ or $\et$. This is a stable $\infty$-category.
    We write $\D(X_?)$ for the 1-categorical derived category. There is a canonical map $\mls D(X_?)\to\D(X_?)$ which identifies the latter with the homotopy category of the former \cite[1.3.5.15]{LurieHA}. In particular, this map induces a bijection on isomorphism classes of objects. When written without a subscript, we will always intend $\mls D(X)=\mls D(X_{\fppf})$ and $\D(X)=\D(X_{\fppf})$.
\end{notation}

\begin{notation}\label{not:derived pushforward and pullback}
    Given a morphism $f\colon X\to Y$ we have the pushforward functor $\R f^?_*\colon \mls D(X_?)\to\mls D(Y_?)$. If we think confusion is unlikely, we will omit the superscript $?$, and so write simply $\R f_*$. We adopt similar conventions with respect to pullback.
\end{notation}

\begin{notation} 
    For $?=\fppf$ or $\et$,  we write $\mls D(X_{?},\mls O_{?})$ for the derived $\infty$-category of the category of sheaves of $\mls O_{?}$-modules on $X_{?}$. For locally noetherian $X$ we write $\mls D_{\coh}(X_{\et },\mls O_{\et})$ for the full subcategory of $\mls D(X_{\et},\mls O_{\et})$ spanned by the objects with coherent cohomology.
\end{notation}

\begin{notation}
     We let $\Ab$ denote the category of abelian groups, $\mls D(\Ab)$ the derived $\infty$-category of $\Ab$, and $\D(\Ab)$ the 1-categorical derived category. For a ring $R$, we will write $\mls D(\Mod_R)$ for the derived $\infty $-category of the category $\Mod_R$ of $R$-modules.
\end{notation}

\begin{notation}\label{not:infinity sheaves}
   For an $\infty$-category $\mls D$ with all finite limits we let $\Shv_{\mls D}(X)$ denote the $\infty$-category of $\infty$-sheaves on $X_{\fppf}$ with values in $\mls D$. That is, $\Shv_{\mls D}(X)$ is the full $\infty$-subcategory of $\text{Fun}(\Alg_X,\mls D)$ consisting of those functors
     \[
        \mls F\colon\Alg_X\to\mls D
     \]
     satisfying fppf descent in the $\infty$-categorical sense.
     In the special case when $\mls D=\mls D(\Ab)$, we will omit the subscript and write simply $\Shv(X)$ for $\Shv_{\mls D(\Ab)}(X)$.
\end{notation}

\begin{pg}\label{pg:the functor H}
    There is a functor
    \begin{equation}\label{eq:der cat map to shv}
        \rH\colon\mls D(X)\to\Shv(X),\hspace{1cm}\bC^{}\mapsto\rH_{\bC^{}}
    \end{equation}
    which sends a complex $\bC^{}$ of fppf sheaves on $X$ to the fppf sheafification (in the $\infty$-categorical sense) of the functor $U\mapsto \bC^{}(U)$. Explicitly, $\rH_{\bC^{}}$ is the functor
    \[
        (u\colon U\to X)\mapsto\R\Gamma(U_{\fppf},u^{-1}\bC^{}).
    \]
\end{pg}

\begin{pg}
    For a sheaf $\mls F\in\Shv(X)$ and an integer $n$ we may form the \emph{$n$th cohomology sheaf} $\mls H^n(\mls F)$ of $\mls F$ as the (1-categorical) sheaf associated to the presheaf $U\mapsto\mls H^n(\mls F(U))$. This is compatible under the embedding $\rH$ with the usual cohomology sheaves of a complex, in the sense that for a complex $\bC^{}\in\mls D(X)$ there is a natural isomorphism $\mls H^n(\bC^{})\simeq\mls H^n(\rH_{\bC^{}})$.
    
    In keeping with the usual notation for subcategories of derived categories, we use superscripts to denote the full subcategory of sheaves spanned by objects whose cohomology sheaves $\mls H^n$ vanish in a particular range. So, for example, $\Shv^{\leq 0}(X)$ is the full subcategory spanned by sheaves with $\mls H^n=0$ for $n>0$. As noted above, $\rH$ respects such conditions, and so for example restricts to a functor
    \[
        \rH\colon \mls D^{\leq 0}(X)\to\Shv^{\leq 0}(X).
    \]
\end{pg}

\begin{pg}\label{pg:hypercompleteness}
    It follows from \cite[C.3.6.1]{LurieSAG} that the functor $\rH$ is fully faithful, and induces an equivalence between $\mls D(X)$ and the full subcategory of $\Shv(X)$ spanned by those sheaves which are \emph{hypercomplete} in the sense of \cite[6.5.2.8ff]{LurieHTT} (explicitly, a sheaf $\mls F\in\Shv(X)$ is hypercomplete if $\Hom_{\Shv(X)}(\mls U,\mls F)$ is contractible for every $\mls U\in\Shv(X)$ such that $\mls H^n(\mls U)=0$ for all $n\in\bZ$). Moreover, by a result of Lurie \cite[6.5.3.12,6.5.3.13]{LurieHTT}, a sheaf $\mls F\in\Shv(X)$ is hypercomplete if and only if it is a \emph{hypersheaf}, meaning that it satisfies descent for fppf hypercoverings. Thus, the functor $\rH$ identifies $\mls D(X)$ with the category of fppf hypersheaves on $X$ with values in $\mls D(\Ab)$.
\end{pg}

\begin{pg}\label{pg:derived pushforward}
    Given a morphism $f\colon X\to S$ of algebraic spaces we write $\R f_*\colon \Shv(X)\to\Shv(S)$ for the resulting pushforward between categories of sheaves. Under the embedding $\rH$ this is compatible with the derived pushforward $\R f_*\colon \mls D(X)\to\mls D(S)$. 
\end{pg}

\begin{pg}\label{SS:qcoh}{\bf Coherent sheaves and vector bundles.}
    We will view quasicoherent sheaves over an algebraic space $X$ as sheaves on $X_{\et}$.
    We will be careful to distinguish between quasicoherent sheaves and their associated functors on big sites, and in particular between locally free sheaves and vector bundles.
\end{pg}
\begin{notation}
     If $\mls V$ is a locally free sheaf on $X$, we write
    \[
        \bV(\mls V):=\cSpec_X\Sym^{\ast}(\mls V^{\vee})
    \]
    for the associated vector bundle on $X$. The fppf sheaf $\Alg_X\to\Ab$ represented by $\bV(\mls V)$ is given by
    \[
        (u\colon U\to X)\mapsto\Hom_{\Sch_X}(U,\bV(\mls V))=\Gamma(U_{\et},u^{*}\mls V).
    \]
    A group scheme over $X$ is a \emph{vector group} if it is isomorphic to the underlying group scheme associated to a vector bundle. 
\end{notation}

\begin{example}
    When $\mls V=\mls O_X$ we have $\bV(\mls O_X)=\bG_a$.
\end{example}

\begin{pg}{\bf Perfect complexes.}\end{pg}

\begin{defn}
    A \emph{strictly perfect complex of locally free sheaves} on $X$ is a bounded complex of small \'etale sheaves whose terms are locally free sheaves on $X$ and whose differentials are $\mls O_X$-linear maps. A \emph{perfect complex of locally free sheaves} on $X$ (or simply a \emph{perfect complex on $X$}) is a complex $\mls V^{}\in\mls D(X_{\et},\mls O_{\et})$ for which there exists an \'etale cover $u\colon U\to X$ over which the pullback $\rL u^*\mls V^{}$ is isomorphic as an object of the category $\mls D(U_{\et},\mls O_{\et})$ to a strictly perfect complex of locally free sheaves on $U$.

    A \emph{strictly perfect complex of vector bundles} on $X$ is a bounded complex of big fppf sheaves on $X$ whose terms are vector bundles and whose differentials are $\mls O_{\fppf}$-linear maps. A \emph{perfect complex of vector bundles} on $X$ is a complex $\bV^{}\in\mls D(X_{\fppf},\mls O_{\fppf})$ for which there exists an \'etale cover $u\colon U\to X$ over which the pullback $\rL u^*\mls V^{}$ is isomorphic as an object of the category $\mls D(U_{\fppf},\mls O_{\fppf})$ to a strictly perfect complex  of vector bundles on $U$.
\end{defn}


\begin{notation}\label{not:perfect complexes and group schemes in Shv}
    Given a complex $\mls V^{}\in\mls D(X_{\et},\mls O_{\et})$, we write $\R\bV(\mls V^{})\in\mls D(X)$ for the image of $\mls V^{}$ under the composition
    \[
        \mls D(X_{\et},\mls O_{\et})\xrightarrow{\rL\epsilon^*}\mls D(X_{\fppf},\mls O_{\fppf})\xrightarrow{\mathrm{forget}}\mls D(X_{\fppf}),
    \]
    where $\epsilon\colon X_{\fppf}\to X_{\et}$ denotes the canonical morphism of ringed topoi. The corresponding sheaf $\rH_{\R\bV(\mls V^{})}\in\Shv(X)$ is given by
    \[
        (u\colon U\to X)\mapsto\R\Gamma(U_{\et},\rL u^{*}\mls V^{}).
    \]
    To lighten the notation, we will write $\rH_{\mls V^{}}$ for $\rH_{\R\bV(\mls V^{})}$.
\end{notation}
If $\mls V^{}$ is a perfect complex of locally free sheaves then $\R\bV(\mls V^{})$ is a perfect complex of vector bundles.

\begin{pg}{\bf $p$-linear maps.}
Key to our understanding of flat cohomology are certain complexes whose objects are locally free sheaves and whose differentials are $p$-linear maps. In this subsection we review the basic definitions and set up notation in this regard.
\end{pg}
\begin{pg}
\label{pg:p-linear maps}
    Let $X$ be an algebraic space over $\bF_p$ and let $\mls E$ and $\mls G$ be quasi-coherent sheaves of $\mls O_{X}$-modules (recall that by convention \ref{SS:qcoh} such sheaves are considered in the small \'etale topology). A map $\varphi\colon \mls E\to\mls G$ of sheaves is \emph{$p$-linear} if it is additive and satisfies $\varphi(se)=s^p\varphi(e)$ for local sections $s\in\mls O_X$ and $e\in\mls E$. Observe that giving such a $p$-linear map is equivalent to giving an $\mls O_X$-linear map $\mls E\rightarrow \F_{X*}\mls G$ or by adjunction an $\mls O_X$-linear map $\F_X^*\mls E\rightarrow \mls G$.
    In particular, for a quasi-coherent sheaf $\mls E$ the adjunction map $\F_X^*\F_{X*}\mls E\rightarrow \mls E$ corresponds to the $p$-linear map $\tau _{\mls E}\colon\F_{X*}\mls E\rightarrow \mls E$ given locally by $e\mapsto e$. Any $p$-linear map $\mls E\to\mls G$ factors uniquely through $\tau_{\mls G}$, and the association $\rho\mapsto\widetilde{\rho}:=\tau_{\mls G}\circ\rho$ defines a bijection between the set of $\mls O_X$-linear maps $\mls E\to\F_{X*}\mls G$ and the set of $p$-linear maps $\mls E\to\mls G$.

    For $i=0,1$ let $\mls E_i$ and $\mls G_i$ be $\mls O_{X}$-modules and let $\widetilde{\rho}_i\colon \mls E_i\to\mls G_i$ be $p$-linear maps. The \emph{tensor product} of $\widetilde{\rho}_0$ and $\widetilde{\rho}_1$ is the additive map
    \[
        \widetilde{\rho}_0\otimes\widetilde{\rho}_1\colon\mls E_0\otimes\mls E_1\to\mls G_0\otimes\mls G_1
    \]
    whose action on simple tensors is given by $e_0\otimes e_1\mapsto\widetilde{\rho}_0(e_0)\otimes\widetilde{\rho}_1(e_1)$. One verifies that this is indeed well-defined, and that $\widetilde{\rho}_0\otimes\widetilde{\rho}_1$ is again a $p$-linear map. This is the same as the $p$-linear map corresponding to the $\mls O_X$-linear tensor product $\F^*_X(\mls E_0\otimes\mls E_1)\simeq(\F^*_X\mls E_0)\otimes (\F^*_X\mls E_1)\to\mls G_0\otimes\mls G_1$, where the second map is the $\mls O_X$-linear tensor product of the maps $\rho_i\colon\F_X^*\mls E_i\to\mls G_i$ corresponding to $\widetilde{\rho}_i$.
    
    There is also a natural notion of pullback of $p$-linear maps: Given a $p$-linear map $\widetilde{\rho}\colon\mls E\to\mls G$ with corresponding $\mls O_X$-linear map $\rho\colon\F^*_X\mls E\to\mls G$ and a morphism $f\colon Y\to X$, the \emph{pullback} of $\widetilde{\rho}$ is the $p$-linear map
    \[
        f^*\widetilde{\rho}\colon f^*\mls E\to f^*\mls G
    \]
    corresponding to the $\mls O_X$-linear map $\F_Y^*(f^*\mls E)\simeq f^*(\F_X^*\mls E)\xrightarrow{f^*\rho}f^*\mls G$. 
\end{pg}
\begin{pg}
    \label{pg:p linear map of vector bundles}
    Suppose that $\mls E$ and $\mls G$ are locally free  sheaves on $X$. Given a $p$-linear map $\widetilde{\rho}\colon\mls E\to\mls G$ we obtain a morphism
    \[
        \widetilde{\brho}\colon\bV(\mls E)\to\bV(\mls G)
    \]
    between the associated vector groups by associating to a morphism $u\colon U\to X$ the map
    \[
        u^*\widetilde{\rho}\colon \Gamma(U,u^*\mls E)\to\Gamma(U,u^*\mls G).
    \]
    We note that $\widetilde{\brho}$ is only a homomorphism of additive group schemes, and not a morphism of vector bundles.

    \begin{example}
        For a locally free sheaf $\mls E$ there is a canonical $p$-linear map $\mls E\to\F_X^*\mls E=\mls E\otimes_{\mls O_X,\F_X^*}\mls O_X$ given by $e\mapsto e\otimes 1$. The induced map $\bV(\mls E)\to\bV(\F_X^*\mls E)\simeq\F_X^{-1}\bV(\mls E)$ on vector groups is the relative Frobenius morphism of $\bV(\mls E)$ over $X$.
    \end{example}
\end{pg}

 \begin{pg}\label{pg:derived p linear maps}
        We will also consider $p$-linear maps in the derived category. Let $\mls E$ and $\mls G$ be complexes of quasi-coherent sheaves, and consider a map $\rho\colon\rL \F_X^*\mls E \rightarrow \mls G $ in $\mls D(X_{\et},\mls O_{\et})$. Associated to such a map is a ``derived $p$-linear map'' $\widetilde{\rho}\colon\mls E\to\mls G$ defined in $\mls D(X_{\et})$ as follows. First note that there is a canonical morphism of functors $\F_{X*}\rightarrow \id_X$ on the category of sheaves of abelian groups on $X$: If $U\rightarrow X$ is a morphism and $\mls F$ is an abelian sheaf on $X$ then $(\F_{X*}\mls F)(U) = \mls F(U\times _{X, \F_X}X)$ which maps to $\mls F(U)$ by pullback along the relative Frobenius $\F_{U/X}$, and deriving this transformation we get a morphism $\tau \colon \R \F_{X*}\rightarrow \id_X$ of functors $\mls D(X)\rightarrow \mls D(X)$. We define $\widetilde{\rho}:=\tau_{\mls G}\circ\rho'$ to be the composition of the map $\rho'\colon\mls E \rightarrow \R \F _{X*}\mls G$ corresponding to $\rho$ under adjunction with the map $\tau _{\mls G}\colon\R\F_{X*}\mls G\to\mls G$.
        
        We note that, in contrast to the non-derived setting, the forgetful map $\mls D(X_{\et},\mls O_{\et})\to\mls D(X_{\et})$ is typically not faithful. Thus, we cannot in general recover $\rho$ from $\widetilde{\rho}$. Furthermore, there is not an obvious way to tensor or pull back maps of the latter form. Accordingly, when working with derived $p$-linear maps, we will choose to specify $\rho$, and refer to $\widetilde{\rho}$ as the \emph{associated derived $p$-linear map}. This enables us to define the tensor product and pullback of derived $p$-linear maps as in the non-derived case: given maps $\rho_i\colon\rL\F_X^*\mls E_i\to\mls G_i$ in $\mls D(X_{\et},\mls O_{\et})$ ($i=0,1$), we obtain a map
        \[
            \rL\F_X^*(\mls E_0\otimes^{\rL}\mls E_1)\simeq(\rL\F_X^*\mls E_0)\otimes^{\rL}(\rL\F_X^*\mls E_1)\xrightarrow{\rho_0\otimes^{\rL}\rho_1}\mls G_0\otimes^{\rL}\mls G_1
        \]
        in $\mls D(X_{\et},\mls O_{\et})$, and given a map $\rho\colon\rL\F_X^*\mls E\to\mls G$ in $\mls D(X_{\et},\mls O_{\et})$ and a morphism $f\colon Y\to X$ we get an induced map
        \[
            \rL \F _Y^*(\rL f^*\mls E )\simeq\rL f^*(\rL\F_X^*\mls E)\xrightarrow{\rL f^*\rho} \rL f^*\mls G.
        \]
        in $\mls D(X_{\et},\mls O_{\et})$.
    \end{pg}

\begin{pg}{\bf Group schemes.} We fix notation and review some fundamental results from \cite{SGA31} regarding group schemes. Let $X$ be an algebraic space over $\mathbf{F}_p$.
\end{pg}
\begin{notation}
    For a group scheme $\bG/X$, the \emph{$n$th Frobenius kernel} of $\bG$ is the group scheme
    \[
        \bG[\F^n]:=\ker\left(\F_{\bG/X}^n\colon\bG\to\bG^{(n)}\right),
    \]
    where $\bG^{(n)}$ is the fiber product $\bG \times _{X, \F_X^n}X.$
    \begin{example}
        We have $\bmu_p=\bG_m[\F]$ and $\balpha_p=\bG_a[\F]$.
    \end{example}
\end{notation}
\begin{defn}[{\cite[expos\'e $\text{VII}_{A}$, 4.1.3, D\'efinition]{SGA31}}]
    A group scheme $\bG/X$ has \emph{height $\leq n$} if $\F_{\bG /X}^n=0$, or equivalently if $\bG[\F^n] = \bG $ . We say that $\bG$ has \emph{height $n$} if $\bG$ has height $\leq n$ and does not have height $\leq n-1$. The \emph{coheight} of a commutative finite locally free group scheme $\bG$ is the height of the Cartier dual $\bG^{\vee}$.
    \begin{example}
        The trivial group scheme is the unique group scheme of height 0. The group schemes $\bmu_p$ and $\balpha_p$ have height 1 and the group schemes $\balpha_p$ and $\bZ/p$ have coheight 1.
    \end{example}
\end{defn}

\begin{pg}[The Lie algebra]
    Let $\bG$ be a group scheme over $X$. We set $\omega_{\bG}:=e^*\Omega^1_{\bG/X}$, where $e:X\to\bG$ is the identity section of $\bG$. The \emph{Lie algebra} of $\bG$ is the quasicoherent sheaf
    \[
        \mls L_{\bG}:=(\omega_{\bG})^{\vee}=e^*(\Omega^1_{\bG/X})^{\vee}.
    \]
    The bracket of differential forms endows $\mls L_{\bG}$ with a canonical structure of a sheaf of Lie algebras, and the $p$th power map (the map sending a differential form to its $p$th iterate) with a canonical restricted structure. A morphism $\varphi\colon \bG\to\bH$ of group schemes induces a map $\mathrm{d}\varphi\colon \mls L_{\bG}\to\mls L_{\bH}$ of sheaves of Lie algebras which is compatible with the restricted structure.
\end{pg}

\begin{pg}\label{P:2.34}
    Let $\bG$ be a finite locally free group scheme of height $\leq 1$ on $X$. The Lie algebra $\mls L_{\bG}$ is then locally free, and by \cite[expos\'e $\text{VII}_{A}$, Remarque 7.5]{SGA31} the functor $\bG\mapsto\mls L_{\bG}$ defines an equivalence between the category of finite locally free group schemes on $X$ of height $\leq 1$ and the category of locally free sheaves of restricted Lie algebras on $X$.
\end{pg}

\begin{pg}\label{pg:height 1 group schemes and lie algebras}
    In the commutative case, the target category for this equivalence may be described in simpler terms. Indeed, if $\bG$ is commutative, then the Lie bracket on $\mls L_{\bG}$ is trivial, and the restricted structure on $\mls L_{\bG}$ is simply a $p$-linear map
    \[
         \widetilde{\rho}_{\bG}:\mls L_{\bG}\to\mls L_{\bG}.
    \]
    Write $\rho_{\bG}:\mls L_{\bG}\to\F_{X*}\mls L_{\bG}$ for the corresponding $\mls O_X$-linear map. We refer to either of these as the \emph{$p$th power map}. The functor $\bG\mapsto (\mls L_{\bG},\rho_{\bG})$ defines an equivalence between the category of commutative finite locally free group schemes on $X$ of height $\leq 1$ and the category of pairs $(\mls L,\rho)$, where $\mls L$ is a locally free sheaf on $X$ of finite rank and $\rho \colon \mls L\to\F_{X*}\mls L$ is an $\mls O_X$-linear map.
    
   \begin{example}
        The group scheme $\bmu_p$ has Lie algebra $\mls L_{\bmu_p}\simeq\mls O_X$ and $\widetilde{\rho}_{\bmu_p}\colon\mls O_X\to\mls O_X$ is the $p$th power map $s\mapsto s^p$. The group scheme $\balpha_p$ has Lie algebra $\mls L_{\balpha_p}\simeq\mls O_X$ and $\widetilde{\rho}_{\balpha_p}=0$.
   \end{example}
    
\end{pg}

    \begin{pg}[The moduli stack of commutative height 1 group schemes]
    \label{pg:moduli stack of height 1 group schemes}
        Let $\mls B$ denote the moduli stack of commutative finite locally free height 1 group schemes over $\bF_p$. We can give an explicit quotient presentation of this stack using the above equivalence of categories. We identify $\mls B$ with the stack of pairs $(\mls L,\rho)$ as above. The rank of a locally free sheaf is a locally constant function, so we have
        \[
            \mls B=\bigsqcup_{r\geq 1}\mls B_r,
        \]
        where $\mls B_r$ classifies pairs $(\mls L, \rho )$ with $\mls L$ of rank $r$. Each $\mls B_r$ may be presented as the quotient stack 
        \[
            \mls B_r = \left[\mathrm{M}_r/\mathbf{GL}_r\right],
        \]
        where $\mathrm{M}_r$ denotes the scheme of $r\times r$ matrices and $\mathbf{GL}_r$ acts on $\mathrm{M}_r$ by the formula $M\mapsto \Lambda^{-1}M\Lambda^{(p)}$, where $M\in\mathrm{M}_r$, $\Lambda\in\mathbf{GL}_r$, and $\Lambda^{(p)}$ denotes the matrix obtained by raising the entries of $\Lambda$ to the $p$th power.
       We note that this quotient presentation shows that $\mls B$ is smooth over $\bF_p$ and has affine diagonal.
    \end{pg}
    


\begin{pg}[The B\'egueri resolution]\label{pg:begueri resolution}
    Let $\bG$ be a commutative finite locally free group scheme on $X$. The \emph{B\'egueri resolution} \cite[2.2.1]{Begueri} of $\bG$ is a certain canonical and functorial short exact sequence
    \[
        0\to\bG\to\bB^0\to\bB^1\to 0
    \]
    of group schemes over $X$ in which both $\bB^0$ and $\bB^1$ are commutative, smooth, and affine.
    Note that the existence of this sequence implies, in particular, that $\bG$ is a local complete intersection over $X$. 
    
\end{pg}

\begin{pg}[The Artin--Milne resolution]\label{pg:artin milne resolution}
    Let $\bG$ be a commutative finite locally free group scheme on $X$ of coheight 1 (e.g. $\balpha_p$ or $\mathbf{Z}/p$, but not $\bmu_p$). Let $\bV=\mathbf{V}(\omega_{\bG^{\vee}})$ be the vector bundle associated to the locally free sheaf $\omega_{\bG^{\vee}}$ and let $\bV^{(p)}$ be the Frobenius twist of $\bV$ over $X$. Artin and Milne \cite[\S 2]{ArtinMilne} construct a short exact sequence
    \begin{equation}\label{eq:AM sequence}
        0\to\bG\to\mathbf{V}\xrightarrow{\F-v}\mathbf{V}^{(p)}\to 0
    \end{equation}
    of group schemes over $X$, where $\F=\F_{\bV/X}$ is the relative Frobenius morphism of $\bV$ over $X$ and $v$ is the map induced by the Verschiebung of $\bG^{\vee}$.

    \begin{example}\label{ex:artin schreier}
        If $\bG=\balpha_p$ then~\eqref{eq:AM sequence} specializes to the sequence~\eqref{eq:alpha_p sequence, first appearence}. If $\bG=\mathbf{Z}/p$ then~\eqref{eq:AM sequence} specializes to the Artin--Schreier sequence~\eqref{eq:Z_p sequence, first appearence}.
    \end{example}

    \begin{rem}
        The Artin--Milne resolution gives a resolution by smooth group schemes whose dimensions are typically smaller than those appearing in the B\'{e}gueri resolution. This construction has the additional property of being a resolution by vector groups. Such a resolution does not exist in general: for instance, there does not exist any nontrivial map from $\bmu_p$ to a vector group.
    \end{rem}

\end{pg}

\begin{pg}[Hoobler's 4-term exact sequence]\label{P:Hoobler}
    We will use a certain exact sequence associated to a height 1 group scheme constructed by Hoobler \cite{Hoobler}. For the sake of exposition, we describe this here only in a special case which will suffice for our applications.
    
    Let $X$ be a smooth algebraic space over $\bF_p$.
    The differential $\mathrm{d}\colon\Omega^1_{X}\to\Omega^2_{X}$ is $\ms O_{X}^p$-linear, and so the pushforward
    \[
        \F_{X*}\mathrm{d}\colon\F_{X*}\Omega^1_{X}\to\F_{X*}\Omega^2_{X}
    \]
    is a map of $\ms O_{X}$-modules. Thus, the subsheaf
    \[
        \F_{X*}\Z\Omega^1_{X}:=\ker\left(\F_{X*}\mathrm{d}\colon\F_{X*}\Omega^1_{X}\to\F_{X*}\Omega^2_{X}\right)
    \]
    has a natural $\mls O_{X}$-module structure and is locally free of finite rank \cite[2.2.8]{Ill79}.  Let
        $\C_X\colon\F_{X*}\Z\Omega^1_{X}\to\Omega^1_{X}$
    denote the Cartier operator and let $\iota\colon\F_{X*}\Z\Omega^1_X\hookrightarrow\F_{X*}\Omega^1_X$ denote the inclusion. Both of these maps are $\mls O_X$-linear. We let $\widetilde{\iota}:=\tau\circ\iota\colon\F_{X*}\Z\Omega^1_X\to\Omega^1_X$ denote the $p$-linear map associated to $\iota$.


    Let $\bG$ be a commutative group scheme which is flat and locally of finite presentation over $X$ whose Frobenius kernel $\bG[\F]$ is locally free. These assumptions hold, for instance, if $\bG$ is finite locally free of height $1$, or if $\bG$ is smooth. Then the Lie algebra $\mls L=\mls L_{\bG}$ of $\bG$ is locally free, and comes with a  $p$-linear map $\widetilde{\rho}_{\bG}\colon\mls L\to\mls L$. Consider the maps
    \begin{equation}\label{eq:hooblers maps}
        \C,\widetilde{\rho}\colon\mls L\otimes\F_{X*}\Z\Omega^1_{X}\to\mls L\otimes\Omega^1_{X}
    \end{equation}
    where $\C=1_{\mls L}\otimes\C_X$ and $\widetilde{\rho}=\widetilde{\rho}_{\bG}\otimes\widetilde{\iota}$ is the tensor product of $p$-linear maps as defined above \ref{pg:p-linear maps}. Thus, $\C$ is $\mls O_X$-linear and $\widetilde{\rho}$ is $p$-linear. As described in \ref{pg:p linear map of vector bundles}, these maps extend to homomorphisms
    \[
        \bC,\widetilde{\brho}\colon\bV(\mls L\otimes\F_{X*}\Z\Omega^1_{X})\to\bV(\mls L\otimes\Omega^1_{X})
    \]
    between the associated vector groups. Hoobler shows \cite[3.2]{Hoobler} that the sequence
    \begin{equation}\label{eq:hoobler's sequence}
        \begin{tikzcd}
            0\arrow{r}&\bG\arrow{r}&\F_{X*}\bG^{(p)}\arrow{r}{\dlog}&\bV\left(\mls L\otimes\F_{X*}\Z\Omega^1_{X}\right)\arrow{r}{\bC-\widetilde{\brho}}&\mathbf{V}\left(\mls L\otimes\Omega^1_{X}\right)\arrow{r}&0
        \end{tikzcd}
    \end{equation}
    of group schemes on $X$ is exact, where $\F_{X*}\bG^{(p)}$ is the Weil restriction of $\bG^{(p)}:=\bG\times_{X,\F_X}X$ along $\F_X$ and $\bG\to\F_{X*}\bG^{(p)}$ is the adjunction map.

\end{pg}

\section{Cohomology of height 1 group schemes via Kan extension}\label{sec:flat coho via Kan}




In this and the following sections we will describe the cohomology of a height 1 group scheme over a ring of characteristic $p$ in terms of the cotangent complex. We will also show that the cohomology of height 1 group schemes on rings of characteristic $p$ is determined by its values on smooth $\bF_p$-algebras via Kan extension. The main result of these two sections is Theorem \ref{thm:main kan theorem, over a stack}.

\begin{pg}{\bf Flat cohomology over a smooth $\bF_p$-algebra.} Let $P$ be a smooth $\bF_p$-algebra and let $\bG$ be a commutative finite locally free height 1 group scheme over $P$. We will use Hoobler's exact sequence \ref{P:Hoobler} to describe the flat cohomology $\R\Gamma(\Spec P,\bG)$. Let $\epsilon\colon(\Spec P)_{\fppf}\to (\Spec P)_{\et}$ denote the projection from the big fppf topos  of $X$ to the small \'{e}tale topos of $X$.
\end{pg}
    
    \begin{prop}\label{prop:cotangent complex description on smooth algebras}
        We have $\R\epsilon_*(\F_{P*}\bG^{(p)})=0$.
    \end{prop}
    \begin{proof}
        This follows from \cite[2.2 and 2.4]{ArtinMilne}.
    \end{proof}

    Consider Hoobler's sequence~\eqref{eq:hoobler's sequence} for $\bG$ with $X=\Spec P$. Adopting the notation of \ref{P:Hoobler}, we take derived global sections to obtain a distinguished triangle
    \[
        \R\Gamma(\Spec P,[\bG\to\F_{P*}\bG^{(p)}])[1]\to\mls L\otimes_P\F_{P*}\Z\Omega^1_P\xrightarrow{\C-\widetilde{\rho}} \mls L\otimes_P\Omega^1_P\xrightarrow{+1}
    \]
    in $\mls D(\Ab)$, where the complex $[\bG \rightarrow \F_{P*}\bG^{(p)}]$ has terms in degrees 0 and 1. By Proposition \ref{prop:cotangent complex description on smooth algebras} the projection map $[\bG \rightarrow \F_{P*}\bG^{(p)}]\rightarrow \bG$ induces an isomorphism upon applying $\R\Gamma(\Spec P,\rule{.25cm}{0.4pt})$. We obtain the following result.

  \begin{cor}\label{cor:hoobler for smooth algebras} Hoobler's 4-term sequence \eqref{eq:hoobler's sequence} induces a distinguished triangle
  \begin{equation}\label{eq:triangle in smooth case}
        \R\Gamma(\Spec P,\bG)[1]\to\mls L\otimes_P\F_{P*}\Z\Omega^1_P\xrightarrow{\C-\widetilde \rho } \mls L\otimes_P\Omega^1_P\xrightarrow{+1}
    \end{equation}
in $\mls D(\Ab)$.
\end{cor}



\begin{pg}{\bf Kan extension.}
    We will analyze flat cohomology of height 1 group schemes over general $\bF_p$-algebras by Kan extension from smooth $\bF_p$-algebras. We recall here some facts regarding Kan extensions. For a ring $R$, we write $\Alg_R$ for the category of $R$-algebras and $\Poly_R\subset\Alg_R$ for the full subcategory of finitely generated polynomial algebras over $R$. We say that a functor
            \[
                \mls F\colon \Alg_R\to\mls D(\Ab)
            \]
    is \emph{left Kan extended} from a full subcategory $\mls C\subset\Alg_R$ if $\mls F$ is a left Kan extension of the restriction $\mls F\big\lvert_{\mls C}$ along the inclusion $\mls C\hookrightarrow\Alg_R$. Recall from \cite[5.5.8.15]{LurieHTT} that for any functor $\mls F\colon\Poly_R\to \mls D(\Ab )$ the left Kan extension $\rL \mls F\colon\Alg _R\to\mls D(\Ab )$ may be characterized as the unique extension of $\mls F$ which commutes with sifted colimits.
    
      \end{pg}  
        \begin{example}[The cotangent complex]
        \label{ex:cotangent complex}
        We consider the absolute cotangent complex as a functor
        \[
            \uLOmega^1\colon\Alg_{\bF_p}\to\mls D(\Ab),\hspace{1cm}A\mapsto\rL\Omega^1_{A}.
        \]
        This functor is a left Kan extension along the inclusion $\Poly_R\subset\Alg_R$ of the functor
        \[
            \underline{\Omega}^1\colon\Poly_{\bF_p}\to\mls D(\Ab),\hspace{1cm}P\mapsto\Omega^1_{P}.
        \]
        For an $\mathbf{F}_p$-algebra $R$, we let
        \[
            \uLOmega^1_{R}\colon\Alg _R\rightarrow \mls D(\Ab ),\hspace{1cm}A\mapsto\rL\Omega^1_A
        \]
        denote the precomposition of $\uLOmega^1$ by the forgetful functor $\Alg_R\rightarrow \Alg _{\mathbf{F}_p}$. The forgetful functor commutes with sifted colimits, so $\uLOmega^1_R$ does as well. Thus, $\uLOmega^1_R$ is left Kan extended from $\Poly_R$.
        Similarly, we define a functor
        \[
            \uLZOmega ^1\colon\Alg _{\bF_p}\rightarrow \mls D(\Ab ),\hspace{1cm}A\mapsto\rL\Z\Omega^1_A
        \]
        by left Kan extension of the functor of closed forms $P\mapsto \Z\Omega^1_{P}\colon\Poly_{\mathbf{F}_p}\rightarrow \mls D(\Ab )$, and for an $\bF_p$-algebra $R$ we let $\uLZOmega^1_R\colon\Alg_R\to\mls D(\Ab)$ denote its precomposition with the functor $\Alg _R\rightarrow \Alg _{\mathbf{F}_p}$. As before, $\uLZOmega^1_{R}$ is left Kan extended from $\Poly_R$.
        \end{example}
        \begin{rem} Note the distinction between $\uLOmega ^1_R$ and $\rL \Omega ^1_R$. The former is a functor on $    \Alg _R$ while the latter is a complex of $R$-modules.
        \end{rem}
        
        \begin{example}\label{ex:functor from a complex}
            Let $\mls V^{}$ be a complex of $R$-modules. Consider the functor
            \[  
                \rH_{\mls V^{}}\colon\Alg_R\to\mls D(\Ab),\hspace{1cm}A\mapsto\mls V^{}\otimes^{\rL}_RA.
            \]
            If $\mls V^{}$ is a perfect complex then this functor is left Kan extended from $\Poly_R$ (equivalently, commutes with sifted colimits). To see this, note first that as $\Spec R$ is affine we may assume that $\mls V^{}$ is strictly perfect \cite[2.3.1(d)]{MR1106918}.
            Arguing using the tautological filtration on $\mls V^{}$ we reduce to the case when $\mls V^{}$ is a single projective $R$-module say $P$. Then $P$ is a retract of a finitely generated free $R$-module, and hence is strongly of finite presentation, meaning that $\Hom_R(P,\gap)$ commutes with sifted colimits (see \cite[5.1.3]{CS} ff.). It follows that the functor $\R\Hom_R(P,\gap)\colon\Mod_R\to\mls D(\Ab)$ also commutes with sifted colimits. Applying these observations to $P^{\vee}$, we get that the functor $\R\Hom_R(P^{\vee},\gap)\simeq P\otimes^{\rL}_R\gap\colon\Mod_R\to\mls D(\Ab)$ commutes with sifted colimits. The claim follows by noting that $\rH_{\mls V}$ agrees with the precomposition of this functor with the forgetful functor $\Alg_R\to\Mod_R$, which also preserves sifted colimits (as in e.g. \cite[5.1.3]{CS} ff.).
        \end{example}

\begin{pg}{\bf Hoobler's complex in general.}
        Let $R$ be a ring of characteristic $p$, let $\mls V^{}_R$ be a complex of $R$-modules, and let $\rho_R\colon\rL\F_R^*\mls V^{}_R\to\mls V^{}_R$ be a map in $\mls D(\Mod_R)$. We define a map
        \[
            \C\,\,\,\,(\text{resp. }\widetilde{\rho})\colon\mls V^{}_R\otimes^{\rL}_{R}\F_{R*}\rL\Z\Omega^1_{R}\to\mls V^{}_R\otimes^{\rL}_R\rL\Omega^1_R
        \]
        in $\mls D(\Mod_R)$ (resp. $\mls D(\Ab)$) as follows. Let
        \[
            \C_R\colon\F_{R*}\rL\Z\Omega^1_R\to\rL\Omega^1_R\hspace{1cm}(\text{resp.}\,\,\iota_R\colon\F_{R*}\rL\Z\Omega^1_R\to\F_{R*}\rL\Omega^1_R)
        \]
        denote the map in $\mls D(\Mod_R)$ obtained by Kan extension from polynomial $\bF_p$-algebras of the Cartier operators $\C_P\colon\F_{P*}\Z\Omega^1_P\to\Omega^1_P$ (resp. the inclusions $\F_{P*}\Z\Omega^1_P\to\F_{P*}\Omega^1_P$). We define $\C$ to be the derived tensor product $\C:=1_{\mls V^{}_R}\otimes^{\rL}_R\C_R$, and define $\widetilde{\rho}$ using the derived tensor product of derived $p$-linear maps as in \ref{pg:derived p linear maps}; explicitly, $\widetilde{\rho}$ is the derived $p$-linear map associated to the map
        \[
            \rL\F_R^*(\mls V^{}_R\otimes^{\rL}_{R}\F_{R*}\rL\Z\Omega^1_{R})\simeq(\rL\F_R^*\mls V_R)\otimes^{\rL}_R(\rL\F_R^*\F_{R*}\rL\Z\Omega^1_R)\xrightarrow{\rho_R\otimes^{\rL}\iota'_R}\mls V^{}_R\otimes^{\rL}_R\rL\Omega^1_R
        \]
        where $\iota'_R$ is the adjoint to $\iota_R$.
    \end{pg}
    \begin{notation}\label{not:notation for C_G}
        Let $\mls V^{}_R$ be a complex of $R$-modules and let $\rho_R\colon\rL\F_R^*\mls V^{}_R\to\mls V^{}_R$ be a map in $\mls D(\Mod_R)$. We define the associated \emph{Hoobler complex} by
        \begin{equation}\label{eq:def of Hoobler complex}
            \mls C(\mls V^{}_R,\rho_R):=\cocone\left(\mls V^{}_R\otimes^{\rL}_{R}\F_{R*}\rL\Z\Omega^1_{R}\xrightarrow{\C-\widetilde \rho}\mls V^{}_R\otimes^{\rL}_R\rL\Omega^1_R\right)\in\mls D(\Ab).
        \end{equation}
        As in \ref{pg:derived p linear maps}, these constructions may be pulled back in a natural way, and so we obtain a functor
        \[
            \mls C_{(\mls V^{},\rho)}\colon\Alg_R\to\mls D(\Ab),\hspace{1cm}A\mapsto\mls C(\mls V^{}_A,\rho_A).
        \]
    \end{notation}

    
        \begin{lem}\label{lem:kan extension for C_R}
            Let $R$ be a ring of characteristic $p$ and let $(\mls V^{}_R,\rho_R)$ be a pair as in \ref{not:notation for C_G}. If $\mls V^{}_R$ is perfect, then the functor $\mls C_{(\mls V^{}_R,\rho_R)}\colon\Alg_R\to\mls D(\Ab)$ is left Kan extended from $\Poly_R$.
        \end{lem}
        \begin{proof}
            As noted in \ref{ex:functor from a complex} and \ref{ex:cotangent complex}, the functors $\Alg_R\to\mls D(\Ab)$ given by $A\mapsto\rL\Omega^1_A$, $A\mapsto\F_{A*}\rL\Z\Omega^1_A$, and $A\mapsto\mls V^{}_R\otimes^{\rL}_RA$ are all left Kan extended from $\Poly_R$. The claim follows.
        \end{proof}

    \begin{pg}
        We will eventually want to work universally over the moduli stack of height 1 group schemes. To that end, we now consider an algebraic stack $\mls X$ over $\bF_p$ with affine diagonal. We write $\Alg_{\mls X}$ for the category whose objects are rings $A$ equipped with a map $\Spec A\to\mls X$, and $\Alg_{\mls X}^{\sm}\subset\Alg_{\mls X}$ (resp. $\Alg_{\mls X}^{\sm/\bF_p}\subset\Alg_{\mls X}$) for the full subcategory spanned by objects which are smooth over $\mls X$ (resp. smooth over $\bF_p$).


    \begin{notation}\label{not:notation for C_G, the functor}
        We globalize \ref{not:notation for C_G} as follows. Let $\mls V^{}$ be a perfect complex of locally free sheaves on $\mls X$ and let $\rho \colon\rL\F_{\mls X}^*\mls V^{}\to\mls V^{}$ be a morphism in $\mls D(\mls X_{\liset},\mls O_{\liset})$. Given a ring $A$ and a morphism $a\colon\Spec A\to\mls X$, we define the pullback $(\mls V^{}_A,\rho_A)$ as in \ref{pg:derived p linear maps}, and denote the resulting functor by
        \[
            \mls C_{(\mls V^{},\rho)}\colon\Alg_{\mls X}\to\mls D(\Ab),\hspace{1cm}A\mapsto\mls C(\mls V^{}_A,\rho_A).
        \]
    \end{notation}    

    \end{pg}
    \begin{pg}\label{pg:description of C}
        The functor $\mls C_{(\mls V^{},\rho)}$ may be described globally as follows. Write $\uLOmega^1_{\mls X}$ (resp. $\uLZOmega^1_{\mls X}$, resp. $\F_*\uLZOmega^1_{\mls X}$) for the functor $\Alg_{\mls X}\to\mls D(\Ab)$ defined by $A\mapsto\rL\Omega^1_A$ (resp. $A\mapsto\rL\Z\Omega^1_A$, resp. $A\mapsto\F_{A*}\rL\Z\Omega^1_A$). It follows from \cite[Theorem 3.1]{BMS} and \cite[2.2]{devalapurkar2023ptypical} that these functors are all fppf sheaves. We have
        \begin{equation}\label{eq:description of C as a functor}
            \mls C_{(\mls V^{},\rho)}\simeq\cocone\left(\rH_{\mls V^{}}\otimes^{\rL}\F_*\uLZOmega^1_{\mls X}\xrightarrow{\C-\widetilde{\rho}}\rH_{\mls V^{}}\otimes^{\rL}\uLOmega^1_{\mls X}\right)
        \end{equation}
    where the cocone is taken in $\Shv(\mls X)$. In particular, $\mls C_{(\mls V^{},\rho)}$ is an fppf sheaf.
    \end{pg}

        Before establishing various properties of $\mls C_{(\mls V^{},\rho)}$ it will be useful to note the following lemma (which will also be used later in the article).

\begin{lem}\label{L:3.15c} Consider categories $\mc A$ and $\mc B$ with subcategories $\mc A^\circ \subset \mc A$ and $\mc B^\circ \subset \mc B$ and a functor $r\colon \mc A\rightarrow \mc B$ sending $\mc A^\circ $ to $\mc B^\circ $. Assume that these four categories admit finite coproducts, and let $\mc U$ be an $\infty $-category which admits small sifted colimits.  Let $\mls F\colon \mc B\rightarrow \mc U$ be a functor, and let $\widehat {\mls F}$ be the left Kan extension of $\mls F|_{\mls B^\circ }$ so that there is a canonical map $\varepsilon \colon \widehat {\mls F}\rightarrow \mls F$.  Assume that the following conditions hold:
\begin{enumerate}
    \item [(a)] The functor $r$ has a left adjoint $u\colon \mc B\rightarrow \mc A$ sending $\mc B^\circ $ to $\mc A^\circ $.
    \item [(b)] The functor $\mls F\circ r\colon\mc A\rightarrow \mc U$ is the left Kan extension of its restriction to $\mc A^\circ $.
\end{enumerate}
Then $\varepsilon |_{\mc A}\colon \widehat {\mls F}\circ r\rightarrow \mls F\circ r$ is an isomorphism.
\end{lem}
\begin{proof}
Let $A\in \mc A$ be an object and let $\mc B^\circ |_{r(A)}$ and $\mc A^\circ |_A$ be the associated over-categories.  Then $r$ induces a functor, which we denote by the same letter
\begin{equation}\label{E:functorcategorymap}
r\colon\mc A^\circ |_{A}\rightarrow \mc B^\circ |_{r(A)},
\end{equation}
and the map $(\widehat {\mls F}\circ r)(A)\rightarrow (\mls F\circ r)(A)$ is identified with the natural map
\[
    \hocolim _{(B\rightarrow r(A))\in \mc B^\circ |_{r(A)}}\mls F(B)\rightarrow \mls F(r(A)).
\]
Now since $\mls F\circ r$ is Kan extended from $\mls A^\circ $ the map
\[
    \hocolim _{(A'\rightarrow A)\in \mls A^\circ |_A}\mls F(r(A'))\rightarrow \mls F(r(A))
\]
is an isomorphism, so it suffices to show that the map
\begin{equation}\label{E:colimitmap}
    \hocolim _{(A'\rightarrow A)\in \mc A^\circ |_{A}}\mls F(r(A'))\rightarrow \hocolim _{(B\rightarrow r(A))\in \mc B^\circ |_{r(A)}}\mls F(B)
\end{equation}
induced by \eqref{E:functorcategorymap} is an isomorphism.
Now observe that for any object $(\alpha \colon B\rightarrow r(A))\in \mls B^\circ |_{r(A)}$ the category of pairs $(\beta \colon A'\rightarrow A, h\colon B\rightarrow r(A'))$, consisting of an object $(A'\rightarrow A)\in \mls A^\circ |_A$ and a map $h$ making the diagram
\[
\begin{tikzcd}
    B\arrow{rr}{h}\arrow{dr}&&r(A')\arrow{dl}\\
    &A&
\end{tikzcd}
\]
commute, has an initial object given by the  pair $(\alpha ^t\colon(u(B)\rightarrow A, \mathrm{adj}\colon B\rightarrow ru(B))$, where $\alpha ^t$ denotes the adjoint of $\alpha $ and $\mathrm{adj}$ denotes the adjunction map.  It follows from this and general properties of homotopy colimits (see for example \cite[4.1.1.8]{LurieHTT}) that the map \eqref{E:colimitmap} is an isomorphism.
\end{proof}

    \begin{lem}\label{lem:kan ext lemma}
         Let $\mls F\colon \Alg_{\mls X}\to\mls D(\Ab)$ be an fppf sheaf with the property that for every $A\in\Alg_{\mls X}$, the restriction of $\mls F$ to $\Alg_A$ is left Kan extended from $\Alg_A^{\sm}$. Then $\mls F$ is the fppf sheafification of the left Kan extension of its restriction to the subcategory $\Alg_{\mls X}^{\sm}\subset\Alg_{\mls X}$.
    \end{lem}
    \begin{proof}
           Let $\widehat{\mls F}$ be the left Kan extension along $\Alg_{\mls X}^{\sm}\hookrightarrow\Alg_{\mls X}$ of the restriction of $\mls F$. There is a canonical map $\widehat{\mls F}\to\mls F$. As $\mls F$ is an fppf sheaf, to prove the result it will suffice to show that this map restricts to an isomorphism on some smooth cover of $\mls X$. Consider a smooth morphism $\Spec R\to\mls X$ from an affine scheme. We claim that if $A$ is an $R$-algebra then the map
            $\widehat{\mls F}(A)\to\mls F(A)$
        is an isomorphism.   If $\mls X$ admits a smooth cover by an affine scheme, then this will suffice to prove the lemma. In general, we consider a smooth cover $U\to\mls X$ where $U=\sqcup_i\Spec R_i$ is a disjoint union of affine schemes, and argue that is suffices to consider the case when $\Spec A$ is connected, so that any map $\Spec A\to U$ factors through some $\Spec R_i$.
        
To prove the claim  we apply \ref{L:3.15c} with $\mc B = \Alg _{\mls X }$, $\mc B^\circ = \Alg _{\mls X }^{\mathrm{sm}}$, $\mc A = \Alg _R$, $\mc A^\circ = \Alg _R^{\mathrm{sm}}$, and the 
 functor $r\colon\Alg_R\to\Alg_{\mls X}$ given by composition. To see that $r$ has a left adjoint consider a morphism $\Spec B\to\mls X$ and  the cartesian diagram
        \[
            \begin{tikzcd}
                \Spec B\times_{\mls X}\Spec R\arrow{d}\arrow{r}&\Spec R\arrow{d}\\
                \Spec B\arrow{r}&\mls X.
            \end{tikzcd}
        \]
        As $\mls X$ has affine diagonal, the fiber product $\Spec A\times_{\mls X}\Spec R$ is affine, say equal to $\Spec C$, and the association $B\mapsto C$ defines the desired left adjoint $u\colon\Alg_{\mls X}\to\Alg_R$. As $\Spec R\to\mls X$ is smooth, these maps restrict to a pair of adjoint functors $(u^{\sm }, r^{\sm })$ between the subcategories $\Alg_R^{\sm}$ and $\Alg_{\mls X}^{\sm}$. This implies the lemma.
    \end{proof}

    \begin{prop}\label{prop:its a Kan extension}
        For any pair $(\mls V^{},\rho)$ over $\mls X$ as in \ref{not:notation for C_G, the functor}, the functor $\mls C_{(\mls V^{},\rho)}\colon\Alg_{\mls X}\to\mls D(\Ab)$ is the fppf sheafification of the left Kan extension of its restriction to the subcategory $\Alg_{\mls X}^{\sm}\subset\Alg_{\mls X}$.
    \end{prop}
    \begin{proof}
        We check the conditions of \ref{lem:kan ext lemma}. As noted in \ref{pg:description of C} the functor $\mls C_{(\mls V^{},\rho)}$ is an fppf sheaf. Let $R$ be a ring and let $r\colon\Spec R\to\mls X$ be a morphism. The restriction of $\mls C_{(\mls V^{},\rho)}$ to $\Alg_R$ is the functor $\mls C_{(\mls V^{}_R,\rho_R)}$, where $(\mls V^{}_R,\rho_R)$ is the pullback of $(\mls V^{},\rho)$ along $r$. By \ref{lem:kan extension for C_R} the functor $\mls C_{(\mls V^{}_R,\rho_R)}$ is left Kan extended from $\Poly_R$, and hence also left Kan extended from $\Alg_R^{\sm}$. The result follows from \ref{lem:kan ext lemma}.
    \end{proof}




\begin{pg}{\bf Flat cohomology over general $\bF_p$-algebras.}
We now consider cohomology over $\bF_p$-algebras which are not necessarily smooth over $\bF_p$. Let $\mls X$ be an algebraic stack over $\bF_p$ with affine diagonal and let $\bG_{\mls X}$ be a commutative finite locally free height 1 group scheme on $\mls X$.
\end{pg}
    
    \begin{notation}\label{not:notation for C_G, just a group}
        We write $\mls C_{\bG_{\mls X}}$ for the functor $\mls C_{(\mls L,\rho)}\colon\Alg_{\mls X}\to\mls D(\Ab)$, where $\mls L$ is the Lie algebra of $\bG_{\mls X}$ and $\rho \colon\F_{\mls X}^*\mls L\to\mls L$ is the $\mls O$-linear $p$th power map (see \ref{P:2.34}).
    \end{notation}

    We also consider the functor
    \[
        \rH_{\bG_{\mls X}}\colon\Alg_{\mls X}\to\mls D(\Ab),\hspace{1cm}A\mapsto\R\Gamma(\Spec A,\bG_A)
    \]
    obtained by viewing $\bG_{\mls X}$ as a complex concentrated in degree 0 and applying the embedding~\eqref{eq:der cat map to shv}. We will define a certain canonical morphism
    \begin{equation}\label{eq:kan extension map, over a stack}
        \mls C_{\bG_{\mls X}}\to\rH_{\bG_{\mls X}}[1]
    \end{equation}
    of functors $\Alg_{\mls X}\to\mls D(\Ab)$. We will then show in \ref{thm:main kan theorem, over a stack} that this map is an isomorphism.

\begin{pg}[Definition of the morphism~\eqref{eq:kan extension map, over a stack}]
\label{pg:definition of the map via kan extension}
    We first consider the universal case of the moduli stack $\mls B$ of commutative height 1 group schemes over $\bF_p$ and the universal height 1 group scheme $\bG_{\mls B}$ over $\mls B$. As discussed in \ref{pg:moduli stack of height 1 group schemes}, $\mls B$ is smooth over $\bF_p$ and has affine diagonal.
    
    We identify $\Alg_{\mls B}$ with the category whose objects are pairs $(A,\bG_A)$, where $A$ is a ring of characteristic $p$ and $\bG_A$ is a commutative finite locally free height 1 group scheme over $A$, and whose morphisms $(A,\bG_A)\to(B,\bG_B)$ consist of pairs $(\varphi,\psi)$, where $\varphi\colon A\to B$ is a ring homomorphism and $\psi\colon \bG_B\to\bG_A$ is a morphism of schemes such that the diagram
    \[
        \begin{tikzcd}
            \bG_B\arrow{d}\arrow{r}{\psi}&\bG_A\arrow{d}\\
            \Spec B\arrow{r}&\Spec A
        \end{tikzcd}
    \]
    commutes and is cartesian and the induced isomorphism $\bG_B\simeq\bG_A\otimes_AB$ is an isomorphism of group schemes over $B$. The subcategory $\Alg_{\mls B}^{\sm/\bF_p}\subset\Alg_{\mls B}$ is the full subcategory spanned by those pairs $(A,\bG_A)$ such that $A$ is smooth over $\bF_p$.

    Let $P$ be a smooth $\bF_p$-algebra and let $\bG_P$ be a commutative finite locally free height 1 group scheme over $P$. The complex $\mls C_{\bG_P}(P)$ is then identified with the cocone of the right  morphism in the distinguished triangle~\eqref{eq:triangle in smooth case}, and the roof
    \[
        \begin{tikzcd}
            &\R\Gamma(\Spec P,[\bG_P\to\F_{P*}\bG^{(p)}_P])[1]\arrow{dl}{\mathrm{qis}}[swap]{\dlog}\arrow{dr}{\pi}[swap]{\mathrm{qis}}&\\
            \mls C_{\bG_P}(P)&&\R\Gamma(\Spec P,\bG_P)[1]
        \end{tikzcd}
    \]
    determines an isomorphism
    \begin{equation}\label{eq:kan extension map, smooth case}
        \mls C_{\bG_P}(P)\iso\R\Gamma(\Spec P,\bG_P)[1]
    \end{equation}
    in $\mls D(\Ab)$. Varying $(P,\bG_P)$, these maps give an isomorphism
    \begin{equation}\label{eq:iso on restrictions to smooth}
        \mls C_{\bG_{\mls B}}\big\lvert_{\Alg_{\mls B}^{\sm/\bF_p}}\iso \rH_{\bG_{\mls B}}[1]\big\lvert_{\Alg_{\mls B}^{\sm/\bF_p}}
    \end{equation}
    of functors $\Alg_{\mls B}^{\sm/\bF_p}\to\mls D(\Ab)$. By \ref{prop:its a Kan extension} the source $\mls C_{\bG_{\mls B}}$ is the fppf sheafification of the left Kan extension of its restriction to the subcategory $\Alg_{\mls B}^{\sm}\subset\Alg_{\mls B}$. As $\mls B$ is smooth over $\bF_p$, we have inclusions
    \[
        \Alg_{\mls B}^{\sm}\subset\Alg_{\mls B}^{\sm/\bF_p}\subset\Alg_{\mls B},
    \]
    and thus $\mls C_{\bG_{\mls B}}$ is also the sheafification of the left Kan extension of its restriction to the subcategory $\Alg_{\mls B}^{\sm/\bF_p}$. As the functor $\rH_{\bG_{\mls B}}$ is an fppf sheaf, the map~\eqref{eq:iso on restrictions to smooth} extends canonically to a morphism
    \begin{equation}\label{eq:kan extension map}
        \mls C_{\bG_{\mls B}}\to\rH_{\bG_{\mls B}}[1]
    \end{equation}
    of functors $\Alg_{\mls B}\to\mls D(\Ab)$. This defines the morphism~\eqref{eq:kan extension map, over a stack} in the universal case. We define this morphism in the general case by pullback from the universal case.
\end{pg}

    \begin{thm}\label{thm:main kan theorem, over a stack}
        For any algebraic stack $\mls X$ over $\bF_p$ with affine diagonal and any commutative finite locally free height 1 group scheme $\bG_{\mls X}$ over $\mls X$, the morphism~\eqref{eq:kan extension map, over a stack} is an isomorphism.
    \end{thm}
    The proof of \ref{thm:main kan theorem, over a stack} is involved, and will be postponed to \S\ref{sec:proof of kan ext result}. We remark that the result is significantly easier in the special case when $\bG=\balpha_p$, or more generally when $\bG$ is the Frobenius kernel of a vector group (see the proof of \ref{thm:kan extension theorem for a complex}). In the remainder of this section we derive some consequences of Theorem \ref{thm:main kan theorem, over a stack}.

\begin{pg}{\bf Proofs of \ref{thm:flat cohomology is Kan extended from polynomial algebras} and \ref{thm:derived hoobler}.}
    We prove Theorems \ref{thm:flat cohomology is Kan extended from polynomial algebras} and \ref{thm:derived hoobler}, granting Theorem \ref{thm:main kan theorem, over a stack}. Applying the latter to the universal pair $(\mls B,\bG_{\mls B})$, we obtain an isomorphism $\mls C_{\bG_{\mls B}}\simeq\rH_{\bG_{\mls B}}[1]$ of functors $\Alg_{\mls B}\to\mls D(\Ab)$. Recalling the description of the category $\Alg_{\mls B}$ given in \ref{pg:definition of the map via kan extension}, this means that for a pair $(R,\bG_R)$ we have an isomorphism
    \begin{equation}\label{eq:the flatest cohomology}
        \mls C_{\bG_R}(R)\simeq\R\Gamma(\Spec R,\bG_R)[1]
    \end{equation}
    that is functorial with respect to both $R$ and $\bG_R$. By its definition as a cocone, the left  term $\mls C_{\bG_R}(R)$ sits in a distinguished triangle
    \[
        \mls C_{\bG_R}(R)\to\mls L\otimes^{\rL}_R\F_{R*}\rL\Z\Omega^1_R\xrightarrow{\C-\widetilde \rho}\mls L\otimes^{\rL}_R\rL\Omega^1_R\xrightarrow{+1}
    \]
    in $\mls D(\Ab)$, and we obtain the desired triangle by combining this with the isomorphism~\eqref{eq:the flatest cohomology}. As this isomorphism was defined by Kan extension from the subcategory $\Alg_{\mls B}^{\sm}$, the resulting triangle agrees with~\eqref{eq:triangle in smooth case} if $R$ is smooth over $\bF_p$. This completes the proof of \ref{thm:derived hoobler}. Theorem \ref{thm:flat cohomology is Kan extended from polynomial algebras} follows from this combined with \ref{lem:kan extension for C_R}.
\end{pg}


\begin{pg}{\bf Proof of \ref{thm:proj bundle formula}.}
As an application of Theorem \ref{thm:main kan theorem, over a stack} we prove the projective bundle formula for flat cohomology stated in Theorem \ref{thm:proj bundle formula}. This is not used elsewhere in the article so if the reader desires they can skip this section.

We adopt the notation of the statement, so $X$ is an algebraic space over $\bF_p$ and $\mls E$ is a locally free sheaf on $X$ with associated projective bundle $\pi \colon Y := \mathbf{P}(\mls E)\rightarrow X$. We will deduce \ref{thm:proj bundle formula} as a consequence of the following result. Let $(\mls V^{}_X, \rho_X)$ be a pair consisting of a perfect complex $\mls V^{}_X$ on $X$ and a map $\rho_X\colon\rL\F_X^*\mls V^{}_X\to\mls V^{}_X$ in $\mls D(X_{\et},\mls O_{\et})$. Let $(\mls V^{}_Y,\rho_Y)$ be the pullback of this pair to $Y$ (in the sense of \ref{pg:derived p linear maps}). Let $\widetilde{\rho}_X:=\tau\circ\rho_X\colon\mls V^{}_X\to\mls V^{}_X$ (resp. $\widetilde{\rho}_Y$) be the derived $p$-linear map in $\mls D(X_{\et})$ associated to $\rho_X$ (resp. $\rho_Y$). The following relates the functors $\mls C_{(\mls V^{}_X, \rho_X)}\colon\Alg_X\to\mls D(\Ab)$ and $\mls C_{(\mls V^{}_Y,\rho_Y)}\colon\Alg_Y\to\mls D(\Ab)$.
\end{pg}

\begin{thm}\label{T:4.23}  The first Chern class of the universal line bundle on $Y$ induces an isomorphism
\[
    \R\Gamma(Y,\mls C_{(\mls V^{}_Y,\rho_Y)})\simeq\R\Gamma(X,\mls C_{(\mls V^{}_X,\rho_X)})\oplus\R\Gamma(X,[\mls V^{}_X\xrightarrow{1-\widetilde{\rho}_X}\mls V^{}_X])[-1],
\]
where the complex $[\mls V^{}_X\xrightarrow{1-\widetilde{\rho}_X}\mls V^{}_X]$ has terms in degrees 0 and 1.
\end{thm}

In particular, combining \ref{T:4.23} with \ref{thm:main kan theorem, over a stack}, and noting that the map $\boldsymbol{1}-\widetilde{\brho}$ is surjective, we obtain \ref{thm:proj bundle formula} as an immediate consequence. Before proving \ref{T:4.23} we prove the following decomposition for sheaves of differential forms on $X$ and $Y$.

\begin{lem}\label{lem:decompositions of differentials on projective bundle}
Suppose that $X$ is smooth over $\bF_p$. The first Chern class of the universal invertible sheaf $\mls O_Y(1)$ on $Y$ induces isomorphisms
\begin{equation}\label{E:cherniso1}
    \R\pi_*\Omega^1_Y\simeq\Omega^1_X\oplus\mls O_X[-1]
\end{equation}
and
\begin{equation}\label{E:cherniso2}
    \R\pi_*\left(\F_{Y*}\Z\Omega^1_{Y}\right)\simeq\F_{X*}\Z\Omega^1_{X}\oplus \mls O_X[-1].
\end{equation}
\end{lem}
\begin{proof}
Applying $\R\pi_*$ to the canonical short exact sequence of K\"ahler differentials for $\pi$ and using the isomorphism $\mls O_X\simeq\R\pi _*\mls O_Y$ and the projection formula, we find a distinguished triangle
\begin{equation}\label{eq:triangle for kahler differentials}
    \begin{tikzcd}
    \Omega^1_X\arrow{r}&\R\pi_*\Omega^1_Y\arrow{r}&\R\pi_*\Omega^1_{Y/X}\arrow{r}{+1}&\,.
    \end{tikzcd}
\end{equation}
We recall that the first Chern class map
\[
    c_1\colon\Pic(Y)\to\rH^1(Y,\Omega^1_Y)
\]
is the map induced by the map
\[
    \dlog \colon \mls O_Y^{\times}\rightarrow\Omega^1_{Y}, \ \ u\mapsto \mathrm{d}u/u
\]
of small \'{e}tale sheaves. Consider the map $\mls O_Y[-1]\to\Omega^1_{Y}$ in the derived category corresponding to the first Chern class $c_1(\mls O_Y(1))\in\rH^1(Y,\Omega^1_Y)$. Applying $\R\pi_*$ gives a map $c\colon\mls O_X[-1]\to\R\pi_*\Omega^1_Y$. Arguing by restricting to the fibers of $\pi$, we see that this map induces an isomorphism $\mls O_X[-1]\iso\R\pi_*\Omega^1_{Y/X}$. Thus, $c$ splits the above distinguished triangle~\eqref{eq:triangle for kahler differentials}, and induces the desired direct sum decomposition~\eqref{E:cherniso1}.

Next, we note that the map $\dlog$ factors through the subsheaf $\Z\Omega^1_Y\subset\Omega^1_Y$ of closed forms, and we have a commutative diagram
\[
    \begin{tikzcd}
        \F_{Y*}\mls O_Y^{\times}\arrow{r}{\sim}\arrow{d}[swap]{\F_{Y*}\dlog}&\mls O_Y^{\times}\arrow{d}{\dlog}\\
        \F_{Y*}\Z\Omega^1_Y\arrow{r}{\C_Y}&\Omega^1_Y
    \end{tikzcd}
\]
of small \'{e}tale sheaves. Here, the top horizontal arrow is the canonical isomorphism, given on sections over an \'{e}tale $Y$-scheme $U\to Y$ by the pullback map $\F_{U/Y}^*\colon (\F_{Y*}\mls O_Y^{\times})(U)=\mls O_Y^{\times}(U')\iso\mls O_Y^{\times}(U)$. Thus, the first Chern class map factors canonically through $\C_Y$. The invertible sheaf $\mls O_Y(1)$ yields a map $\mls O_Y[-1]\to\F_{Y*}\Z\Omega^1_Y$, and applying $\R\pi_*$ we obtain a map $\tilde{c}\colon\mls O_X[-1]\to\R\pi_*\left(\F_{Y*}\Z\Omega^1_Y\right)$ rendering the diagram
\[
    \begin{tikzcd}
        \mls O_X[-1]\arrow{r}[swap]{\tilde{c}}\arrow[bend left=20]{rr}{c}&\R\pi_*\left(\F_{Y*}\Z\Omega^1_Y\right)\arrow{r}[swap]{\C_Y}&\R\pi_*\Omega^1_Y
    \end{tikzcd}
\]
commutative. Now we apply $\R\pi_*$ to the exact sequence
\[
    \begin{tikzcd}
        0\arrow{r}&\mls O_Y\arrow{r}&\F_{Y*}\mls O_Y\arrow{r}{\mathrm{d}}&\F_{Y*}\Z\Omega^1_Y\arrow{r}{\C_Y}&\Omega^1_Y\arrow{r}&0
    \end{tikzcd}
\]
to obtain a commutative diagram
\[
    \begin{tikzcd}
        \mls O_X\arrow[d,"\sim"' sloped]\arrow{r}&\F_{X*}\mls O_X\arrow[d,"\sim"' sloped]\arrow{r}{\mathrm{d}}&\F_{X*}\Z\Omega^1_X\arrow{d}{\beta}\arrow{r}{\C_X}&\Omega^1_X\arrow{d}{\alpha}\\
        \R\pi_*\mls O_Y\arrow{r}&\R\pi_*\left(\F_{Y*}\mls O_Y\right)\arrow{r}{\mathrm{d}}&\R\pi_*\left(\F_{Y*}\Z\Omega^1_Y\right)\arrow{r}{\C_Y}&\R\pi_*\Omega^1_Y.
    \end{tikzcd}
\]
We have already observed that the map $c$ gives rise to a splitting of $\alpha$. It follows that $\tilde{c}$ induces a splitting of $\beta$, giving the desired decomposition~\eqref{E:cherniso2}.
\end{proof}

\begin{proof}[Proof of \ref{T:4.23}]
Suppose first that $X$ is smooth over $\bF_p$. The restriction of the functor $\R\pi _*\mls C_{(\mls V^{}_Y,\rho_Y)}$ to the (affine) small \'{e}tale site of $X$ is the cocone of the morphism
\begin{equation}\label{eq:pi_* morphism}
    \C-\widetilde{\rho}_Y\colon \R\pi _*(\pi^*\mls V^{}_X\otimes\F_{Y*}\Z\Omega^1_Y)\rightarrow\R\pi _*(\pi^*\mls V^{}_X\otimes \Omega ^1_Y).
\end{equation}
One checks that, under the isomorphisms
\[
    \R\pi _*(\pi ^*\mls V^{}_X\otimes \F_{Y*}\Z\Omega^1_Y)\simeq \mls V^{}_X\otimes\R\pi _*\left(\F_{Y*}\Z\Omega^1_Y\right)\simeq \left(\mls V^{}_X\otimes\F_{X*}\Z\Omega^1_X\right)\oplus \mls V^{}_X[-1], 
\]
and
\[
    \R\pi _*(\pi ^*\mls V^{}_X\otimes \Omega ^1_Y)\simeq \left(\mls V^{}_X\otimes \Omega ^1_X\right)\oplus \mls V^{}_X[-1]
\]
obtained from the decompositions of Lemma \ref{lem:decompositions of differentials on projective bundle}, the map~\eqref{eq:pi_* morphism} is identified with the direct sum of the maps
\[
    \C-\widetilde{\rho}_X\colon\mls V^{}_X\otimes\F_{X*}\Z\Omega^1_X\rightarrow \mls V^{}_X\otimes \Omega ^1_X
\]
and
\[
    1-\widetilde{\rho}_X\colon\mls V^{}_X[-1]\rightarrow \mls V^{}_X[-1].
\]
This implies \ref{T:4.23} in the case when $X$ is smooth over $\bF_p$.

The general case of \ref{T:4.23} follows from the smooth case as follows. Consider the stack $\mathrm{B}\mathbf{GL}_r$ over $\bF_p$ classifying vector bundles of rank $r$ and the universal projective bundle $[\mathbf{P}^{r-1}/\mathbf{GL}_r]\rightarrow \mathrm{B}\mathbf{GL}_r$. As $\mathrm{B}\mathbf{GL}_r$ is smooth over $\bF_p$, we have inclusions
\[
    \Alg_{\mathrm{B}\mathbf{GL}_r}^{\sm}\subset\Alg_{\mathrm{B}\mathbf{GL}_r}^{\sm/\bF_p}\subset\Alg_{\mathrm{B}\mathbf{GL}_r}.
\]
We have have already proved the result when $X$ is an affine scheme which is smooth over $\bF_p$, and in particular when $X$ is an affine scheme smooth over $\mathrm{B}\mathbf{GL}_r$. By Kan extension this implies the result for all affine $X$. For a general $X$ the result may be checked locally, so this implies the result.
\end{proof}

\begin{pg}\label{ssec:complexes of height 1 group schemes}{\bf Generalization to complexes of height 1 group schemes.}
    Let $X$ be an algebraic space over $\bF_p$. It is natural to seek to generalize Theorem \ref{thm:main kan theorem, over a stack} by replacing $\bG$ with a \emph{perfect complex of height 1 group schemes}, that is, an object $\bG^{}\in\mls D(X)$ which is fppf locally on $X$ quasi-isomorphic to a bounded complex of commutative finite locally free height 1 group schemes, and replacing the Lie algebra and its $p$th power map with a suitably chosen pair $(\mls V,\rho)$, where $\mls V$ is a perfect complex on $X$ and $\rho\colon\rL\F_X^*\mls V\to\mls V$ is a map in $\mls D(X_{\et},\mls O_{\et})$. In this section we will make some steps in this direction.

    \begin{lem}\label{lem:C is a hypersheaf}
        For any pair $(\mls V,\rho)$ as above, the functor $\mls C_{(\mls V^{},\rho)}\colon\Alg_{X}\to\mls D(\Ab)$ is an fppf hypersheaf.
    \end{lem}
    \begin{proof}
        By \cite[08C9]{stacks-project} we may choose an fppf cover $U\to X$ over which $\mls V^{}$ is quasi-isomorphic to a strictly perfect complex and $\rho$ is induced by a map $\F_X^*\mls V\to\mls V$ of complexes (note that as $\mls V$ is strictly perfect the Frobenius pullback does not need to be derived). By \cite[6.5.2.22]{LurieHTT}, the property of being an fppf hypersheaf may be checked fppf locally, so we may assume that $\mls V^{}$ is a strictly perfect complex and that $\rho$ is a strict morphism of complexes. Let $\bG^{}$ be the bounded complex of height 1 group schemes corresponding to $(\mls V^{},\rho)$ under the equivalence described in \ref{pg:height 1 group schemes and lie algebras}. By functoriality, the map~\eqref{eq:kan extension map, over a stack} extends to a map
        \[
            \mls C_{(\mls V^{},\rho)}\to\rH_{\bG^{}}[1]
        \]
        of functors $\Alg_{X}\to\mls D(\Ab)$, and Theorem \ref{thm:main kan theorem, over a stack} implies that this map is an isomorphism. As described in \ref{pg:hypercompleteness}, a sheaf $\mls F\in\Shv(X)$ is a hypersheaf if and only if $\mls F$ is in the essential image of the functor $\rH$, so we conclude the result.
    \end{proof}

    \begin{thm}\label{thm:main kan theorem for a perfect complex of height 1 group schemes}
        For any pair $(\mls V,\rho)$ on $X$, there exists a perfect complex of height 1 group schemes $\bG^{}\in\mls D(X)$ (unique up to isomorphism) such that $\mls C_{(\mls V^{},\rho)}\simeq\rH_{\bG^{}}[1]$.
    \end{thm}
        \begin{proof}
            By \ref{lem:C is a hypersheaf} the functor $\mls C_{(\mls V,\rho)}$ is an fppf hypersheaf, so there exists a complex $\bA\in\mls D(X)$ such that $\mls C_{(\mls V,\rho)}\simeq\rH_{\bA}$. Thus, setting $\bG=\bA[-1]$, we have $\mls C_{(\mls V,\rho)}\simeq\rH_{\bG[1]}\simeq\rH_{\bG}[1]$. We claim that $\bG$ is a perfect complex of height 1 group schemes. This can be checked fppf locally on $X$, so we may assume that $\mls V$ is strictly perfect and $\rho$ is induced by a map of complexes. Reasoning as in the proof of \ref{lem:C is a hypersheaf}, the result follows from Theorem \ref{thm:main kan theorem, over a stack}.
        \end{proof}

\end{pg}

\begin{rem}
    Note that we have not described how to associate a pair $(\mls V,\rho)$ to a perfect complex of height 1 group schemes (Theorem \ref{thm:main kan theorem for a perfect complex of height 1 group schemes} describes an association in the reverse direction).
\end{rem}

\begin{pg}\label{pg:perfect complex of vector bundles}
    We now consider a special case in which we can describe explicitly the perfect complex of height 1 group schemes produced in \ref{thm:main kan theorem for a perfect complex of height 1 group schemes}.
    Let $\mls V^{}$ be a perfect complex on $X$ and let $\bV^{}=\R\bV(\mls V^{})\in\mls D(X_{\fppf},\mls O_{\fppf})$ be the associated perfect complex of vector bundles. We write $\bV^{{(p)}}=\rL\F_{X}^*\bV$ for the Frobenius twist of $\bV^{}$ and $\F_{\bV^{}/X}\colon\bV^{}\to\bV^{{(p)}}$ for the relative Frobenius of $\bV^{}$. We define a complex $\bG_{\mls V^{}}\in\mls D(X)$ by
        \[
            \bG_{\mls V^{}}=\cocone\left(\bV^{}\xrightarrow{\F_{\bV^{}/X}}\bV^{(p)}\right).
        \]
    It is immediate that $\bG_{\mls V^{}}$ is a perfect complex of height 1 (and coheight 1) group schemes.
    \begin{example}\label{ex:alpha_p case}
        When $\mls V^{}=\mls O_{X}$, we have $\bV^{}=\bG_a$, and so $\bG_{\mls V^{}}\simeq\balpha_p$.
    \end{example}
    \begin{thm}\label{thm:kan extension theorem for a complex}
        For a perfect complex $\mls V^{}$ over $X$, there is a canonical isomorphism  (to be defined in the proof below)
        \[
            \mls C_{(\mls V^{},0)}\simeq\rH_{\bG_{\mls V^{}}}[1]
        \]
        of functors $\Alg_{X}\to\mls D(\Ab)$.
    \end{thm}
    \begin{proof}
    Consider Hoobler's sequence for $\bG_a$ and a smooth $\bF_p$-algebra $P$. Taking global sections, this yields the exact sequence
    \[
        0\to P\to\F_{P*}P\xrightarrow{\F_{P*}\mathrm{d}}\F_{P*}\Z\Omega^1_P\xrightarrow{\C_P}\Omega^1_P\to 0.
    \]
    We note that the morphisms in this sequence are $P$-linear, so this is an exact sequence of $P$-modules. By Kan extension, this sequence gives for a general $\bF_p$-algebra $A$ an isomorphism
    \begin{equation}\label{eq:triangle for G_a}
        \mls C_{(A,0)}(A)=[\F_{A*}\rL\Z\Omega^1_A\xrightarrow{\C_A}\rL\Omega^1_A]\simeq\R\Gamma(\Spec A,[A\to\F_{A*}A])[1]
    \end{equation}
    in $\mls D(\Mod_A)$ (not merely in $\mls D(\Ab)$). Now suppose we have a map $a\colon\Spec A\to X$. We tensor the above with the pullback $\mls V^{}_A:=\rL a^*\mls V$ and use the projection formula isomorphism $\mls V^{}_A\otimes_A\F_{A*}A\simeq\rL\F_A^*\mls V_A=:\mls V^{(p)}_A$. This yields an isomorphism
    \[
        \mls C_{(\mls V^{},0)}(A)=\R\Gamma(\Spec A,[\mls V^{}_A\to\mls V^{(p)}_A])[1]\simeq\rH_{\bG_{\mls V^{}}}[1](A)
    \]
    Varying $A$ we obtain the desired isomorphism of functors.
    \end{proof}

    \begin{rem}
        Note that, unlike \ref{lem:C is a hypersheaf} and \ref{thm:main kan theorem for a perfect complex of height 1 group schemes}, the proof of \ref{thm:kan extension theorem for a complex} is independent of Theorem \ref{thm:main kan theorem, over a stack}.
    \end{rem}
\end{pg}

\section{Proof of Theorem \ref{thm:main kan theorem, over a stack}}
\label{sec:proof of kan ext result}

Our proof is a generalization of an argument due to Bhatt and Lurie (private communication), who considered the case of $\bmu_p$ 
(see also \cite{BL2} for a different argument in that case). The case of $\bmu _p$ had also previously been observed by Peter Scholze.

\begin{pg}
We first note that it will suffice to prove \ref{thm:main kan theorem, over a stack} in the universal case, that is, for the pair $(\mls B,\bG_{\mls B})$ where $\mls B$ is the moduli stack of commutative height 1 group schemes (see \ref{pg:moduli stack of height 1 group schemes}) and $\bG_{\mls B}$ is the universal height 1 group scheme over $\mls B$. Let $\Spec R\to\mls B$ be a smooth morphism from an affine scheme. We claim that the map
\begin{equation}\label{eq:kan ext map 3}
    \mls C_{\bG_R}\to\rH_{\bG_R}[1]
\end{equation}
of functors $\Alg_R\to\mls D(\Ab)$ is an isomorphism. By \ref{prop:its a Kan extension}, the functor $\mls C_{\bG_R}$ is an fppf sheaf, as is the functor $\rH_{\bG_R}$, so applying this to a smooth cover of each connected component of $\mls B$ by an affine scheme we will obtain the result. By \ref{prop:its a Kan extension} the functor $\mls C_{\bG_R}$ is left Kan extended from the subcategory $\Alg_R^{\sm}\subset\Alg_R$. As $\mls B$ is smooth over $\bF_p$, so is any smooth $R$-algebra, and therefore~\eqref{eq:kan ext map 3} restricts to an isomorphism on $\Alg_R^{\sm}$. Thus, to show that~\eqref{eq:kan ext map 3} is an isomorphism it will suffice to show that $\rH_{\bG_R}$ is left Kan extended from $\Alg_R^{\sm}$.
\end{pg}
    \begin{notation}\label{not:notation for C_R(A)}
        For an $R$-algebra $A$ we write $C_R(A)$ for the category whose objects are smooth $R$-algebras $P$ equipped  with an $R$-algebra homomorphism $P\to A$.
    \end{notation}

By definition, the statement that $\rH_{\bG_R}$ is left Kan extended from $\Alg_R^{\sm}$ is equivalent to the following proposition, whose proof occupies the remainder of the section.

    \begin{prop}\label{prop:reduction to universal case}
        For every $R$-algebra $A$ the canonical map
        \begin{equation}\label{eq:hocolim isomorphism}
            \hocolim_{P\in C_{R}(A)}\R\Gamma(\Spec P,\bG_P)\to\R\Gamma(\Spec A,\bG_A)
        \end{equation}
        is an isomorphism.
    \end{prop}

The proof of \ref{prop:reduction to universal case} is by a series of reductions to the case when $A$ is perfect, which we handle with a direct argument. The key to this reduction is the following proposition.

\begin{prop}\label{prop:infinitesimal thickening induction}
    Let $B\twoheadrightarrow A$ be a surjective homomorphism of $R$-algebras all elements of the kernel of which are nilpotent. If the map~\eqref{eq:hocolim isomorphism} is an isomorphism for $A$, it is also an isomorphism for $B$.
\end{prop}
\begin{proof}
 Consider the diagram
\[
    \begin{tikzcd}
        \displaystyle\hocolim_{P\in C_{R}(B)}\R\Gamma(\Spec P,\bG_P)\arrow{r}\arrow{d}&\R\Gamma(\Spec B,\bG_B)\arrow{d}\\
        \displaystyle\hocolim_{P\in C_{R}(A)}\R\Gamma(\Spec P,\bG_P)\arrow{r}{\sim}&\R\Gamma(\Spec A,\bG_A).
    \end{tikzcd}
\]
It follows from the triangle~\eqref{eq:triangle in smooth case} that the cohomology $\R\Gamma(\Spec P,\bG_P)$ is concentrated in degrees $[0,2]$. Thus, all of the terms of the above diagram, excepting possibly $\R\Gamma(\Spec B,\bG_B)$, are in $\mls D^{\leq 2}(\Ab)$. In fact, by \cite[5.2.10]{CS}, the reduction map $\rH^i(\Spec B,\bG_B)\to\rH^i(\Spec A,\bG_A)$ is surjective for $i\geq 1$ and bijective for $i\geq 2$. Thus, $\R\Gamma(\Spec B,\bG_B)$ is also in $\mls D^{\leq 2}(\Ab)$.

This enables us to reformulate \ref{prop:infinitesimal thickening induction} in terms of torsors and gerbes, which represent cohomology classes in degrees 1 and 2, respectively. For an $R$-algebra $C$ we let $\Gb(C)$ denote the $2$-groupoid of $\bG_C$-gerbes over $C$. The set of isomorphism classes of  $\Gb (C)$ is $\rH ^2(C, \bG _C)$, the set of isomorphism classes of automorphisms of the neutral gerbe $\mathrm{B}\bG_C\in\Gb (C)$ is $\rH ^1(C, \bG _C)$, and the group of automorphisms of the identity automorphism of $\mathrm{B}\bG_C$ is $\rH ^0(C, \bG _C)$.  Therefore the $2$-groupoid $\Gb (C)$ captures all the cohomological information about $\bG _C$ over $C$.

Consider the diagram
\begin{equation}\label{eq:square of cats of gerbes}
    \begin{tikzcd}
        \displaystyle\hocolim_{P\in C_{R}(B)}\Gb(P)\arrow{r}\arrow{d}&\Gb(B)\arrow{d}\\
        \displaystyle\hocolim_{P\in C_{R}(A)}\Gb(P)\arrow{r}{\sim}&\Gb(A).
    \end{tikzcd}
\end{equation}
It follows from our assumptions that the lower horizontal arrow is an isomorphism. To prove Proposition \ref{prop:infinitesimal thickening induction}, it will suffice to show that the upper horizontal arrow is also an isomorphism. 

\begin{lem}\label{L:11.9}
Let $Q$ be a smooth $R$-algebra equipped with an $R$-algebra homomorphism $Q\to B$. Let $\mls G_Q$ be a $\bG_Q$-gerbe over $Q$ and let $\alpha_B\colon\Spec(B)\rightarrow \mls G_Q$ be a morphism over $Q$. There exists a smooth $Q$-algebra $Q'$, $R$-algebra homomorphisms $Q\to Q'\to B$, and a smooth morphism $\alpha_{Q'}\colon\Spec Q'\to\mls G_Q$ such that the diagram
\begin{equation}\label{E:11.9.1}
\begin{tikzcd}
    &&\mls G_Q\arrow{d}\\
    \Spec B\arrow[bend left=20]{urr}{\alpha_B}\arrow{r}&\Spec Q'\arrow{ur}{\alpha_{Q'}}\arrow{r}&\Spec Q
\end{tikzcd}
\end{equation}
commutes.
\end{lem}
\begin{proof}
Write $i\colon \Spec A\to\Spec B$ for the map induced by the surjection $B\twoheadrightarrow A$ and set $\alpha_A=\alpha_B\circ i$. Then $\alpha_A$ induces a trivialization of the pullback of $\mls G_Q$ to $A$. The functor $\Alg_Q\to\Alg_R$ admits a left adjoint given by forming the fiber product over $R$. As $Q$ is smooth over $R$, this restricts to a left adjoint for the restriction $\Alg^{\sm}_Q\to\Alg^{\sm}_R$, and to a left adjoint for the induced map $C_R(Q)\to C_R(A)$. By \ref{L:3.15c} it follows that the canonical map
\[
     \hocolim_{P\in C_Q(A)}\Gb(P)\iso\hocolim_{P\in C_R(A)}\Gb(P)
\]
is an equivalence. By assumption, the right colimit maps isomorphically to $\Gb(A)$. We may therefore find a smooth $Q$-algebra $Q'$, an $R$-algebra homomorphism $Q'\to A$, and a morphism $\alpha_{Q'}\colon\Spec Q'\to\mls G_Q$ whose pullback to $A$ is isomorphic to $\alpha_A$. Thus, we have a $2$-commutative diagram 
\[
    \begin{tikzcd}
    &&\mls G_Q\arrow{d}\\
    \Spec A\arrow[bend left=25]{urr}{\alpha_A}\arrow{r}&\Spec Q'\arrow{ur}{\alpha_{Q'}}\arrow{r}&\Spec Q.
\end{tikzcd}
\]
Moreover, as $Q'$ is smooth over $R$ the map $Q'\to A$ factors through $B$. Replacing $Q$ by $Q'$, we are therefore reduced to the case when $\mls G_{Q}=\mathrm{B}\bG_Q$. The map $\alpha_B$ then corresponds to a $\bG_B$-torsor $T_{B}\to\Spec B$ over $B$. Let $T_A=T_B\otimes_BA$ be its restriction to $A$. Applying the above argument again, we may assume by replacing $Q$ with a smooth $Q$-algebra that $T_A$ lifts to a $\bG_Q$-torsor $T_Q\to\Spec Q$ over $Q$. By replacing $T_B$ with the difference of $T_B$ and the pullback $T_Q\otimes_QB$, we may assume that the reduction $T_A$ of $T_B$ is the trivial $\bG_A$-torsor. Now consider the B\'egueri resolution
\[
    0\to\bG_R\to\bB^0\to\bB^1\to 0
\]
for $\bG_R$ (see \ref{pg:begueri resolution}). The torsor $T_B\wedge^{\bG}\bB^0$ is a deformation of the trivial $\bB^0$-torsor over $A$, and $\bB^0$ is smooth, so $T_B\wedge^{\bG_B}\bB^0$ is trivial. From the long exact sequence
\[
    0\to\bG_R(B)\to\bB^0(B)\to\bB^1(B)\to\rH^1(\Spec B,\bG_B)\to\rH^1(\Spec B,\bB^0)\to\ldots
\]
we see that $T_B$ is isomorphic to the $\bG_B$-torsor of liftings of some section $s_B\in\bB^1_B(B)$ to $\bB^0_B$. Let $s\colon\Spec B\to\bB^1_Q$ denote the morphism induced by $s_B$. Let $\beta\colon\Spec Q\to\mathrm{B}\bG_Q$ be the canonical quotient map and let $\alpha\colon\bB^1_Q\to\mathrm{B}\bG_Q$ be the map corresponding to the $\bG_Q$-torsor $\bB^0_Q\rightarrow \bB^1_Q$. We claim that $\alpha$ is a smooth morphism and that the diagram
\[
    \begin{tikzcd}
        &&\mathrm{B}\bG_Q\arrow{d}\\
        \Spec B\arrow[bend left=25]{urr}{\alpha_B}\arrow{r}{s}&\bB^1_Q\arrow{ur}{\alpha}\arrow{r}&\Spec Q
    \end{tikzcd}
\]
is 2-commutative. Both these claims follows from a consideration of the diagram
\[
    \begin{tikzcd}
        T_B\arrow{d}\arrow{r}&\bB^0_Q\arrow{d}\arrow{r}&\Spec Q\arrow{d}{\beta}\\
        \Spec B\arrow{r}{s}&\bB^1_Q\arrow{r}{\alpha}&\mathrm{B}\bG_Q
    \end{tikzcd}
\]
in which both squares are 2-cartesian. As $\bB^1_Q$ is affine, this completes the proof.
\end{proof}

We now complete the proof of Proposition \ref{prop:infinitesimal thickening induction} by showing that the canonical map
\begin{equation}\label{eq:hocolim for B}
    \hocolim_{P\in C_R(B)}\Gb(P)\to\Gb(B)
\end{equation}
is an isomorphism. Let $\mls G_B$ be a $\bG_B$-gerbe over $B$ and let $\mls F(\mls G_B)$ be the $2$-category of tuples $(P\to B,\mls G_P,\psi)$, where $P$ is a smooth $R$-algebra, $P\to B$ is a map of $R$-algebras, $\mls G_P$ is a $\bG_P$ gerbe over $P$, and $\psi\colon\mls G_B\to\mls G_P\otimes_PB$ is an equivalence of $\bG_B$-gerbes over $B$.
To show that~\eqref{eq:hocolim for B} is an isomorphism it will suffice to show that for any $\bG_B$-gerbe $\mls G_B$ over $B$ the category $\mls F(\mls G_B)$ is contractible.
   
We first show that $\mls F(\mls G_B)$ is non-empty. Set $\mls G_A:=\mls G_B\otimes_BA$. As the lower horizontal arrow of the diagram~\eqref{eq:square of cats of gerbes} is an equivalence, we can find a smooth $R$-algebra $P$, an $R$-algebra map $P\to A$, and a $\bG_P$-gerbe $\mls G_P$ over $P$ whose restriction to $A$ is isomorphic to $\mls G_A$. As $P$ is smooth over $R$ the map $P\to A$ factors as a composition $P\to B\twoheadrightarrow A$. The pullback of $\mls G_P$ along $\Spec B\to\Spec P$ has restriction to $A$ isomorphic to $\mls G_A$. The reduction map $\rH^2(B,\bG_B)\to\rH^2(A,\bG_A)$ is injective \cite[5.2.10]{CS}, so the pullback of $\mls G_P$ along $\Spec B\to\Spec P$ is isomorphic to $\mls G_B$. Thus, $\mls F(\mls G_B)$ is nonempty.

Fix two objects $(P_i\rightarrow B, \mls G_{P_i}, \psi_i), \ (i=1,2)$ of $\mls F(\mls G_B)$. To show that $\mls F(\mls G_B)$ is contractible it will suffice to show that the category of cospans
\[
    \begin{tikzcd}
        (P_1\to B,\mls G_{P_1},\psi_1)\arrow{dr}&&(P_2\to B,\mls G_{P_2},\psi_2)\arrow{dl}\\
        &(P\to B,\mls G_P,\psi)&
    \end{tikzcd}
\]
in $\mls F(\mls G_B)$ is contractible. Equivalently, by setting $Q=P_1\otimes_{R}P_2$, letting $\mls G_Q\rightarrow \Spec Q$ denote the $\bG_Q$-gerbe given by the difference of the pullbacks of the $\mls G_{P_i}$, and replacing $\mls G_B$ and $\mls G_P$ with their differences with the respective pullbacks of $\mls G_{P_2}$, it will suffice to show that
the category $\mls F_0$ of 2-commutative diagrams (the solid arrows being fixed and the dashed arrows being allowed to vary)
\begin{equation}\label{eq:diagram for F}
\begin{tikzcd}
    &&\mls G_Q\arrow{d}\\
    \Spec B\arrow[bend left=20]{urr}{\alpha_B}\arrow[dashed]{r}&\Spec P\arrow[dashed]{ur}{\alpha_{P}}\arrow[dashed]{r}&\Spec Q
\end{tikzcd}
\end{equation}
over $R$, where $P$ is a smooth $R$-algebra, is contractible. By Lemma \ref{L:11.9} the category $\mls F_0$ is nonempty, and moreover we may find an object $\Spec B\to\Spec Q'\xrightarrow{\alpha_{Q'}}\mls G_Q$ of $\mls F_0$ with the additional property that $\alpha_{Q'}$ is a smooth morphism. Given an arbitrary object  $\Spec B\to\Spec P\xrightarrow{\alpha_P}\mls G_Q$ of $\mls F_0$, we may form the pullback diagram
\[
    \begin{tikzcd}
        &\Spec P\times_{\mls G_Q}\Spec Q'\arrow{r}\arrow{d}&\Spec Q'\arrow{d}{\alpha_{Q'}}\\
        \Spec B\arrow{ur}\arrow{r}&\Spec P\arrow{r}{\alpha_P}&\mls G_Q.
    \end{tikzcd}
\]
As the diagonal of $\mls G_Q$ is affine, the scheme $\Spec P\times_{\mls G_Q}\Spec Q'$ is also affine. Furthermore, the morphism $\alpha_{Q'}$ is smooth, so the projection $\Spec P\times_{\mls G_Q}\Spec Q'\to\Spec P$ is smooth. As $P$ is smooth over $R$, it follows that $\Spec P\times_{\mls G_Q}\Spec Q'$ is smooth over $R$. We therefore get a canonical path in $\mls F_0$ from any object to our fixed object, showing that $\mls F_0$ is contractible. This completes the proof of Proposition \ref{prop:infinitesimal thickening induction}.
\end{proof}

\begin{pg}{\bf Proof of \ref{prop:reduction to universal case} when $A$ is perfect.}\end{pg}

\begin{lem}\label{lem:frob pullback lemma}
    For a scheme $S$ of characteristic $p$ and a commutative group scheme $\bG$ over $S$ which is either finite locally free or smooth, the map
    \begin{equation}\label{eq:frob pullback on coho}
        \F_S^*\colon \R\Gamma(S,\bG)\to\R\Gamma(S,\bG')
    \end{equation}
    induced by pullback along the absolute Frobenius $\F_S\colon S\to S$ is canonically isomorphic to the map on derived global sections induced by the map $\F_{\bG/S}\colon\bG\to\bG'$ of group schemes over $S$ (here  $\bG ':= \bG \times _{S. \F_S}S$). In particular, if $\bG$ is finite locally free of height $1$, then~\eqref{eq:frob pullback on coho} is zero.
\end{lem}
\begin{proof}
    For a group scheme $\bG$ over $S$, we define a morphism $\tau\colon\F_{S*}\bG'\to\bG'$ of fppf sheaves by associating to an $S$-scheme $U\to S$ the map
    \[
        \tau(U)\colon(\F_{S*}\bG')(U)=\bG'(U')\xrightarrow{\bG'(\F_{U/S})}\bG'(U).
    \]
    One checks that the diagram
    \begin{equation}\label{eq:diagram for tau}
        \begin{tikzcd}
            \bG\arrow{r}\arrow[bend left=25]{rr}{\F_{\bG/S}}&\F_{S*}\bG'\arrow{r}[swap]{\tau}&\bG'
        \end{tikzcd}
    \end{equation}
    commutes, where the left map is the canonical adjunction map. Let $\epsilon\colon S_{\fppf}\to S_{\et}$ denote the projection. Write $\mls G=\epsilon_*\bG$ for the restriction of $\bG$ to $S_{\et}$. If $U\to S$ is \'{e}tale, then the relative Frobenius $\F_{U/S}$ is an isomorphism, and thus $\tau(U)\colon (\F_{S*}\bG')(U)\to\bG'(U)$ is an isomorphism. It follows that $\tau$ induces an isomorphism $\tau^{\et}\colon\F_{S*}^{\et}\mls G'\iso \mls G'$ of small \'{e}tale sheaves.

    We now prove the result. The map~\eqref{eq:frob pullback on coho}
    is identified under the isomorphism $\R\Gamma(S,\rule{.25cm}{0.4pt})=\R\Gamma(S,\R\F_{S*}(\rule{.25cm}{0.4pt}))$ with the map on derived global sections induced by the derived adjunction map
    \[
        \eta\colon\bG\to\R\F_{S*}\bG'.
    \]
    Suppose that $\bG$ is smooth. By a result of Grothendieck \cite[Th\'eor\`eme 11.7]{MR0244271}, we have $\R^i\F_{S*}\bG=0$ and $\R^i\epsilon_*\bG=0$ for $i>0$. It follows that $\R\F_{S*}\bG'=\F_{S*}\bG'$ and $\R\epsilon_*\bG=\mls G$. Applying first $\R\epsilon_*$ and then $\R\Gamma(S_{\et},\rule{.25cm}{0.4pt})$ to the diagram~\eqref{eq:diagram for tau} we obtain a commutative diagram
    \[
        \begin{tikzcd}
            \R\Gamma(S,\bG)\arrow{r}{\R\Gamma(\eta)}\arrow[bend left=25]{rr}{\F_{\bG/S}}&\R\Gamma(S,\R\F_{S*}\bG')\arrow{r}{\sim}&\R\Gamma(S,\bG').
        \end{tikzcd}
    \]
    This gives the result in the case when $\bG$ is smooth. If $\bG$ is finite locally free, the result follows from the B\'{e}gueri resolution and the smooth case.
\end{proof}
    
We now suppose that $A$ is a perfect $\bF_p$-algebra, and show that the map~\eqref{eq:hocolim isomorphism} is an equivalence. By Lemma \ref{lem:frob pullback lemma}, the map 
\[
    \F_A^*\colon \R\Gamma(\Spec A,\bG_A)\to\R\Gamma(\Spec A,\bG_A')
\]
given by pullback along the absolute Frobenius of $\Spec A$ is trivial. On the other hand, as $A$ is perfect, $\F_A$ is an isomorphism. Thus, the map $\F_A^*$ is also an isomorphism, and therefore $\R\Gamma(\Spec A,\bG_A)=0$. To prove the result it will therefore suffice to show that the homotopy colimit on the left  side of~\eqref{eq:hocolim isomorphism} is also trivial.

For an $R$-algebra $\varphi\colon R\to A$ we write $A_{(1)}$ for the ring $A$ regarded as an $R$-algebra by the composition $R\xrightarrow{\F_R}R\xrightarrow{\varphi} A$. We note that the absolute Frobenius of $A$ gives a map $\F_A\colon A\to A_{(1)}$ of $R$-algebras. We consider the functors
\[
    \alpha,\beta\colon C_R(A)\to C_R(A_{(1)})
\]
defined by
$\alpha(P\xrightarrow{\varphi}A)=(P\xrightarrow{\varphi}A\xrightarrow{\F_A}A_{(1)})$ and $\beta(P\xrightarrow{\varphi}A)=(\varphi_{(1)}:P_{(1)}\to A_{(1)})$. For an object $\varphi\colon P\to A$ of $C_R(A)$, we have a commutative diagram
    \[
        \begin{tikzcd}
            P\arrow{d}[swap]{\varphi}\arrow{r}{\F_P}&P_{(1)}\arrow{d}{\varphi_{(1)}}\\
            A\arrow{r}{\F_A}&A_{(1)}
        \end{tikzcd}
    \]
    of $R$-algebras. Thus, the absolute Frobenius $\F_P\colon P\to P_{(1)}$ gives a morphism $\alpha(P\xrightarrow{\varphi}A)\to\beta(P\xrightarrow{\varphi}A)$
    in the category $C_R(A_{(1)})$. These morphisms define a natural transformation $\theta\colon\alpha\to\beta$ of functors. As described in \cite[A.3]{MR1870516}, this gives rise to a diagram
    \[
        \begin{tikzcd}[column sep=tiny]
            \displaystyle\hocolim_{P\in C_R(A)}\R\Gamma(\Spec P,\bG_P)\arrow{rr}{\theta_*}\arrow{dr}{\phi_{\alpha}}&&
            \displaystyle\hocolim_{P\in C_R(A)}\R\Gamma(\Spec P_{(1)},\bG_P')\arrow{dl}[swap]{\phi_{\beta}}\\
            &\displaystyle\hocolim_{P\in C_R(A_{(1)})}\R\Gamma(\Spec P,\bG_P)&
        \end{tikzcd}
    \]
    of homotopy colimits which is commutative up to homotopy. It follows from our assumption that $A$ is perfect that $\alpha$ and $\beta$ are both equivalences of categories. Thus, the maps $\phi_{\alpha}$ and $\phi_{\beta}$ are equivalences, and so $\theta_*$ is also an equivalence. On the other hand, we have that $\theta_*$ is the map induced by the functors
    \[
        \theta_*(P)\colon\R\Gamma(\Spec P,\bG_P)\to\R\Gamma(\Spec P_{(1)},\bG_P')
    \]
    coming from the natural transformation $\rH_{\bG_A}\circ\alpha\to\rH_{\bG_A}\circ\beta$ induced by $\theta $.  Thus, $\theta_*(P)$ is the map given by pullback along the absolute Frobenius $\F_P\colon P\to P_{(1)}$, and so by Lemma \ref{lem:frob pullback lemma} is trivial. It follows that $\theta_*$ is trivial, and so all of these homotopy colimits are trivial.

This completes the proof of \ref{prop:reduction to universal case} in the case when $A$ is perfect. \qed

\begin{pg}{\bf Proof of \ref{prop:reduction to universal case}  in general.}
We now complete the proof of \ref{prop:reduction to universal case}, and hence \ref{thm:main kan theorem, over a stack}. Recall that a ring $R$ of characteristic $p$ is \emph{semiperfect} if the absolute Frobenius $\F_R\colon R\to R$ is surjective.
\end{pg}
\begin{lem}\label{lem:fpqc cover of a ring}
    Any ring $R$ of characteristic $p$ admits an fpqc cover by a semiperfect ring.
\end{lem}
\begin{proof}
    If $R$ is a polynomial ring over $\bF_p$ (possibly infinitely generated) then the perfection $R\to R^{\pf}$ is faithfully flat. For a general $R$, we may write $R$ as a quotient $P\twoheadrightarrow R$ of a polynomial ring over $\bF_p$. Then the tensor product $R\otimes_PP^{\pf}$ is semiperfect, and the map $R\to R\otimes_PP^{\pf}$ is faithfully flat.
\end{proof}

To complete the prof of \ref{prop:reduction to universal case} note that by \ref{prop:its a Kan extension} the functor $\mls C_{\bG_R}$ is an fpqc sheaf, as is the functor $\rH_{\bG_R}$ by \cite[5.5.2]{CS}. Also note that if $A$ is a semiperfect $R$-algebra then the self tensor product $A\otimes_RA$ is also semiperfect. The same is true for the $n$-fold self tensor product. It will thus suffice to prove \ref{prop:reduction to universal case} for a semiperfect $A$. For such an $A$, the map $A\to A^{\pf}$ is surjective and has kernel consisting of only nilpotents. By the perfect case we know the result for $A^{\pf}$, and applying \ref{prop:infinitesimal thickening induction} gives the result for $A$.  This completes the proof of \ref{prop:reduction to universal case}. \qed

\section{Algebraic complexes}\label{S:section10}


In this section we study various algebraicity properties of complexes of sheaves of abelian groups on the big fppf site of an algebraic space $S$. We identify a class of complexes $\bC\in\mls D(S)$ which we call \emph{stably algebraic}. Examples of stably algebraic complexes include perfect complexes of vector bundles and shifts of flat group algebraic spaces which are locally of finite presentation over $S$. Furthermore, the collection of stably algebraic complexes is closed under shifts and cones and forms a stable $\infty$-subcategory $\mls D^{\salg}(S)\subset\mls D(S)$, and membership in $\mls D^{\salg}(S)$ my be checked fppf locally on $S$. The main results from this section that we will use later relate properties of a stably algebraic complex $\bC\in\mls D(S)$ to the representability of the individual cohomology sheaves $\mls H^n(\bC)$ (see Corollaries \ref{cor:coho sheaves of a stably alg complex}, \ref{cor:generic representability of stably alg complex}, and \ref{cor:cohomology of a stably algebraic complex is constructible}).

We approach algebraicity from the perspective of algebraic $n$-stacks, as developed by Simpson \cite{Simpson}, To\"en-Vezzosi \cite{MR2394633}, Lurie \cite{lurie2009derived}, and others. For the convenience of the reader we do not assume familiarity with these works, and summarize the key aspects of the theory of higher algebraic stacks that we will need, following to large extent \cite{Simpson}.  The main technical difference between the approach here and that in loc. cit. is that we work with stacks admitting fppf covers rather than smooth covers. 
 In the case of algebraic spaces and stacks in the classical sense this leads to an equivalent notions due to a theorem of Artin \cite[\href{https://stacks.math.columbia.edu/tag/06DC}{Tag 06DC}]{stacks-project}.  This result has been generalized to arbitrary $n$-stacks by Toen \cite{Toen3}, though we will not use this in what follows. 


\begin{rem}
 The content of this section is also related to the work of Hartshorne \cite{MR1656482}, which can be viewed as studying the problem of algebraicity for individual cohomology sheaves associated to perfect complexes of coherent sheaves.
\end{rem}

\newcommand{\sShv}{\mathscr{S}\!\mathscr{h}\!\mathscr{v}_{\Spc}}
\newcommand{\Spc}{\text{\rm Spc}}
\newcommand{\DK }{\text{\rm DK}}
\newcommand{\st }{\text{\rm st}}
\newcommand{\bX}{\mathbf{X}}
\newcommand{\bY}{\mathbf{Y}}

\begin{pg}\label{SS:5.5}{\bf Algebraic $\infty$-stacks.} Let $S$ be an algebraic space. Let $\Spc$ denote the $\infty$-category of spaces as defined in \cite[1.2.16.1]{LurieHTT}. As in \ref{not:infinity sheaves} we have the $\infty$-category $\sShv(S)$ of sheaves on $S_{\text{\rm fppf}}$ taking values in $\Spc$ \cite[6.2.2.6]{LurieHTT}. For an integer $n\geq 0$ let $\sShv^{\leq n}(S)\subset \sShv (S)$ denote the $\infty $-category of $n$-truncated objects \cite[6.4.1]{LurieHTT}.  We refer to objects of $\sShv^{\leq n}(S)$ as \emph{$n$-stacks} (over $S$).
\end{pg}
\begin{pg}\label{pg:definition of an algebraic n stack}
Following the method of \cite{Simpson} we now define by induction on $n\geq 0$ the following notions:
\begin{enumerate}
    \item [(i)] What it means for an $n$-stack $\mls X\in \sShv^{\leq n}(S)$ to be \emph{algebraic}.
    \item [(ii)] What it means for a morphism $f\colon\mls X\to\mls Y$ of algebraic $n$-stacks to be an \emph{fppf cover}
\end{enumerate}
\end{pg}

\begin{pg}[Base case]
For $n=0$ the category $\sShv^{\leq 0}(S)$ is equivalent to the category of sheaves of sets, and in this case we define an object $\mls X \in \sShv^{\leq 0}(S)$ to be \emph{algebraic} if it is an algebraic space in the usual sense.
We define (ii) as for algebraic spaces.
\end{pg}

\begin{pg}[Induction step]
Assume that $n>0$ and that the notions (i) and (ii) have been defined for $0\leq m<n$. Let $\mls X \in \sShv^{\leq n}(S)$ be an $n$-stack. For an $S$-scheme $U$ and sections $u, u'\in \mls X (U)$ let $\I _{u, u'}$ denote $\mls X \times _{\Delta , \mls X \times \mls X , u\times u'}U$.
Observe that since the diagonal map is a section of either projection $\mls X \times \mls X \rightarrow \mls X$ it induces an injection on homotopy sheaves $\pi _{i}(\mls X )\rightarrow \pi _{i}(\mls X \times \mls X)$ for all $i$, and therefore $\I _{u, u'}\in \sShv^{\leq n-1}(U)$.
By induction we therefore know what it means for $\I _{u, u'}$ to be algebraic. Similarly, for two $S$-schemes $U$ and $V$ with sections $u\in \mls X (U)$ and $v\in \mls X (V)$, the homotopy fiber product $U\times_{u,\mls X,v}V$ fits into the  homotopy pullback diagram
\[
    \begin{tikzcd}
        U\times_{u,\mls X,v}V\arrow{d}\arrow{r}&U\times_SV\arrow{d}{u\times v}\\
        \mls X\arrow{r}{\Delta}&\mls X\times_S\mls X.
    \end{tikzcd}
\]
and therefore $U\times_{u,\mls X,v}V\in \sShv^{\leq n-1}(U\times _SV)$.
\end{pg}

\begin{defn} We now define (i) and (ii) for $n$-stacks. We say that an $n$-stack $\mls X$ is \emph{algebraic} if the following conditions hold:
\begin{enumerate}
    \item\label{item:inductive step 1} The diagonal $\Delta \colon \mls X \rightarrow \mls X \times \mls X $ is representable by algebraic $(n-1)$-stacks: For every $S$-scheme $U$ and sections $u, u'\in \mls X (U)$ the homotopy fiber product $\I _{u, u'}\in \sShv^{\leq n-1}(U)$ is an algebraic $(n-1)$-stack.
    \item\label{item:inductive step 2} $\mls X$ admits an fppf cover by a scheme: There exists an $S$-scheme $U$ and section $u\in \mls X (U)$ such that for every $S$-scheme $V$ and $v\in \mls X (V)$ the homotopy fiber product $U\times_{u,\mls X,v}V\in\sShv^{\leq n-1}(V)$ is algebraic and the morphism $U\times_{u,\mls X,v}V\to V$ is an fppf cover (note that this notion has been defined by our induction hypothesis).
\end{enumerate}
Let $f\colon\mls X\to\mls Y$ be a morphism of algebraic $n$-stacks. We say that $f$ is an \emph{fppf cover} if there exist $S$-schemes $U,V$, sections $u\in\mls X(U)$ and $v\in\mls Y(V)$ having the properties of~\eqref{item:inductive step 2} above, and a commutative diagram
\begin{equation}\label{eq:cover-cover diagram}
    \begin{tikzcd}
        U\arrow{d}[swap]{u}\arrow{r}{g}&V\arrow{d}{v}\\
        \mls X\arrow{r}{f}&\mls Y
    \end{tikzcd}
\end{equation}
where $g$ is an fppf cover of schemes.
This completes the inductive step and hence the definition of the notions (i) and (ii).
\end{defn}

\begin{defn}
    We say that an object $\mls X \in \sShv (S)$ is \emph{algebraic} if it is isomorphic to the image of an algebraic $n$-stack under the embedding $\sShv^{\leq n}(S)\hookrightarrow\sShv(S)$ for some $n\geq 0$. To distinguish from the usual terminology of algebraic stacks we will refer to an algebraic object $\mls X \in \sShv(S)$ as an \emph{algebraic $\infty$-stack}. We remark that an $n$-stack is algebraic if and only if it is algebraic viewed as an $(n+1)$-stack, as follows from the same argument used in \cite[Proof of 1.2]{Simpson}. Thus an $n$-stack is algebraic if and only if it is an algebraic $\infty$-stack when viewed as an object of $\sShv(S)$.
\end{defn}

\begin{pg}\label{ssec:properties of algebraic stacks}{\bf Properties of algebraic $\infty$-stacks and their morphisms.}
In \cite{Simpson} a theory of $n$-stacks satisfying certain basic properties is assumed to exist and the article then proceeds to draw consequences for the theory of algebraic $n$-stacks assuming appropriate foundations.  With the foundations provided by \cite{LurieHTT} as discussed above, the results of \cite{Simpson} can then be applied. As noted above we work in this article with fppf covers rather than smooth covers as in \cite{Simpson}.  Nonetheless, for the following  basic properties of algebraic $n$-stacks the arguments of \cite{Simpson} can be copied over replacing only ``smooth'' with ``fppf'' and we simply give the reference.
\end{pg}

\begin{lem}\label{lem:algebraic stacks and fppf cover}
    If $\mls X \in \sShv (S)$ is a stack for which there exists an fppf cover $U\rightarrow S$ such that the base change $\mls X \times _SU\in \sShv (U)$ is algebraic, then $\mls X $ is algebraic.
\end{lem}
\begin{proof}
    See \cite[2.2 and 2.5]{Simpson}.
\end{proof}

\begin{lem}\label{lem:fiber product of algebraic stacks}
    If $\mls X \rightarrow \mls Z$ and $\mls Y\rightarrow \mls Z$ are morphisms of algebraic $\infty$-stacks then the homotopy fiber product $\mls X \times _{\mls Z }\mls Y$ is algebraic.
\end{lem}
\begin{proof}
    See \cite[2.1]{Simpson}.
\end{proof}


Certain properties of algebraic spaces and morphisms of algebraic spaces extend in a formal way to algebraic $\infty$-stacks. Let $\mls P$ be a property of morphisms of algebraic spaces which is fppf local on the source and target, is preserved under arbitrary base change on the target, and is held by all isomorphisms.
For example, we may take $\mls P$ to be the property of being \emph{flat},  \emph{surjective}, \emph{faithfully flat}, \emph{locally of finite type}, or \emph{locally of finite presentation}.

\begin{defn}\label{def:properties of stacks}
     We say that a morphism $f\colon\mls X\to\mls Y$ of algebraic $\infty$-stacks \emph{has property $\mls P$} if there exist $S$-schemes $U,V$, fppf covers $u\colon U\to\mls X$ and $v\colon V\to\mls Y$, and a diagram~\eqref{eq:cover-cover diagram} such that the morphism $g\colon U\to V$ has property $\mls P$.
\end{defn}

We note that this agrees with our previous definition of an fppf cover in \ref{pg:definition of an algebraic n stack}. 

\begin{lem}\label{lem:P lemma}
    If $\mls P$ is a property of morphisms of algebraic spaces as above, then the resulting notion for morphisms of algebraic $\infty$-stacks is again fppf local on the source and target and is preserved under arbitrary base change on the target.
\end{lem}
\begin{proof}
    See the proof of \cite[1.1]{Simpson}. 
\end{proof}

\begin{defn}\label{def:qc qs}
    We say that an algebraic $\infty$-stack $\mls X$ is \emph{quasicompact} over $S$ if there exists an $S$-scheme $U$ and a surjective morphism $U\to\mls X$ such that the composition $U\to\mls X\to S$ is a quasi-compact morphism. A morphism $f\colon\mls X\to\mls Y$ of algebraic $\infty$-stacks is \emph{quasicompact} if for every $S$-scheme $U$ and section $u\in\mls Y(U)$ the fiber product $\mls X\times_{f,\mls Y,u}U$ is quasicompact over $U$.
\end{defn}

\begin{lem}\label{lem:qc lemma}
    The property of a morphism of algebraic $\infty$-stacks being quasicompact is fppf local on the target (but not on the source), is preserved under arbitrary base change on the target, and is stable under compositions.
\end{lem}
\begin{proof}
    This follows from the same argument proving the corresponding result for ordinary stacks \cite[0DQK]{stacks-project}.
\end{proof}

\begin{lem}\label{lem:lfp and qc lemma}
    Let $f\colon\mls X\to\mls Y$ be a morphism of algebraic $\infty$-stacks over $S$. If $\mls X$ is locally of finite presentation over $S$ and $\mls Y$ is locally of finite type over $S$, then $f$ is locally of finite presentation.
\end{lem}
\begin{proof}
    This again follows from the same argument as for ordinary stacks  \cite[06Q6]{stacks-project}.
\end{proof}

\begin{pg}{\bf The Dold--Kan correspondence.}
The Dold--Kan correspondence defines a functor
\[
    \DK \colon \mls D^{\leq 0}(\Ab )\rightarrow \Spc 
\]
with the property that for $\C \in \mls D^{\leq 0}(\Ab )$ we have a canonical isomorphism $\rH^n(\C)\simeq \pi _{-n}(\DK (\C ))$. Let $S$ be an algebraic space. The functor $\DK$ induces a functor
\begin{equation}\label{eq:DK corr}
    \st\colon\mls D^{\leq 0}(S)\to\sShv(S).
\end{equation}
For $\bC \in\mls D^{\leq 0}(S)$ we call $\st (\bC )$ the \emph{associated stack of $\bC $.}  For $n\geq 0$, the functor $\st$ restricts to a functor
\[
    \mls D^{[-n,0]}(S)\rightarrow \sShv^{\leq n}(S)
\]
taking values in $n$-stacks.
\end{pg}

The Dold--Kan correspondence, as stated here, can be extended to an equivalence of stable $\infty$-categories between $\mls D(\Ab )$ and the $\infty$-category of certain symmetric spectra, and the forgetful functor from this second category to spectra preserves limits \cite{Shipley}. It follows that the functor $\st$ preserves limits.

\begin{notation}
    For a complex $\bC\in\mls D(S)$ and an integer $m$, we write $\tau_{\leq m}\bC$ for the truncation whose $i$th term is given by $(\tau_{\leq m}\bC)^i=\bC^i$ if $i<m$, $(\tau_{\leq m}\bC)^m=\ker(d^m_{\bC}:\bC^m\to\bC^{m+1})$, and is 0 otherwise. There is a canonical inclusion $\tau_{\leq m}\bC\to\bC$ which induces an isomorphism on $\mls H^i$ for $i\leq m$.
\end{notation}

\begin{lem}\label{lem:triangle and fiber square}
If 
\[
    \bA \rightarrow \bB \rightarrow \bC \rightarrow \bA [1]
\]
is a distinguished triangle in $\mls D(S)$ then the induced diagram
\[
    \begin{tikzcd}
        \st(\tau_{\leq 0}\bA)\arrow{r}\arrow{d}&\st(\tau_{\leq 0}\bB)\arrow{d}\\
        S\arrow{r}{0}&\st(\tau_{\leq 0}\bC)
    \end{tikzcd}
\]
in $\sShv(S)$ is homotopy cartesian.
\end{lem}
\begin{proof}
    By the definition of a distinguished triangle in a stable $\infty$-category the diagram
    \[
        \begin{tikzcd}
            \bA\arrow{d}\arrow{r}&\bB\arrow{d}\\
            0\arrow{r}{0}&\bC
        \end{tikzcd}
    \]
    is homotopy cartesian. The truncation functor $\tau_{\leq 0}\colon\mls D(S)\to\mls D^{\leq 0}(S)$ is right adjoint to the inclusion $\mls D^{\leq 0}(S)\hookrightarrow\mls D(S)$, hence preserves limits, and as noted above the functor $\st$ preserves limits.
\end{proof}


\begin{pg}{\bf Algebraic complexes.}
\end{pg}
\begin{defn}\label{def:algebraic complexes}
    Let $\bC\in\mls D(S)$ be a complex of fppf sheaves on $S$. We say that $\bC$ is \emph{algebraic} if $\bC\in\mls D^{\leq 0}(S)$ and the associated stack $\st(\bC)$ is algebraic. We say that a sheaf $\mls F\in\Shv(S)$ is \emph{algebraic} if it is isomorphic to $\rH_{\bC}$ for an algebraic complex $\bC\in\mls D(S)$.
\end{defn}


\begin{rem}
    An algebraic complex is necessarily bounded below, and hence is contained in $\mls D^{[-n,0]}(S)$ for some $n\geq 0$. 
\end{rem}

\begin{example}
    If $\bG$ is a commutative group algebraic space over $S$ then the corresponding complex $\bG\in\mls D(S)$ given by $\bG$ placed in degree 0 is algebraic.
\end{example}

\begin{rem}[Relation with Picard stacks]\label{P:6.29}
In \cite[XVIII, 1.4]{SGA4} Deligne associates to a two-term complex $\bC\in \mls D^{[-1, 0]}(S)$ a stack $\text{ch}(\bC)$ over $S$.  We can understand this construction as follows.
As discussed in \cite[1.2.3]{LurieHTT} passage to the homotopy category defines an equivalence of categories between $1$-truncated $\infty$-categories and ordinary categories, with a quasi-inverse provided by the nerve functor.  From this it follows that $\sShv^{\leq 1}(S)$ is equivalent to the $2$-category of stacks in categories, and it follows from the constructions of $\st (\bC )$ and $\text{ch}(\bC )$ for a complex $\bC \in \mls D^{[-1, 0]}(S)$ that the ordinary stack associated to $\st (\bC )$ is equivalent to $\ch (\bC )$.  From this and Artin's theorem \cite[\href{https://stacks.math.columbia.edu/tag/06DC}{Tag 06DC}]{stacks-project} it follows that $\st (\bC )$ is algebraic as an $\infty $-stack if and only if $\text{ch}(\bC )$ is an algebraic stack in the classical sense.
\end{rem}





\begin{notation}
    For a property $\mls P$ as in \S\ref{ssec:properties of algebraic stacks}, we say that a morphism $f\colon\bA\to\bB$ of algebraic complexes has $\mls P$ if the corresponding morphism $\st(\bA)\to\st(\bB)$ of algebraic $\infty$-stacks has $\mls P$. We adopt the same convention for quasicompactness (in the sense of \ref{def:qc qs}).
\end{notation}

\begin{notation}\label{not:cat of alg complexes}
    We write $\mls D^{\mathrm{alg}}(S)\subset\mls D^{\leq 0}(S)$ for the full subcategory spanned by the algebraic complexes.
\end{notation}

\begin{lem}\label{lem:big list of properties}
    Let $\bA,\bB,\bC\in\mls D(S)$ be complexes and let
    \[
        \bA\to\bB\to\bC\to\bA[1]
    \]
    be a distinguished triangle in $\mls D(S)$. 
    \begin{enumerate}
        \item\label{item:fppf cover} If $U\to S$ is an fppf cover, then $\bC$ is algebraic if and only if $\bC|_U$ is algebraic.
        \item\label{item:2 out of 3 for alg} If $\tau_{\leq 0}\bB$ and $\tau_{\leq 0}\bC$ are algebraic, then $\tau_{\leq 0}\bA$ is algebraic.
    \end{enumerate}
\end{lem}
\begin{proof}
    Claim~\eqref{item:fppf cover} follows from \ref{lem:algebraic stacks and fppf cover}. For claim~\eqref{item:2 out of 3 for alg}, we use \ref{lem:triangle and fiber square} to identify $\st (\tau _{\leq 0}\bA )$ with the fiber product $\st (\tau _{\leq 0}\bB)\times _{\st (\tau _{\leq 0}\bC ), 0}S$, which is algebraic by \ref{lem:fiber product of algebraic stacks}.
\end{proof}

\begin{notation}
    For a complex $\bC\in\mls D(S)$ and an integer $n\leq 0$ we set
    \[
        \mls S_n(\bC):=\tau_{\leq 0}(\bC[n]).
    \]
\end{notation}
By construction, for a complex $\bC\in\mls D(S)$ we have that $\mls S_n(\bC)\in\mls D^{\leq 0}(S)$, and for each integer $i\leq 0$, taking $\mls H^{i}$ of the inclusion $\mls S_n(\bC)\to\bC[n]$ yields an isomorphism
    \[
        \mls H^{i}(\mls S_n(\bC))\simeq\mls H^{n+i}(\bC).
    \]

\begin{lem}\label{lem:diagonal of alg complex}
    If $\bC\in\mls D^{\leq 0}(S)$ is an algebraic complex then for any $n\leq 0$ the complex $\mls S_n(\bC)$ is algebraic. Furthermore, we have a natural isomorphism
    \[
        \st\left(\mls S_{-1}(\bC)\right)\simeq S\times _{0, \st (\bC ), 0}S=\I_{0,0}.
    \]
\end{lem}
\begin{proof}
    To show that $\mls S_n(\bC)$ is algebraic it will suffice by descending induction on $n$ to consider the case when $n=-1$. This case follows from applying \ref{lem:triangle and fiber square} to the distinguished triangle
    \[
        \bC [-1]\rightarrow 0\rightarrow \bC \rightarrow (\bC [-1])[1]\simeq \bC,
    \]
    which shows that the corresponding diagram
    \begin{equation}\label{eq:diagonal fiber square}
        \begin{tikzcd}
            \st(\mls S_{-1}(\bC))\arrow{r}\arrow{d}&S\arrow{d}{0}\\
            S\arrow{r}{0}&\st(\bC)
        \end{tikzcd}
    \end{equation}
    is homotopy cartesian.
\end{proof}


For shifts in the other direction we need some assumptions.

\begin{prop}\label{prop:diagonal of shift}
    If $\bC\in\mls D^{\leq 0}(S)$ is an algebraic complex then the diagonal $\Delta_{\st(\bC[1])}$ of $\st(\bC[1])$ is relatively algebraic. 
\end{prop}
\begin{proof}
    Let $T$ be an $S$-scheme and consider maps $t,t'\colon T\to\st(\bC[1])$. We will verify that the fiber product 
    \[
        \I _{t, t'}:=T\times_{t,\st(\bC[1]),t'}T
    \]
    is algebraic. For this note that since $\mls H^0(\bC [1])=0$ there exists an fppf cover $U\rightarrow T$ such that the restrictions of $t$ and $t'$ to $U$ are homotopic to $0$. By \ref{lem:algebraic stacks and fppf cover} it will suffice to check that $\I _{t, t'}$ is algebraic after base change to $U$, which reduces the proof to the case when $T=S$ and $t$ and $t'$ are the $0$-morphism. In this case, applying Lemma \ref{lem:diagonal of alg complex} to $\bC[1]$ we get isomorphisms $\st(\mls S_{-1}(\bC[1]))=\st(\bC)\simeq\I_{0,0}$, and so we have a homotopy cartesian diagram
    \begin{equation}\label{eq:diagonal diagram}
        \begin{tikzcd}[column sep=large]
            \st(\bC)\arrow{r}\arrow{d}&S\arrow{d}{(0,0)}\\
            \st(\bC[1])\arrow{r}{\Delta_{\st(\bC[1])}}&\st(\bC[1])\times_S\st(\bC[1]).
        \end{tikzcd}
    \end{equation}
    By assumption $\st(\bC)$ is algebraic, so the diagonal of $\st(\bC[1])$ is relatively algebraic.
\end{proof}

\begin{lem}\label{lem:0 morphism is an fppf cover}
    If $\bC$ is an algebraic complex with $\mls H^0(\bC)=0$, then the zero morphism $0\colon S\to\st(\bC)$ is an fppf cover.
\end{lem}
\begin{proof}
     Choose an fppf cover $u\colon U\to\st(\bC)$ by a scheme. As $\mls H^0(\bC)=0$, any two morphisms $U\to\st(\bC)$ are locally isomorphic, so there is an fppf cover $V\to U$ such that the restriction of $u$ to $V$ is homotopic to the zero morphism $0\colon V\to\st(\bC)$, which is therefore an fppf cover. This morphism factors as the composition
    \[
        V\to S\xrightarrow{0}\st(\bC).
    \]
    It follows from our construction that $V\to S$ is a surjection of fppf sheaves. By an identical argument as in the proof of \cite[06NB]{stacks-project}, it follows that the zero morphism $0\colon S\to\st(\bC)$ is an fppf cover.
\end{proof}

\begin{prop}\label{prop:algebraic complexes and shifts}
    For an algebraic complex $\bC\in\mls D^{\leq 0}(S)$ the following are equivalent.
    \begin{enumerate}
        \item  The shift $\bC[1]$ is algebraic.
        \item The zero morphism $0\colon S\to\st(\bC[1])$ is an fppf cover (that is, is represented by fppf covers of algebraic $\infty$-stacks).
        \item  The associated stack $\st(\bC)$ is flat and locally of finite presentation over $S$.
    \end{enumerate}
    Furthermore, if these conditions hold, then both $\st(\bC)$ and $\st(\bC[1])$ are flat and locally of finite presentation over $S$.
\end{prop}
\begin{proof}    
    $(1)\iff (2)\colon$ If $\bC[1]$ is algebraic then as $\mls H^0(\bC[1])=0$ Lemma \ref{lem:0 morphism is an fppf cover} implies that the zero morphism $0\colon S\to\st(\bC[1])$ is an fppf cover. Conversely, by \ref{prop:diagonal of shift} the diagonal of $\st(\bC[1])$ is always algebraic, so if the 0-morphism if an fppf cover then $\bC[1]$ is algebraic.

    $(2)\iff (3)\colon$ If the $0$-morphism $0\colon S\to\st(\bC[1])$ is an fppf cover then from the homotopy cartesian diagram
    \begin{equation}\label{eq:homotopy cart diagram for C}
        \begin{tikzcd}
            \st(\bC)\arrow{d}\arrow{r}&S\arrow{d}{0}\\
            S\arrow{r}{0}&\st(\bC[1])
        \end{tikzcd}
    \end{equation}
    we see that $\st(\bC)$ is flat and locally of finite presentation over $S$. Conversely, suppose that $\st(\bC)$ is flat and locally of finite presentation over $S$.
    To verify that the zero morphism is an fppf cover we must check that for every $S$-scheme $T$ and morphism $t\colon T\to\st(\bC[1])$ the fiber product $T\times_{t,\st(\bC[1]),0}S$ is faithfully flat and locally of finite presentation over $T$. As $\mls H^0(\bC[1])=0$ we may find an fppf cover $U\to T$ over which $t$ becomes homotopic to $0$. As being fppf is fppf local on the target, this reduces us to the case when $T=S$ and $t=0$. In this case the diagram~\eqref{eq:homotopy cart diagram for C} shows that the fiber product is simply $\st(\bC)$, which is flat and locally of finite presentation over $S$.

    To verify the final claim, we note that if the zero morphism is an fppf cover, then the diagram~\eqref{eq:homotopy cart diagram for C} shows that $\st(\bC)$ is flat and locally of finite presentation over $S$, and furthermore as the composition $S\xrightarrow{0}\st(\bC[1])\to S$ is flat and locally of finite presentation, it follows that $\st(\bC[1])$ is flat and locally of finite presentation over $S$.
\end{proof}

\begin{pg}{\bf Finitely presented algebraic complexes.}
\end{pg}
\begin{defn}\label{def:algebraic complex fp}
    We say that an algebraic complex $\bC\in\mls D^{\leq 0}(S)$ is \emph{finitely presented} over $S$ if it is locally of finite presentation over $S$ (that is, $\st(\bC)$ is locally of finite presentation over $S$ in the sense of \ref{def:properties of stacks}) and for every integer $n\leq 0$ the stack $\st(\mls S_n(\bC))$ is quasicompact over $S$ and the zero section $0\colon S\to\st(\mls S_n(\bC))$ is quasicompact (in the sense of Definition \ref{def:qc qs}).
\end{defn}

We remark that this is a special case of a more general notion for algebraic $\infty$-stacks, which we omit.

\begin{rem}
    If $\bG$ is a commutative group algebraic space over $S$, then $\bG$ is finitely presented over $S$ in the sense of algebraic spaces \cite[03XP]{stacks-project} if and only if $\bG$ viewed as an algebraic complex concentrated in degree 0 is finitely presented in the sense of \ref{def:algebraic complex fp}. Indeed, we have $\mls S_n(\bG)=0$ for $n\leq -1$, so $\bG$ is finitely presented as an algebraic complex if and only if $\bG$ is locally of finite presentation and the maps $\bG\to S$ and $0\colon S\to\bG$ are both quasicompact. It follows from the pullback diagrams
    \[
        \begin{tikzcd}
            \bG\arrow{d}\arrow{r}{\Delta_{\bG}}&\bG\times_S\bG\arrow{d}{\text{pr}_1-\text{pr}_2}\\
            S\arrow{r}{0}&\bG
        \end{tikzcd}
        \hspace{1cm}
        \begin{tikzcd}
            S\arrow{r}{0}\arrow{d}&\bG\arrow{d}{i_1}\\
            \bG\arrow{r}{\Delta_{\bG}}&\bG\times_S\bG
        \end{tikzcd}
    \]
    that the zero morphism $0\colon S\to\bG$ is quasicompact if and only if $\bG$ is quasiseparated as an algebraic space.
\end{rem}

\begin{lem}\label{lem:finite presentation lemma}
    Let $\bA,\bB,\bC\in\mls D(S)$ be complexes whose truncations $\tau_{\leq 0}\bA,\tau_{\leq 0}\bB,$ and $\tau_{\leq 0}\bC$ are algebraic.
    \begin{enumerate}
        \item\label{item:fp 1} If $U\to S$ is an fppf cover, then $\bC$ is finitely presented over $S$ if and only if $\bC|_U$ is finitely presented over $S$.
        \item\label{item:fp 1.5} If $\tau _{\leq 0}\bC$ is locally of finite presentation over $S$ (resp. of finite presentation over $S$) then for any $n\leq 0$ the complex $\mls S_{n}(\bC)$ is also locally of finite presentation over $S$ (resp. of finite presentation over $S$).
        \item\label{item:fp 2} Let
    \begin{equation}\label{eq:triangle}
        \bA\to\bB\to\bC\to\bA[1]
    \end{equation}
    be a distinguished triangle in $\mls D(S)$. If $\tau_{\leq 0}\bB$ and $\tau_{\leq 0}\bC$ are finitely presented over $S$, then so is $\tau_{\leq 0}\bA$.
    \end{enumerate}
\end{lem}
\begin{proof}
    Claim~\eqref{item:fp 1} follows from the fact that the properties of being locally of finite presentation and quasicompact are fppf local on $S$ (see \ref{lem:P lemma} and \ref{lem:qc lemma}).
    
    For claim~\eqref{item:fp 1.5}, suppose that $\st(\bC)$ is locally of finite presentation over $S$, and consider the homotopy cartesian diagram
    \[
        \begin{tikzcd}
            \st(\mls S_{-1}(\bC))\arrow{r}\arrow{d}&S\arrow{d}{0}\\
            S\arrow{r}{0}&\st(\bC).
        \end{tikzcd}
    \]
    If $\st(\bC)$ is locally of finite presentation over $S$, then by \ref{lem:lfp and qc lemma} the zero morphism $0\colon S\to\st(\bC)$ is also locally of finite presentation, so by \ref{lem:P lemma} the stack $\st(\mls S_{-1}(\bC))$ is locally of finite presentation over $S$. By descending induction on $n$ the same is true for $\st(\mls S_n(\bC))$ for all $n\leq 0$. Moreover, if $\st(\bC)$ is finitely presented, then it is immediate from the definitions that also each $\st(\mls S_n(\bC))$ is finitely presented.
    
    For~\eqref{item:fp 2}, fix $n\leq 0$, and apply a shift by $n$ to the triangle~\eqref{eq:triangle} and then $\st\circ\tau_{\leq 0}$ to obtain a homotopy cartesian diagram
    \[
        \begin{tikzcd}
            \st(\mls S_n(\bA))\arrow{d}\arrow{r}&\st(\mls S_n(\bB))\arrow{d}\\
            S\arrow{r}{0}&\st(\mls S_n(\bC))
        \end{tikzcd}
    \]
    of algebraic $\infty$-stacks. By assumption the zero morphism on the bottom is quasicompact, so \ref{lem:qc lemma} implies that $\st(\mls S_n(\bA))$ is quasicompact over $\st(\mls S_n(\bB))$. The latter is itself quasicompact over $S$ by assumption, so $\st(\mls S_n(\bA))$ is quasicompact over $S$. Similarly, by first rotating the triangle~\eqref{eq:triangle} and then applying the same functors we obtain a homotopy cartesian diagram
    \[
        \begin{tikzcd}
            \st(\mls S_{n-1}(\bC))\arrow{r}\arrow{d}&\st(\mls S_n(\bA))\arrow{d}\\
            S\arrow{r}{0}&\st(\mls S_n(\bB)).
        \end{tikzcd}
    \]
    By assumption the zero morphism on the bottom is quasicompact, so the top horizontal morphism is as well. The zero morphism $0\colon S\to\st(\mls S_n(\bA))$ factors as the precomposition of this map with the zero morphism $0\colon S\to\st(\mls S_{n-1}(\bC))$, which is also quasicompact by assumption. It follows that the zero morphism of $\st(\mls S_n(\bA))$ is quasicompact.
\end{proof}

\begin{lem}\label{lem:qcqs lemma}
    Let $\bC$ be an algebraic complex with $\mls H^0(\bC)=0$. Then $\st(\bC)$ is quasicompact, and the zero morphism $0\colon S\to\st(\bC)$ is quasicompact if and only if $\st(\mls S_{-1}(\bC))$ is quasicompact.
\end{lem}
\begin{proof}
    As $\mls H^0(\bC)=0$, Lemma \ref{lem:0 morphism is an fppf cover} gives that the zero morphism $0\colon S\to\st(\bC))$ is an fppf cover, and in particular surjective, so $\st(\bC)$ is quasicompact. The second claim follows from a consideration of the diagram~\eqref{eq:diagonal fiber square} and the fact that quasicompactness is fppf local on the target \ref{lem:qc lemma}.
\end{proof}

\begin{lem}\label{prop:properties of morphisms of spaces and shifts}
    If $\bC\in\mls D(\bC)$ is a complex such that both $\bC$ and $\bC[1]$ are algebraic, then $\bC$ is finitely presented if and only if $\bC[1]$ is finitely presented.
\end{lem}
\begin{proof}
    By \ref{prop:algebraic complexes and shifts} both $\st(\bC)$ and $\st(\bC[1])$ are automatically locally of finite presentation over $S$. Furthermore, for an integer $n$, we have $\mls S_n(\bC[1])\simeq\mls S_{n+1}(\bC)$. Thus, if $\st(\mls S_n(\bC[1]))$ is quasicompact and has quasicompact zero morphism for all $n\leq 0$ then the same is true for $\st(\mls S_n(\bC))$ for all $n\leq 0$. Conversely, if $\st(\mls S_n(\bC))$ is quasicompact and has quasicompact zero morphism for all $n\leq 0$ then the same is true for $\st(\mls S_n(\bC[1]))$ for $n\leq -1$. We note that we have $\mls H^0(\bC[1])=0$ and $\st(\mls S_{-1}(\bC[1]))=\st(\mls S_0(\bC))=\st(\bC)$, so \ref{lem:qcqs lemma} implies that the same is true for $\st(\mls S_0(\bC[1]))$.
\end{proof}

\begin{pg}{\bf Stably algebraic complexes.} The category $\mls D^{\alg}(S)\subset\mls D^{\leq 0}(S)$ of algebraic complexes is not closed under the shift $[-1]$, and thus is not a stable $\infty$-category. To address this we introduce the following notion of \emph{stably algebraic complexes}, which have the advantage of forming a stable $\infty$-category, and moreover include the examples in our applications in this article.
\end{pg}
\begin{defn}
    We say that a complex $\bC\in\mls D(S)$ is \emph{stably algebraic} if for all $n\gg 0$ the shift $\bC[n]$ is algebraic. A sheaf $\mls F\in\Shv(S)$ is \emph{stably algebraic} if it is isomorphic to $\rH_{\bC}$ for a stably algebraic complex $\bC\in\mls D(S)$.
\end{defn}

\begin{rem}
    A stably algebraic complex is necessarily bounded.
\end{rem}

We give examples of stably algebraic complexes below in \ref{C:6.24}.

\begin{lem}\label{lem:st alg lemma}
Let $\bA,\bB,\bC\in\mls D(S)$ be complexes. 
    \begin{enumerate}
        \item\label{item:stably algebraic and shift} $\bC$ is stably algebraic if and only if $\bC[n]$ is stably algebraic for some integer $n$, and this in turn holds if and only if $\bC[n]$ is stably algebraic for all $n$.
        \item\label{item:fppf cover st alg} If $U\to S$ is an fppf cover, then $\bC$ is stably algebraic if and only if $\bC|_U$ is stably algebraic.
        \item\label{item:2 out of 3 for stably alg} Let
    \[
        \bA\to\bB\to\bC\to\bA[1]
    \]
    be a distinguished triangle in $\mls D(S)$. If any two of $\bA,\bB,$ and $\bC$ are stably algebraic, then so is the third.
    \end{enumerate}
\end{lem}
\begin{proof}
    Claim~\eqref{item:stably algebraic and shift} is immediate from the definitions. Claim~\eqref{item:fppf cover st alg} follows by applying a shift $[n]$ for a sufficiently large $n$ and using \ref{lem:big list of properties}. Similarly, claim~\eqref{item:2 out of 3 for stably alg} follows by taking a shift by a sufficiently large integer and applying \ref{lem:big list of properties}~\eqref{item:2 out of 3 for alg} to an appropriate rotation of the triangle.
\end{proof}

\begin{notation}\label{not:cat of stably alg complexes}
    We write $\mls D^{\salg}(S)\subset\mls D(S)$ for the full subcategory spanned by the stably algebraic complexes.
\end{notation}

\begin{cor}\label{cor:stable infty category}
    The subcategory $\mls D^{\salg}(S)\subset\mls D(S)$ is a stable $\infty$-category and is closed under shifts and cones in $\mls D(S)$. Furthermore, for a complex $\bC\in\mls D(S)$ and an fppf cover $U\to S$, we have that $\bC\in\mls D^{\salg}(S)$ if and only if $\bC|_U\in\mls D^{\salg}(U)$.
\end{cor}
\begin{proof}
    This follows from \ref{lem:st alg lemma}.
\end{proof}

\begin{prop}\label{prop:stably alg criterion}
    The following are equivalent for a complex $\bC\in\mls D(S)$.
    \begin{enumerate}
        \item  $\bC$ is stably algebraic.
        \item  There exists an integer $n\geq 0$ such that the shift $\bC[n]$ is contained in $\mls D^{\leq 0}(S)$ and the associated stack $\st(\bC[n])$ is algebraic and is flat and locally of finite presentation.
        \item  If $n$ is an integer such that $\bC[n]\in\mls D^{\leq 0}(S)$, then the associated stack $\st(\bC[n])$ is algebraic and is flat and locally of finite presentation.
    \end{enumerate}
\end{prop}
\begin{proof}
    $(1)\implies(2)\colon$ If $\bC$ is a stably algebraic complex, then we may find an integer $n$ such that $\bC[n]$ and $\bC[n+1]$ are both algebraic. By \ref{prop:algebraic complexes and shifts} the associated stack $\st(\bC[n])$ is flat and locally of finite presentation over $S$.

    $(2)\implies (3)\colon$ We may assume $\bC$ is nonzero, and by replacing $\bC$ with a shift we may assume that $\bC\in\mls D^{\leq 0}(S)$ and that $\st(\bC)$ is algebraic and fppf over $S$. Suppose that $\bC$ has amplitude $[a,b]$ for some $a\leq b\leq 0$. We will verify that if $n\geq b$ then $\bC[n]$ is algebraic and $\st(\bC[n])$ is flat and locally of finite presentation over $S$. For $n\geq 0$ this follows from \ref{prop:algebraic complexes and shifts}. Suppose that $b\leq n\leq 0$. By induction it will suffice to consider the case $n=-1$. By \ref{lem:diagonal of alg complex} the complex $\mls S_{-1}(\bC)$ is algebraic. The amplitude of $\bC$ is bounded away from zero, so we have $\mls S_{-1}(\bC)\simeq\bC[-1]$. As $\bC\simeq(\bC[-1])[1]$ is algebraic, \ref{prop:algebraic complexes and shifts} gives that $\st(\bC[-1])$ is flat and locally of finite presentation over $S$.

    $(3)\implies (1)\colon$ This follows from the definition.
\end{proof}

We showed in \ref{prop:properties of morphisms of spaces and shifts} that finite presentation is a stable property of algebraic complexes. We use this to define the following notion of finite presentation for stably algebraic complexes.

\begin{defn}\label{def:stably algebraic complex fp}
    We say that a stably algebraic complex $\bC\in\mls D(S)$ is \emph{finitely presented} over $S$ if the following equivalent conditions hold.
    \begin{enumerate}
        \item [(a)] There exists $n\geq 0$ such that $\bC[n]$ is algebraic and is finitely presented over $S$.
        \item [(b)] If $n$ is an integer such that $\bC[n]\in\mls D^{\leq 0}(S)$, then $\bC[n]$ is algebraic and finitely presented over $S$.
    \end{enumerate}
    Here the equivalence follows from noting that by \ref{prop:stably alg criterion} if $n$ is an integer such that $\bC[n]$ is contained in $\mls D^{\leq 0}(S)$ then $\bC[n]$ is algebraic, combined with \ref{prop:properties of morphisms of spaces and shifts}. 
\end{defn}

We note that the definition of finite presentation for algebraic complexes in \ref{def:algebraic complex fp} and for stably algebraic complexes in \ref{def:stably algebraic complex fp} agree whenever they are both applicable.

\begin{rem}
    Unlike finite presentation, many properties of the associated algebraic stacks are not invariant under shifts $\bC\mapsto\bC[1]$. For example, if $\bG$ is a commutative group scheme which is flat and locally of finite presentation over $S$ but not smooth, then $\st(\bG[n])$ is smooth over $S$ for $n\geq 1$ but not smooth over $S$ for $n=0$.
\end{rem}

The following result will be the source of our examples of stably algebraic complexes.

\begin{prop}\label{prop:locally perfect implies stably algebraic}
   If $\bC\in\mls D(S)$ is a complex for which there exists an fppf cover $U\to S$ such that the restriction $\bC|_U$ is quasi-isomorphic to a bounded complex of group algebraic spaces that are flat and locally of finite presentation over $S$ (resp. flat of finite presentation over $S$), then $\bC$ is stably algebraic (resp. stably algebraic and of finite presentation over $S$).
\end{prop}
\begin{proof}
    We induct on the amplitude of $\bC$. By \ref{lem:st alg lemma}~\eqref{item:fppf cover st alg} and the fact that being of finite presentation is fppf local on $S$, it will suffice to show that if $\bC\in\mls D(S)$ is a bounded complex of group algebraic spaces which are flat and locally of finite presentation over $S$ (resp. of finite presentation over $S$) then $\bC^{}$ is stably algebraic (resp. stably algebraic and of finite presentation over $S$). For an integer $s$ we consider the truncation $\sigma_{\geq s}\bC^{}$, whose $i$th term is $(\sigma_{\geq s}\bC^{})^i=\bC^i$ if $i\geq s$ and is $0$ otherwise. We have distinguished triangles
    \[
        \sigma_{\geq s}\bC^{}\to\bC^s[-s]\to\sigma_{\geq s+1}\bC^{}[1]\to\sigma_{\geq s}\bC^{}[1].
    \]
    By induction on the amplitude of $\bC^{}$, Lemma \ref{lem:st alg lemma}~\eqref{item:2 out of 3 for stably alg}, and Lemma \ref{lem:finite presentation lemma}~\eqref{item:fp 2}, we are reduced to the case when $\bC^{}\simeq\bG[m]$ where $m$ is an integer and $\bG$ is a group algebraic space which is flat and locally of finite presentation over $S$ (resp. of finite presentation over $S$), viewed as a complex concentrated in degree 0. If $\bG$ is flat and locally of finite presentation over $S$, then it follows from \ref{prop:stably alg criterion} that $\bG[m]$ is stably algebraic. It remains to show that if $\bG$ is flat and finitely presented over $S$ then $\bG[m]$ is also finitely presented over $S$, in the sense of \ref{def:stably algebraic complex fp}. For this we may assume that $m=0$. Now as $\mls S_n(\bG)=0$ for $n\leq -1$, to complete the proof we need only observe that the morphisms $\bG\to S$ and $0\colon S\to\bG$ are both quasicompact, which follow from our assumption that $\bG$ is finitely presented over $S$, and in particular is quasicompact and quasiseparated over $S$.
\end{proof}

\begin{example}\label{C:6.24}
    As immediate consequences of \ref{prop:locally perfect implies stably algebraic} we have the following examples of stably algebraic complexes.
    \begin{enumerate}
        \item [(i)] If $\bG$ is a commutative group algebraic space over $S$ which is flat and locally of finite presentation, then for any $n\in\bZ$ the shift $\bG[n]\in\mls D(S)$ is stably algebraic. Furthermore, if $\bG$ is flat and finitely presented, then $\bG[n]$ is also finitely presented for any $n\in\bZ$ (in the sense of \ref{def:stably algebraic complex fp}).
        \item [(ii)] If $\bV^{}\in\mls D(S)$ is a perfect complex of vector bundles then $\bV^{}$ is stably algebraic and finitely presented over $S$.
    \end{enumerate}
\end{example}

\begin{pg}\label{ssec:algebraicity of cohomology}{\bf Algebraicity of cohomology sheaves.} 
\end{pg}
\begin{prop}\label{T:A.1} 
Let $\bC\in \mls D^{\leq 0}(S)$ be an algebraic complex. If $\mls S_{-1}(\bC)$ is flat and locally of finite presentation over $S$, then
\begin{enumerate}
    \item\label{item:TA11} $\mls H^0(\bC)$ is an algebraic space,
    \item\label{item:TA12} the natural map $\bC\rightarrow \mls H^0(\bC)$ is a quasicompact fppf cover, and
    \item\label{item:TA13} $\bC$ is flat (resp. locally of finite presentation) over $S$ if and only if $\mls H^0(\bC)$ is flat (resp. locally of finite presentation) over $S$.
\end{enumerate}
Furthermore, if $\mls S_{-1}(\bC)$ is flat and finitely presented, then $\bC$ is finitely presented if and only if $\mls H^0(\bC)$ is finitely presented.
\end{prop}
\begin{proof}
We begin by verifying that the diagonal of $\mls H^0(\bC)$ is representable.  Let $T$ be an $S$-scheme and let $x, y\in\mls H^0(\bC)(T)$ be two sections. The functor expressing the condition $x=y$ is isomorphic to the functor expressing the condition $x-y=0$. Therefore, to verify the representability of the diagonal it suffices to show that if $T$ is an $S$-scheme and $x\in \mls H^0(\bC)(T)$ is a section then the functor $F$ given by
\[
    F(T'\rightarrow T) = \begin{cases} \emptyset & \text{if $x|_{T'}\neq 0$},\\
    \{*\} & \text{if $x|_{T'} = 0$} \end{cases}
\]
is an algebraic space.  For this note that $F$ is a subfunctor of $T$ so the diagonal of $F$ over $T$ is an isomorphism.  To find an fppf cover of $F$ we may replace $T$ by an fppf cover and therefore may assume that $x$ lifts to a section $\widetilde x\in \rH ^0(T, \bC )$.  Let $\widetilde{F}$ denote the fiber product 
\[
    \widetilde{F}:= \st (\bC )\times _{\Delta , \st (\bC )\times \st (\bC ), 0\times\widetilde x}T,
\]
which is an algebraic stack since $\bC $ is algebraic by assumption. There is a natural map
$\widetilde{F}\rightarrow F$, which we claim is an fppf cover.
 To prove this consider the commutative diagram
 \begin{equation}\label{eq:eqst}
    \begin{tikzcd}
        \bC \oplus \tau _{\leq -1}\bC \arrow{r}{\varepsilon}\arrow{d}[swap]{\mathrm{pr}_1}& \bC \oplus \bC \arrow{d}{\id\oplus\pi}\arrow{r}{\mathrm{pr}_2}&\bC\arrow{d}{\pi}\\
        \bC \arrow{r}{\delta}& \bC \oplus \mls H^0(\bC )\arrow{r}{\mathrm{pr}_2}&\mls H^0(\bC)
    \end{tikzcd}
 \end{equation}
 where $\mathrm{pr}_1$ and $\mathrm{pr}_2$ are projections, $\varepsilon$ is the map $(x, y)\mapsto (x, x+y)$, $\pi\colon \bC\to\mls H^0(\bC)$ is the natural quotient map, and $\delta$ is the map $x\mapsto (x,\pi(x))$. One checks that both squares are homotopy cartesian, and hence the squares in the corresponding diagram of stacks are homotopy cartesian. Using the same symbols for the induced morphisms of stacks we obtain a commutative diagram
\[
    \begin{tikzcd}
        \widetilde{F}\arrow{d}\arrow{r}& F\arrow{d}\arrow{r}& T\arrow{d}{0\times \widetilde x}\\
        \st (\bC )\arrow{r}{\mathrm{id}\times 0}\arrow[equals]{dr}& \st (\bC )\times_S \st (\tau _{\leq -1}\bC )\arrow{d}{\mathrm{pr}_1}\arrow{r}{\varepsilon}& \st (\bC )\times_S \st (\bC )\arrow{d}{\id\times\pi}\\
        & \st (\bC )\arrow{r}{\delta}&  \st (\bC )\times_S \mls H^0(\bC ).
    \end{tikzcd}
\]
The bottom right square is induced by the leftmost square of~\eqref{eq:eqst}, and so is cartesian. It follows that the top two squares are as well. By assumption the stack $\st(\mls S_{-1}(\bC))$ is flat and locally of finite presentation over $S$, so by \ref{prop:algebraic complexes and shifts} the zero morphism $0\colon S\rightarrow\st(\mls S_{-1}(\bC)[1])=\st (\tau _{\leq -1}\bC )$ is an fppf cover. We conclude that $\widetilde{F}\to F$ is an fppf cover, as claimed. As $\widetilde{F}$ is algebraic, it admits an fppf cover by a scheme, and hence so does $F$. We conclude that $F$ is an algebraic space, and therefore that the diagonal of $\mls H^0(\bC )$ is representable.

We now show that the map $\pi\colon\bC\to\mls H^0(\bC)$ is a quasicompact fppf cover. This will prove~\eqref{item:TA12}, and also complete the proof of~\eqref{item:TA11}, as then any cover of $\bC$ by a scheme will yield a cover of $\mls H^0(\bC)$ by a scheme. We must show that for any $S$-scheme $T$ and element $x\in\mls H^0(\bC)(T)$ the map
\[
    \widetilde T:= \st (\bC )\times _{\mls H^0(\bC ), x}T\rightarrow T
\]
is an fppf cover. This can be verified fppf-locally on $T$, so we may assume as before that $x$ lifts to a section $\widetilde x\in \rH ^0(T, \bC )$. Consider the commutative diagram
\begin{equation}\label{eq:diagram for S-1}
    \begin{tikzcd}
        \widetilde{T}\arrow{d}\arrow{r}&T\arrow{d}{\widetilde{x}}\\
        \st(\bC)\times_S\st(\tau_{\leq -1}\bC)\arrow{r}{+}\arrow{d}[swap]{\mathrm{pr}_1}&\st(\bC)\arrow{d}{\pi}\\
        \st(\bC)\arrow{r}{\pi}&\mls H^0(\bC).
    \end{tikzcd}
\end{equation}
The lower square is the diagram of stacks induced by the outer rectangle of~\eqref{eq:eqst}, and so is cartesian. Thus the upper square is cartesian as well, and so $\widetilde{T}\to T$ is isomorphic to $T\times_S\st(\tau_{\leq -1}\bC)\simeq\st(\tau_{\leq -1}\bC|_T)\to T$. By assumption $\st(\mls S_{-1}(\bC))$ is flat and locally of finite presentation over $S$, so by \ref{prop:algebraic complexes and shifts} the stack $\st(\mls S_{-1}(\bC)[1])\simeq\st(\tau_{\leq -1}\bC)$ is algebraic and fppf over $S$. We conclude that $\widetilde{T}\to T$ is an fppf cover, as claimed, and therefore $\pi\colon\bC\to\mls H^0(\bC)$ is an fppf cover. To show that it is also quasicompact, we note that by \ref{lem:qcqs lemma} the stack $\st(\tau_{\leq -1}\bC)$ is quasicompact over $S$. We have shown that $\pi$ is an fppf cover, so from the lower square of the diagram~\eqref{eq:diagram for S-1} and the fact that quasicompactness is fppf local on the target we deduce that $\pi$ is quasicompact. 

For~\eqref{item:TA13}, we note that the properties of flatness and local finite presentation are fppf local on the source, so~\eqref{item:TA12} implies that $\mls H^0(\bC)$ is flat or locally of finite presentation over $S$ if and only if the same is true for $\bC$.

It remains to verify the final claim following the ``Furthermore''. Suppose that $\mls S_{-1}(\bC)$ is flat and finitely presented. Consider the triangle
    \[
        \bC\to\mls H^0(\bC)\to\mls S_{-1}(\bC)\to\bC[1].
    \]
From \ref{lem:finite presentation lemma}~\eqref{item:fp 2} we get that $\mls H^0(\bC)$ finitely presented implies $\bC$ finitely presented. Conversely, if $\bC$ is finitely presented then a quasicompact cover of $\bC$ by a scheme gives rise to a quasicompact cover of $\mls H^0(\bC)$, which therefore is quasicompact. Furthermore, the zero morphism of $\mls H^0(\bC)$ factors as the composition $S\xrightarrow{0}\st(\bC)\xrightarrow{\pi}\mls H^0(\bC)$. By assumption the first map is quasicompact, and by~\eqref{item:TA12} the second map is also quasicompact. We conclude that the zero morphism of $\mls H^0(\bC)$ is quasicompact, so $\mls H^0(\bC)$ is finitely presented.
\end{proof}

\begin{prop}\label{prop:all coho flat, with fp}
    Let $\bC\in\mls D^{\leq 0}(S)$ be an algebraic complex which is locally of finite presentation over $S$ (resp. of finite presentation over $S$). If for all $i\leq 0$ the cohomology sheaf $\mls H^i(\bC)$ is an algebraic space which is flat over $S$, then $\bC$ is flat, and $\mls H^i(\bC)$ is locally of finite presentation (resp. of finite presentation) for each $i\leq 0$.
\end{prop}
\begin{proof}
    We first consider the case where we only assume $\bC$ to be locally of finite presentation. For this we induct on the amplitude of $\bC$. The base case of amplitude one is immediate. In the general case, we note that the complex $\mls S_{-1}(\bC)$ is locally of finite presentation by \ref{lem:finite presentation lemma}~\eqref{item:fp 1.5}, and is therefore also flat by our induction hypothesis. Thus the assumptions of \ref{T:A.1} are satisfied. Because $\mls H^0(\bC)$ is flat, \ref{T:A.1}~\eqref{item:TA13} then implies that $\bC$ is flat. Furthermore, because $\bC$ is locally of finite presentation, \ref{T:A.1}~\eqref{item:TA13} implies that $\mls H^0(\bC)$ is locally of finite presentation. The result follows by induction.

    In the case when $\bC$ is assumed to be of finite presentation, the same induction combined with the final assertion of \ref{T:A.1} gives that $\mls H^i(\bC)$ is also of finite presentation for $i\leq 0$.
\end{proof}

\begin{prop}\label{prop:all coho flat, with fp 2}
    Let $\bC\in\mls D^{\leq 0}(S)$ be an algebraic complex which is locally of finite presentation over $S$ (resp. of finite presentation over $S$). If for all $i<0$ the cohomology sheaf $\mls H^i(\bC)$ is an algebraic space which is flat over $S$, then $\mls H^0(\bC)$ is an algebraic space that is locally of finite presentation over $S$ (resp. of finite presentation over $S$).
\end{prop}
\begin{proof}
    By \ref{lem:finite presentation lemma}~\eqref{item:fp 1.5} the complex $\mls S_{-1}(\bC)$ is also locally of finite presentation (resp. of finite presentation). Now apply \ref{prop:all coho flat, with fp} to $\mls S_{-1}(\bC)$ and then apply \ref{T:A.1}. We note that \ref{prop:all coho flat, with fp} implies that in fact $\mls H^i(\bC)$ is locally of finite presentation over $S$ (resp. of finite presentation over $S$) for all $i<0$.
\end{proof}

We next give a stable version of the preceding proposition.


\begin{cor}\label{cor:coho sheaves of a stably alg complex}
    Let $\bC\in\mls D(S)$ be a stably algebraic complex (resp. a finitely presented stably algebraic complex). If $n$ is an integer such that for all $i<n$ the cohomology sheaf $\mls H^i(\bC)$ is an algebraic space that is flat over $S$, then $\mls H^n(\bC)$ is an algebraic space that is locally of finite presentation (resp. of finite presentation).
\end{cor}
\begin{proof}
    By applying a shift by a sufficiently large integer we may assume that $\bC\in\mls D^{\leq 0}(S)$ is an algebraic complex and $n\leq 0$. By \ref{prop:stably alg criterion} $\bC$ is locally of finite presentation over $S$ (resp. of finite presentation over $S$). By \ref{lem:finite presentation lemma}~\eqref{item:fp 1.5} the same is true for the complex $\mls S_n(\bC)$, and the result follows by applying \ref{prop:all coho flat, with fp 2} to $\mls S_n(\bC)$. We note that this implies that $\mls H^i(\bC)$ is also locally of finite presentation (resp. of finite presentation) for all $i<n$.
\end{proof}

\begin{cor}\label{cor:generic representability of stably alg complex}
    Let $S$ be a reduced noetherian scheme and let $\bC\in\mls D(S)$ be a finitely presented stably algebraic complex over $S$. For each $n$ there exists a dense open subset $U\subset S$ such that the fppf sheaf of abelian groups $\mls H^n(\bC)|_U$ is representable by a group scheme that is flat and of finite type over $U$.
\end{cor}
\begin{proof}
    As $\bC$ is bounded we have $\mls H^i(\bC)=0$ for $i\ll 0$. Applying \ref{cor:coho sheaves of a stably alg complex} inductively and using generic flatness, we get that there exists a dense open $U\subset S$ such that $\mls H^n(\bC)|_U$ is a group algebraic space that is flat and of finite presentation over $U$. We claim that after further shrinking $U$ we may arrange for $\mls H^n(\bC)|_U$ to be representable by a scheme. Indeed, as $\mls H^n(\bC)|_U$ is quasiseparated, it contains a dense open subscheme \cite[06NH]{stacks-project}, which after shrinking $U$ we may assume to surject onto $U$. By the homogeneity resulting from the group structure, $\mls H^n(\bC)|_U$ is a scheme.
\end{proof}

\begin{cor}\label{cor:cohomology of a stably algebraic complex is constructible}
    Let $S$ be a reduced noetherian scheme and let $\bC\in\mls D(S)$ be a finitely presented stably algebraic complex. For each $n$ the cohomology sheaf $\mls H^n(\bC)$ has the property that there exists a finite partition $S=\bigsqcup_iZ_i$ of $S$ into reduced locally closed subschemes $Z_i\subset S$ such that for each $i$ the restriction of $\mls H^n(\bC)$ to $Z_i$ is representable by a group scheme that is flat and of finite type over $Z_i$.
\end{cor}
\begin{proof}
    Apply \ref{cor:generic representability of stably alg complex} and noetherian induction.
\end{proof}

\section{Global representability results for flat cohomology}\label{sec:algebraicity of cohomology}

     In this section we prove Theorem \ref{T:strongform}. Let $f\colon X\to S$ be a flat proper finitely presented morphism of algebraic spaces of characteristic $p$ and let $\bG$ be a commutative group scheme over $X$.
     Our aim is to show, under various combinations of assumptions on $f$ and $\bG$, that the flat cohomology $\R f_*\bG\in\mls D(S)$ is a stably algebraic complex and is finitely presented over $S$. We recall the key points of this notion from the previous section: First, the full subcategory $\mls D^{\salg}(S)\subset\mls D(S)$ spanned by the stably algebraic complexes is a stable $\infty$-category and is closed under shifts and cones in $\mls D(S)$, second, membership in $\mls D^{\salg}(S)$ may be checked fppf locally on $S$ (see \ref{cor:stable infty category}), and finally by \ref{prop:locally perfect implies stably algebraic} perfect complexes of vector bundles on $S$ and shifts of flat group algebraic spaces of finite presentation over $S$ are examples of finitely presented stably algebraic complexes.
     
\begin{lem}\label{lem:pushforward of perfect complex}
    If $f\colon X\to S$ is a flat proper finitely presented morphism of algebraic spaces and $\bV$ is a perfect complex of vector bundles on $X$, then the flat pushforward $\R f_*\bV$ is a perfect complex of vector bundles on $S$, and in particular is stably algebraic and finitely presented.
\end{lem}
\begin{proof}
    We have that $\bV=\R\bV(\mls V)$ for a perfect complex of locally free sheaves $\mls V$ on $X$. As $f$ is flat and proper, the derived pushforward $\R f_*\ms V\in \mls D(S_{\et},\mls O_{\et})$ is a perfect complex of locally free sheaves on $S$ and is of formation compatible with arbitrary base change \cite[0CTM]{stacks-project}. Thus $\R\bV(\R f_*\mls V)$ is a perfect complex of vector bundles on $S$, and we have a canonical isomorphism
        \[
            \R\bV(\R f_*\mls V)\simeq\R f_*\bV.
        \]
\end{proof}

    \begin{proof}[Proof of Theorem \ref{T:strongform}, cases~\eqref{item:coho etale ell} and~\eqref{item:coho AM}]
        The result in case~\eqref{item:coho AM} follows from \ref{lem:pushforward of perfect complex} and the Artin--Mazur resolution. In case~\eqref{item:coho etale ell}, we may assume that $\bG$ has constant order. We may then write $\bG$ canonically as a direct sum of a finite \'{e}tale group scheme of constant order a power of $p$ and a finite \'{e}tale group scheme of constant order coprime to $p$. In the former case, by filtering by powers of $p$ we further reduce to the case of order $p$, where the result follows from case~\eqref{item:coho AM}. To treat the latter case, suppose that $\bG$ has constant order coprime to $p$ and let $\mls G:=\epsilon_*\bG$ be the restriction of $\bG$ to the small \'{e}tale site of $X$. It follows from the smooth and proper base change theorems in \'{e}tale cohomology that the usual \'{e}tale cohomology groups $\R^i f_*\mls G\in\mls D(S_{\et})$ are compatible with arbitrary base change and are locally constant for each $i$. This implies that the fppf cohomology groups $\R^if_*\bG$ are representable by finite \'{e}tale group schemes. Arguing as in \ref{prop:locally perfect implies stably algebraic} we then get that $\R f_*\bG$ is contained in the stable subcategory generated by finite \'{e}tale group schemes, which implies the result.
    \end{proof}

    We now consider case~\eqref{item:coho height 1} of Theorem \ref{T:strongform}, which notably covers $\bmu_p$. Unlike the previous two cases, this case is essentially new, and uses crucially our results from \S\ref{sec:flat coho via Kan}. We deduce this as a consequence of a more general result. Let $f\colon X\to S$ be a morphism of algebraic spaces of characteristic $p$ and let $\bG$ be a commutative finite locally free height $1$ group scheme over $X$. By Theorem \ref{thm:main kan theorem, over a stack}, we have isomorphisms
        \[
            \R f_*(\mls C_{\bG}[-1])\simeq\R f_*\rH_{\bG}\simeq\rH_{\R f_*\bG}.
        \]
    Thus Theorem \ref{T:strongform} (3) will follow from the following result.

\begin{thm}\label{T:A.2}
    Let $f\colon X\to S$ be a flat proper lci morphism of finite presentation of algebraic spaces of characteristic $p$. For any pair $(\mls V^{},\rho)$ where $\mls V^{}$ is a perfect complex on $X$ and $\rho \colon \rL\F_X^*\mls V^{}\to\mls V^{}$ is a morphism in $\mls D(X_{\et},\mls O_{\et})$, the sheaf $\R f_*\mls C_{(\mls V^{}, \rho )}\in\Shv(S)$ is stably algebraic and finitely presented over $S$.
    Furthermore, $\R f_*\mls C_{(\mls V^{},\rho)}$ is contained in the full stable $\infty$-subcategory of $\Shv(S)$ generated by the sheaves of the form $\rH_{\mls E}$ for perfect complexes $\mls E$ on $S$ and of the form $\mls C_{(\mls W^{},\sigma)}$ for pairs $(\mls W^{},\sigma)$ on $S$ as in \ref{not:notation for C_G, the functor}.
\end{thm}


\begin{proof}
    We will show that $\R f_*\mls C_{(\mls V^{}, \rho )}$ is stably algebraic and of finite presentation. The remaining claim will follow from the proof. We proceed by first considering some special cases of the result.

\begin{pg}[The case when $X=S$]\label{ssec:X=S}
    Consider first the case when $X=S$ and $f$ is the identity. By \ref{lem:big list of properties}~\eqref{item:fppf cover}, the result is local on $S$, so after taking a Zariski open cover of $S$ we may assume that $\mls V^{}$ is strictly perfect, so that $\mls V^i$ is a locally free sheaf for each $i$ and is nonzero for all but finitely many $i$. By \cite[08C9]{stacks-project} after taking a further Zariski open cover we may assume that $\rho$ is a strict map of complexes, hence given in degree $i$ by a map $\rho^i\colon\F_X^*\mls V^i\to\mls V^i$ of coherent sheaves. For an integer $s$ we consider the truncation $\sigma_{\geq s}\mls V^{}$, which is equipped with the map $\sigma_{\geq s}\rho \colon \sigma_{\geq s}\mls V^{}\to\sigma_{\geq s}\mls V^{}$. We have distinguished triangles
    \[
        \sigma _{\geq s+1}\mls V^{}\rightarrow \sigma _{\geq s}\mls V^{}\rightarrow\mls V^{s}[-s]\rightarrow \sigma _{\geq s+1}\mls V^{}[1]
    \]
    which are compatible with the maps $\sigma_{\geq s}\rho$.  These induce distinguished triangles
    \[
        \mls C_{(\sigma _{\geq s+1}\mls V^{},\sigma _{\geq s+1}\rho)}\rightarrow \mls C_{(\sigma _{\geq s}\mls V^{},\sigma _{\geq s}\rho)}\rightarrow \bC _{(\mls V^{s}, \rho^{s})}[-s]\rightarrow \mls C_{(\sigma _{\geq s+1}\mls V^{},\sigma _{\geq s+1}\rho)}[1].
    \]
    By induction on the amplitude of $\mls V^{}$ we are then reduced to the case when $\mls V$ is a single locally free sheaf viewed as a complex concentrated in degree 0. If $\bG$ is the height $\leq 1$ group scheme associated to the pair $(\mls V,\rho)$, then by \ref{thm:main kan theorem, over a stack} we have an isomorphism $\mls C_{(\mls V,\rho)}\simeq\rH_{\bG}[1]\simeq\rH_{\bG[1]}$. Thus $\mls C_{(\mls V,\rho)}$ is stably algebraic by \ref{prop:locally perfect implies stably algebraic}.
\end{pg}

\begin{pg}[The case when $\rho=0$]
\label{ssec:rho=0}
We now consider the case when $\rho=0$. Set $\bV^{}=\R\bV(\mls V^{})$. As in \ref{pg:perfect complex of vector bundles}, we set
\[
    \bG_{\mls V^{}}=\cocone\left(\bV^{}\xrightarrow{\F_{\bV^{}/X}}\bV^{(p)}\right)\in\mls D(X).
\]
By \ref{thm:kan extension theorem for a complex} we have a canonical isomorphism $\mls C_{(\mls V^{},0)}\iso\rH_{\bG_{\mls V^{}}}[1]$. We obtain isomorphisms
\[
    \R f_*\mls C_{(\mls V^{},0)}\simeq\R f_*\rH_{\bG_{\mls V^{}}}[1]\simeq\rH_{\R f_*\bG_{\mls V^{}}[1]}.
\]
By \ref{lem:pushforward of perfect complex} and \ref{prop:locally perfect implies stably algebraic} the sheaf $\R f_*\bG_{\mls V^{}}$ is stably algebraic, so $\R f_*\mls C_{(\mls V^{},0)}$ is stably algebraic as well.
\end{pg}

\begin{pg}[The general case]


We now consider the general case. As $f$ is flat and proper, the derived pushforward $\R f_*\mls V^{}$ is a perfect complex on $S$ \cite[0CTM]{stacks-project}. Let $\rho'\colon\mls V\to\F_{X*}\mls V$ be the adjoint to $\rho$, and let $\sigma\colon\rL\F_S^*(\R f_*\mls V)\to\R f_*\mls V$ denote the adjoint to the map
\[
    \R f_*\mls V\xrightarrow{\R f_*\rho'}\R f_*(\F_{X*}\mls V)\simeq\F_{S*}(\R f_*\mls V).
\]
We consider the functor $\mls C_{(\R f_*\mls V^{},\sigma)}\colon\Alg_S\to\mls D(\Ab)$ associated to the pair $(\R f_*\mls V^{},\sigma)$. Using the description~\eqref{eq:description of C as a functor} of the functors $\mls C_{(\mls V^{},\rho)}$ and $\mls C_{(\R f_*\mls V^{},\sigma)}$, the adjunction map $f^*\R f_*\mls V\rightarrow\mls V$ gives rise to a morphism
\begin{equation}\label{eq:map of dts, functor level}
    \begin{tikzcd}
        \mls C_{(\R f_*\mls V^{},\sigma)}\arrow{r}\arrow{d}{\alpha}&\left(\R f_{*}\rH_{\mls V^{}}\right)\otimes\F_{*}\uLZOmega^1_{S}\arrow{r}\arrow{d}{\beta}&\left(\R f_{*}\rH_{\mls V^{}}\right)\otimes\uLOmega^1_S\arrow{r}{+1}\arrow{d}{\gamma}&\,\\
        \R f_*\mls C_{(\mls V^{},\rho)}\arrow{r}&\R f_{*}\left(\rH_{\mls V^{}}\otimes\F_{*}\uLZOmega^1_{X}\right)\arrow{r}&\R f_{*}\left(\rH_{\mls V^{}}\otimes\uLOmega^1_{X}\right)\arrow{r}{+1}&\,
    \end{tikzcd}
\end{equation}
of distinguished triangles in $\Shv(S)$. By the case $X = S$ settled above, we have that $\mls C_{(\R f_*\mls V^{},\sigma)}$ is stably algebraic. As $\mls D^{\salg}(S)$ is a stable $\infty$-category and is a full subcategory of $\mls D(S)$, the complex $\R f_*\mls C_{(\mls V^{}, \rho )}$ is stably algebraic if and only if the same is true for $\cone(\alpha)$. Taking cones yields a distinguished triangle
\begin{equation}\label{eq:triangle of cones}
    \cone(\alpha)\to\cone(\beta)\to\cone(\gamma)\xrightarrow{+1}.
\end{equation}
We claim that $\cone(\beta)$ and $\cone(\gamma)$ are stably algebraic. This will suffice to prove the result, as then~\eqref{eq:triangle of cones} will be isomorphic to the image of a distinguished triangle in $\mls D^{\salg}(S)$.

Consider an affine $S$-scheme $g\colon\Spec A\to S$. Write $X_A=X\times_S\Spec A$ and let $f_A$ (resp. $g_A$) be the base change of $f$ (resp. $g$), so that we have a cartesian diagram
\[
    \begin{tikzcd}
        X_A\arrow{d}[swap]{f_A}\arrow{r}{g_A}&X\arrow{d}{f}\\
        \Spec A\arrow{r}{g}&S.
    \end{tikzcd}
\]
Write $\mls V^{}_A=\rL g_A^*\mls V^{}$.
Evaluating the diagram~\eqref{eq:map of dts, functor level} on $A$, we obtain a diagram
\[
    \begin{tikzcd}
        \mls C_{(\R f_*\mls V^{},\R f_*\rho)}(A)\arrow{r}\arrow{d}{\alpha(A)}&\left(\R f_{A*}\mls V^{}_A\right)\otimes\F_{A*}\rL\Z\Omega^1_{A}\arrow{r}\arrow{d}{\beta(A)}&\left(\R f_{A*}\mls V^{}_A\right)\otimes\rL\Omega^1_A\arrow{r}{+1}\arrow{d}{\gamma(A)}&\,\\
        \left(\R f_*\mls C_{(\mls V^{},\rho)}\right)(A)\arrow{r}&\R f_{A*}\left(\mls V^{}_{A}\otimes\F_{X_A*}\rL\Z\Omega^1_{X_A}\right)\arrow{r}&\R f_{A*}\left(\mls V^{}_{A}\otimes\rL\Omega^1_{X_A}\right)\arrow{r}{+1}&\,
    \end{tikzcd}
\]
describing a morphism of distinguished triangles in $\mls D(\Ab)$.
The projection formula gives an isomorphism
\[
    (\R f_{A*}\mls V^{}_A)\otimes\rL\Omega^1_A\simeq \R f_{A*}(\mls V^{}_A\otimes f_{A}^*\rL\Omega^1_A)
\]
(recall that we assume $f$ to be flat). Combining this with the distinguished triangle
\[      
    f_{A}^*\rL\Omega^1_A\rightarrow\rL\Omega^1_{X_A}\rightarrow\rL\Omega^1_{X_A/A}\xrightarrow{+1}
\]
we get a canonical isomorphism 
\[
    \cone(\gamma)(A)\simeq\R f_{A*}(\mls V^{}_A\otimes\rL\Omega^1_{X_A/A})\simeq\rL g^{*}\left(\R f_*\left(\mls V^{}\otimes\rL\Omega^1_{X/S}\right)\right).
\]
Thus, we have $\cone(\gamma)\simeq\rH_{\bW^{}}$, where
\[
    \bW^{}=\R\bV\left(\R f_*(\mls V^{}\otimes\rL\Omega^1_{X/S})\right)\in\mls D(S).
\]
As $f$ is lci, the cotangent complex $\rL\Omega^1_{X/S}$ is perfect, so by \cite[0CTM]{stacks-project} the pushforward $\R f_*(\mls V^{}\otimes\rL\Omega^1_{X/S})$ is a perfect complex on $S$. Thus, $\bW^{}$ is a perfect complex of vector bundles on $S$, so $\bW^{}$ and therefore $\cone(\gamma)$ is stably algebraic. From the triangle~\eqref{eq:triangle of cones} we get that $\cone(\alpha)$ is stably algebraic if and only if $\cone(\beta)$ is stably algebraic. We now observe that the map $\rho$ enters into the diagram~\eqref{eq:map of dts, functor level} only in the horizontal morphisms. In particular, the middle column does not depend on $\rho$, and hence $\cone(\beta)$ is unchanged upon replacing $(\mls V^{}, \rho)$ with $(\mls V^{}, 0)$. This reduces us to the case when $\rho=0$, which was handled above in \S\ref{ssec:rho=0}. This completes the proof of Theorem \ref{T:A.2} and therefore also Theorem \ref{T:strongform}.
\end{pg}
\end{proof}

\section{Animated rings and their cohomology}\label{S:section4}

In this section we recall some derived algebraic geometry, following \cite[\S 5]{CS} (see also \cite[Appendix A]{BL2}).  We also extend the results of \ref{thm:flat cohomology is Kan extended from polynomial algebras} and \ref{thm:derived hoobler} to animated rings.

\begin{pg}{\bf Animated rings and sheaves.}
Let $R$ be a ring. We write $\Alg_R$ for the category of $R$-algebras and $\Poly_R$ for the category of finitely generated polynomial algebras over $R$.
An \emph{animated $R$-algebra} is a functor
\[
    \Poly_R^{\op}\to\Spc
\]
which preserves finite products. The \emph{$\infty$-category of animated $R$-algebras} is the full subcategory $\Alg^{\ani}_R$ of the category of functors $\Poly_R\to\Spc$ spanned by the animated $R$-algebras. As outlined in \cite[Remark A.4]{BL2}, the category $\Alg^{\ani}_R$ may be alternatively described as the $\infty$-category obtained from the category of simplicial objects in $\Alg_R$ by inverting weak equivalences. We will use both perspectives freely.
\end{pg}

\begin{example}\label{ex:ordinary and animated algebras}
    An ordinary $R$-algebra $A$ gives rise to a product preserving functor
    \[
        \Poly_R^{\op}\to\mathrm{Set},\hspace{1cm}P\mapsto \Hom_{\Alg_R}(P,A).
    \]
    Viewing a set as a discrete space, this gives a fully faithful embedding $\Alg_R\hookrightarrow\Alg^{\ani}_R$ whose essential image is the subcategory of $0$-truncated animated $R$-algebras.
\end{example}

\begin{rem}\label{rem:animated abelian groups}
    This definition of $\Alg^{\ani}_R$ is a special case of a more general construction of the \emph{animation} of a 1-category with suitable properties. In particular, replacing $\Poly_R$ in the above with the category $\F\Ab$ of free abelian groups of finite rank yields the $\infty$-category $\Ab^{\ani}$ of \emph{animated abelian groups}. By \cite[Example 5.1.6(2)]{CS}, the category $\Ab^{\ani}$ is identified under the Dold-Kan correspondence with the $\infty$-category $\mls D^{\leq 0}(\Ab)$.
\end{rem}

\begin{pg}\label{pg:push pull functors, affine case}
For a morphism of rings $i\colon R\rightarrow R'$ we have functors
\[
    i^*\colon \Alg _R\rightarrow \Alg _{R'}, \ \ A\mapsto A\otimes _RR'
\]
and
\[
    i_*\colon \Alg _{R'}\rightarrow \Alg _R, \ \ (R'\rightarrow A)\mapsto (R'\rightarrow R\rightarrow A).
\]
Here the notation reflects the geometric viewpoint of the morphism $\Spec(R')\rightarrow \Spec(R)$ defined by $i$.  As discussed in \cite{CS} these functors have animated versions \cite[5.1.7]{CS}
\begin{equation}\label{eq:i^*ani}
    i^{*\ani }\colon\Alg ^\ani _R\rightarrow \Alg ^{\ani }_{R'}, \ \ A\mapsto A\otimes _R^{\mathbf{L}}R'
\end{equation}
and
\begin{equation}\label{eq:i_*ani}
    i_*^\ani \colon \Alg ^\ani _{R'}\rightarrow \Alg ^{\ani }_R, \ \ (R'\rightarrow A)\mapsto (R'\rightarrow R\rightarrow A).
\end{equation}
\end{pg}

\begin{pg}
For an $\infty $-category $\mls D$ we can consider (see \cite[5.1.0.1]{LurieHTT}) the category
\[
    \mls P_{\mls D}(\Alg_R^{\ani,\op}):=\mathrm{Fun}(\Alg_R^{\ani},\mls D)
\]
of presheaves on $\Alg _R^\ani $ taking values in $\mls D$ (note here that since we are considering rings, rather than affine schemes, we consider covariant functors). In the case when $\mls D = \mls D(\Ab )$, the derived $\infty $-category of abelian groups, we will write simply $\mls P(\Alg_R^{\ani,\op})$ for $\mls P_{\mls D(\Ab)}(\Alg_R^{\ani,\op})$.
\end{pg}

\begin{pg}
There is a natural extension of the notion of an \'{e}tale, fppf, or fpqc cover of ordinary rings to animated rings \cite[\S5.2.3]{CS}. We say that a presheaf $\mls F\in\mls P_{\mls D}(\Alg_R^{\ani,\op})$ is a \emph{sheaf} with respect to one of these topologies if it satisfies the sheaf condition for every cover.
We write
\[
    \Shv_{\mls D}(\Alg_R^{\ani,\op})\subset\mls P_{\mls D}(\Alg_R^{\ani,\op})
\]
for the full subcategory spanned by the sheaves for the fppf topology. As before, if we are considering $\mls D= \mls D(\Ab )$ then we omit the subscript $\mls D$.
It follows from  \cite[6.2.2.7]{LurieHTT} that the inclusion
\[
    \Shv_{\mls D}(\Alg_R^{\ani,\op})\hookrightarrow \mls P_{\mls D}(\Alg_R^{\ani,op})
\]
has a left adjoint 
\[
    \sh\colon\mls P_{\mls D}(\Alg_R^{\ani,\op})\rightarrow \Shv_{\mls D}(\Alg_R^{\ani,\op}), 
\]
called \emph{sheafification}.

\begin{defn}\label{def:cohomology on animated rings}
    For a $\mls D(\Ab)$-valued presheaf $\mls F\in \mls P(\Alg_R^{\ani,\op})$ and an animated $R$-algebra $A\in\Alg_R^{\ani}$, we define the \emph{cohomology} of $\mls F$ on $A$ (with respect to the fppf topology) by
    \[
        \R\Gamma (A, \mls F):= (\sh\mls F)(A)\in\mls D(\Ab).
    \]
    For an integer $i$ we set $\rH^i(A,\mls F):=\mls H^i(\R\Gamma(A,\mls F))$.
\end{defn}

\end{pg}

\begin{pg}\label{pg:no new covers}
As described in \cite[\S 5.2.3]{CS}, ordinary $R$-algebras do not acquire any new covers when regarded as $0$-truncated animated $R$-algebras. That is, if $A$ is a $0$-truncated animated $R$-algebra and $A\to A'$ is an \'{e}tale (resp. fppf, resp. fpqc) cover, then $A'$ is also $0$-truncated, and the map $A\to A'$ is an \'{e}tale (resp. fppf, resp. fpqc) cover, in the usual sense. It follows that the restriction functor
\[
    \mls P_{\mls D}(\Alg_R^{\ani,\op})\to\mls P_{\mls D}(\Alg_R^{\op})
\]
restricts to a functor
\[
    \Shv_{\mls D}(\Alg_R^{\ani,\op})\to\Shv_{\mls D}(\Alg_R^{\op})\simeq\Shv_{\mls D}(\Spec R)
\]
on the subcategories of sheaves.

\begin{rem}
    If $\mls D$ is a stable $\infty $-category, such as $\mls D(\Ab)$, then $\mls P_{\mls D}(\Alg_R^{\ani,\op})$ and $\Shv_{\mls D}(\Alg_R^{\ani,\op})$ are again stable $\infty $-categories by \cite[1.1.3.1]{LurieHA}.
\end{rem}

\end{pg}

\begin{pg}\label{SS:3.7}{\bf Continuous sheaves and animation.}
As described in \cite[5.1.4]{CS} and \cite[Proposition A.5]{BL2}, the category $\Alg^{\ani}_R$ can be characterized as the $\infty$-category freely generated by $\Poly_R$ under sifted colimits. More precisely, the functor
\[
    \Poly_{R}\hookrightarrow \Alg _R^\ani
\]
given by the composition of the inclusion $\Poly_{R}\hookrightarrow\Alg_R$ with the embedding of Example \ref{ex:ordinary and animated algebras} is universal among functors from $\Poly_{R}$ to $\infty $-categories admitting sifted colimits, in the sense that for any such $\infty $-category $\mls D$ the restriction functor
\begin{equation}\label{E:4.8.1}
    \mathrm{Fun}_{\cts}(\Alg _R^\ani,\mls D)\rightarrow\mathrm{Fun}(\Poly_{R},\mls D)
\end{equation}
is an equivalence (here, the subscript ``$\cts$'' on the left denotes functors that commute with sifted colimits).
\end{pg}

\begin{defn}\label{D:7.13} 
Let $\mls D$ be an $\infty$-category admitting all sifted colimits. 
A functor $\mls F\colon \Alg _R^\ani\rightarrow \mls D(\Ab)$ is \emph{continuous} if it commutes with sifted colimits.
\end{defn}

\begin{rem}
    By \cite[5.5.8.17]{LurieHTT} a functor $\Alg_R^{\ani}\to\mls D(\Ab)$ is continuous if and only if it commutes with filtered colimits and geometric realizations.
\end{rem}

\begin{pg}[The animation of a functor]
\label{pg:animations}
Given a full subcategory $\mls C\subset\Alg_R^{\ani}$ containing $\Poly_R$ and a functor $\mls F\colon \mls C\to\mls D$, the \emph{animation} of $\mls F$ is the functor
    \[
        \mls F^{\ani}\colon\Alg_R^{\ani}\to\mls D
    \]
(sometimes denoted by $\rL\mls F$) obtained from $\mls F$ by taking the left Kan extension, in the sense of \cite[4.3.2.2]{LurieHTT}, of the restriction $\mls F\big\lvert_{\Poly_R}$ along the inclusion $\Poly_R\hookrightarrow\Alg_R^{\ani}$. In particular, $\mls F^{\ani}$ is continuous, and there is a canonical map $\mls F^{\ani}\to\mls F$ that restricts to an equivalence on $\Poly_R$ and is an equivalence if and only if $\mls F$ is continuous.
\end{pg}


\begin{example}
    Animating the functor $\uLOmega^1$ of \ref{ex:cotangent complex} gives a functor
    \[
        \left(\uLOmega^1\right)^{\ani}\colon\Alg_{\bF_p}^{\ani}\to\mls D(\Ab)
    \]
    whose value on an animated $\bF_p$-algebra $A$ is the animated cotangent complex $\rL\Omega^1_A$ of $A$ (as described by Bhatt--Lurie \cite[Construction B.1]{BL2}). This functor is the animation of the functor $P\mapsto\Omega^1_P$ on polynomial $\bF_p$-algebras, and is an fppf sheaf.
\end{example}

\begin{example}\label{E:4.11} 
    For a ring $R$ and a complex of $R$-modules $\mls V$ we have the functor
    \begin{equation}\label{eq:animated complex of modules}
        \rH^{\ani}_{\mls V}\colon\Alg_R^{\ani}\to\mls D(\Ab),\hspace{1cm}A\mapsto\mls V\otimes^{\rL}_RA,
    \end{equation}
    where the right side denotes the derived tensor product of $R$-modules. We claim that if $\mls V$ is a perfect complex then this functor is continuous, and therefore is canonically isomorphic to the animation of the functor $\rH_{\mls V}\colon\Alg_R\to\mls D(\Ab)$.
    Indeed, arguing as in \ref{ex:functor from a complex} we may reduce to the case when $\mls V$ is a single projective $R$-module say $P$ in degree 0. Consider the functor $\Alg_R^{\ani}\to\mls D^{\leq 0}(\Mod_R)$ obtained by composing the forgetful functor to simplicial $R$-modules with the Dold--Kan correspondence. This functor commutes with sifted colimits (see e.g. \cite[5.1.3]{CS} ff.). The inclusion $\mls D^{\leq 0}(\Mod_R)\hookrightarrow\mls D(\Mod_R)$ is left adjoint to the truncation $\tau_{\leq 0}$, hence preserves all colimits. Hence the composition $\Alg_R^{\ani}\to\mls D(\Mod_R)$ commutes with sifted colimits. The claim follows by noting that $\rH^{\ani}_{\mls V}$ is the composition of this functor with the functor $P\otimes^{\rL}\gap\colon\mls D(\Mod_R)\to\mls D(\Mod_R)$, which arguing as in \ref{ex:functor from a complex} also preserves sifted colimits. Via a similar reduction, one shows that~\eqref{eq:animated complex of modules} is also an fppf sheaf.
\end{example}

\begin{example}\label{E:C functor, animated version}
    For a ring $R$ of characteristic $p$ and a pair $(\mls V,\rho)$ where $\mls V$ is a complex of $R$-modules and $\rho \colon \rL\F_R^*\mls V\to\mls V$ is a map in $\mls D(\Mod_R)$, we define a functor
    \[
        \mls C^{\ani}_{(\mls V,\rho)}\colon\Alg_R^{\ani}\to\mls D(\Ab)
    \]
    by the same formula as in \ref{not:notation for C_G}, but using the animated tensor product in place of the derived tensor product and the animated cotangent complex in place of the ordinary cotangent complex. From the previous examples it follows that, if $\mls V$ is perfect, then $\mls C^{\ani}_{(\mls V,\rho)}$ is the animation of the functor $\mls C_{(\mls V,\rho)}$ defined in \ref{not:notation for C_G}, and is an fppf sheaf.
\end{example}

\begin{example}\label{ec:G ani}
    Let $\bG$ be a commutative affine group scheme over a ring $R$. Let $\mls O_{\bG}$ be the coordinate ring of $\bG$. Under the embedding $\Alg_R\hookrightarrow\Alg_R^{\ani}$ we regard $\mls O_{\bG}$ as a 0-truncated animated $R$-algebra. As discussed in \cite[5.1.11]{CS}, for an animated $R$-algebra $A$ the group operation in $\bG$ endows $\Hom_{\Alg_R^{\ani}}(\mls O_{\bG},A)$ with a functorial structure of an animated abelian group, which under the Dold--Kan correspondence we may regard as an object of $\mls D^{\leq 0}(\Ab)$. We write
    \[
        \bG^{\ani}(\gap):=\Hom_{\Alg_R^{\ani}}(\mls O_{\bG},\gap)\colon\Alg_R^{\ani}\to\mls D^{\leq 0}(\Ab)
    \]
    for the resulting functor. This functor may alternatively be described as the animation of the functor $\Poly_R\to\mls D^{\leq 0}(\Ab),\,P\mapsto\bG(P)$ (regarding $\bG(P)$ as a complex of abelian groups concentrated in degree 0). In particular, $\bG^{\ani}$ is continuous. By fppf descent $\bG^{\ani}$ is also an fppf sheaf. Composing $\bG^{\ani}$ with the inclusion $\mls D^{\leq 0}(\Ab)\hookrightarrow\mls D(\Ab)$ gives a functor $\Alg_R^{\ani}\to\mls D(\Ab)$. The inclusion is left adjoint to the truncation functor $\tau_{\leq 0}$, hence preserves colimits, and so this functor is also continuous. However, it need not be an fppf sheaf. We write
    \[
        \rH_{\bG^{\ani}}\colon\Alg_R^{\ani}\to\mls D(\Ab)
    \]
    for its fppf sheafification. As in Definition \ref{def:cohomology on animated rings}, our notation for the value of this functor on an animated $R$-algebra $A$ is
    \[
        \left(\rH_{\bG^{\ani}}\right)(A)=\R\Gamma(A,\bG^{\ani}).
    \]
    It follows from the above discussion that if $A$ is an animated $R$-algebra then we have
    \[
        \bG^{\ani}(A)\simeq\tau_{\leq 0}\,\R\Gamma(A,\bG^{\ani}).
    \]
    Furthermore, if $A$ is an ordinary $R$-algebra, then as noted above in \ref{pg:no new covers} the cohomology $\R\Gamma(A,\bG^{\ani})$ defined here is identified with the usual fppf cohomology $\R\Gamma(A,\bG)=\R\Gamma(\Spec A,\bG)$.
\end{example}

\begin{example}\label{ex:continuity examples}
    Let $\bG$ be a commutative affine group scheme over $R$. The following examples show in particular that the functor $\rH_{\bG^{\ani}}$ need not be continuous.
    \begin{enumerate}
        \item [(i)] The functor $\rH _{\bG_a^\ani}$ associated to the additive group $\bG_a$ is continuous. Indeed, by \cite[5.2.9]{CS} and the vanishing of coherent cohomology on affine schemes, we have that $\rH^i(A,\bG_a^{\ani})=0$ for $i>0$ for every animated $R$-algebra $A$. Thus the functor $\bG_a^{\ani}\colon\Alg_R^{\ani}\to\mls D(\Ab)$ is already an fppf sheaf, and so agrees with its fppf sheafification $\rH_{\bG_a^{\ani}}$. But $\bG_a^{\ani}$ is the same as the functor $\rH_R^{\ani}$ associated to the perfect complex $R$, which is continuous by Example \ref{E:4.11}.
        \item [(ii)] From (i) and the short exact sequence
        \[
            0\to\balpha_p\to\bG_a\xrightarrow{\F}\bG_a\to 0
        \]
        it follows that $\rH _{\balpha _p^\ani }$ is also continuous. More generally, we show in Theorem \ref{thm:coho is continuous} below that $\rH_{\bG^{\ani}}$ is continuous for any height 1 group scheme $\bG$ over $R$.
        \item [(iii)] If $\ell$ is a prime invertible in $R$, then $\rH_{\bmu_{\ell}^{\ani}}$ need not be continuous in general. For example, if $R$ is a separably closed field, then for any finitely generated polynomial algebra $P$ over $R$ we have that $\rH^i(P,\bmu_{\ell})=0$ for $i\geq 1$. However, the \'{e}tale cohomology groups of affine schemes with coefficients in $\bmu_{\ell}$ can be nonzero in higher degrees.
        \item [(iv)] For similar reasons, the functor $\rH_{\bG_m^{\ani}}$ associated to the multiplicative group $\bG_m$ is not continuous in general.
    \end{enumerate}

\end{example}

\begin{rem}
    Let $\bG$ be a commutative affine group scheme over $R$. Regarding $\bG$ as a complex of fppf sheaves concentrated in degree 0, we may form the functor $\rH_{\bG}\colon\Alg_R\to\mls D(\Ab),R'\mapsto\R\Gamma(R',\bG)$, and then take the animation to obtain a functor $\left(\rH_{\bG}\right)^{\ani}\colon\Alg_R^{\ani}\to\mls D(\Ab)$. From the isomorphism $\rH_{\bG}\simeq\rH_{\bG^{\ani}}\lvert_{\Alg_R}$ of functors $\Alg_R\to\mls D(\Ab)$ we obtain a canonical map $\left(\rH_{\bG}\right)^{\ani}\to\rH_{\bG^{\ani}}$, which is an isomorphism if and only if $\rH_{\bG^{\ani}}$ is continuous.
\end{rem}

The following extends \ref{thm:flat cohomology is Kan extended from polynomial algebras} to animated $R$-algebras.

\begin{thm}\label{thm:coho is continuous}
    For a ring $R$ of characteristic $p$ and a commutative finite locally free height 1 group scheme $\bG$ over $R$, the functor
    \[
        \rH_{\bG^{\ani}}\colon\Alg_R^{\ani}\to\mls D(\Ab),\hspace{1cm}A\mapsto \R\Gamma (A, \bG^{\ani})
    \]
    is continuous.
\end{thm}
\begin{proof}
By \cite[5.1.4]{CS} for each $n\geq 0$ there is a functor $\tau _{\leq n}\colon\Alg _R^\ani \rightarrow \Alg _R^\ani $ which is left adjoint to the inclusion of the full subcategory of $\Alg_R^{\ani}$ spanned by the $n$-truncated objects. For each $n\geq 0$ we have by \cite[5.2.8]{CS} a fiber sequence in $\mls D(\Ab)$
    \begin{equation}\label{E:sift}
        \mathrm{RHom}_R(\ell_{\bG/R},\pi _{n+1}(A)[n+1])\rightarrow \R\Gamma (\tau _{\leq n+1}A, \bG^{\ani})\rightarrow \R\Gamma (\tau _{\leq n}A, \bG^{\ani})
    \end{equation}
    where $\ell_{\bG/R}:=\rL e^*\rL\Omega_{\bG/R}$ is the \emph{co-Lie complex} of $\bG$ and $e$ is the identity section of $\bG$.
    Since $\bG$ is a local complete intersection over $R$, the complex $\ell_{\bG/R}$ has perfect amplitude in $[-M, 0]$ for some $M\geq 0$, and therefore $\mathrm{RHom}_R(\ell_{\bG/R}, \pi _{n+1}(A)[n+1])$ is concentrated in degrees $\leq M-n-1$. In particular, for each integer $i$ the map
    \[
        \rH^i(\tau _{\leq n+1}A, \bG^{\ani})\rightarrow \rH^i(\tau _{\leq n}A, \bG^{\ani})
    \]
    is an isomorphism for  $n>M-i-1$. By \cite[5.2.7]{CS} we have
    \begin{equation}\label{eq:postnikov completeness}
        \R\Gamma(A,\bG^{\ani})\simeq\holim _n\,\R\Gamma(\tau_{\leq n}A,\bG^{\ani})
    \end{equation}
    and so we conclude (see also \cite[\href{https://stacks.math.columbia.edu/tag/08U5}{Tag 08U5}]{stacks-project}) that for every $i$ there exists an integer $n_0$ such that for all $n\geq n_0$  the map
    \[
        \rH^i(A, \bG^{\ani})\rightarrow \rH^i(\tau _{\leq n}A, \bG^{\ani})
    \]
    is an isomorphism for all $A\in \Alg _R^\ani $. Let $\mls D^{\geq s}(\Ab )$ be the $\infty $-category of complexes with $\mls H^i = 0$ for $i<s$ and let $\tau ^{\geq s}\colon\mls D(\Ab )\rightarrow \mls D^{\geq s}(\Ab )$ be the functor left adjoint to the inclusion $\mls D^{\geq s}(\Ab )\hookrightarrow \mls D(\Ab )$. Then we have that for each integer $s$ there exists an integer $n_0$ such that for all $n\geq n_0$ the map
    \[
        \tau^{\geq s}\,\R\Gamma(A,\bG^{\ani})\to\tau^{\geq s}\,\R\Gamma(\tau_{\leq n}A,\bG^{\ani})
    \]
    is an isomorphism for all $A\in\Alg_R^{\ani}$.
    
    Now consider a diagram $B\colon I\rightarrow \Alg _R^\ani $ of animated $R$-algebras. Since $\tau ^{\geq s}$ is a left adjoint it commutes with colimits, so for each $s$ we have
    \[
        \tau^{\geq s}(\hocolim_{d\in I}\R\Gamma (B_d, \bG^{\ani}))\simeq \hocolim_{d\in  I}(\tau^{\geq s}\R\Gamma (B_d, \bG^{\ani}))\simeq \tau^{\geq s}(\hocolim_{d\in I}\R\Gamma (\tau _{\leq n}B_d, \bG^{\ani}))
    \]
    and
\begin{align*}
    \tau^{\geq s}(\R\Gamma (\hocolim_{d\in I}B_d, \bG^{\ani}))&\simeq \tau^{\geq s}(\R\Gamma (\tau _{\leq n}(\hocolim_{d\in I}B_d), \bG^{\ani}))\\
    & \simeq \tau ^{\geq s}(\R\Gamma (\hocolim_{d\in I}(\tau _{\leq n}B_d), \bG^{\ani}))
\end{align*}
for some $n$ depending only on $s$ and $\bG$. It follows from this that to prove that $\R\Gamma (\rule{.25cm}{0.4pt}, \bG^{\ani})$ commutes with sifted colimits in $\Alg _R^\ani $ it suffices to verify that, for every $n\geq 0$, this functor commutes with sifted colimits taking values in $n$-truncated objects. This we do by induction on $n$.

For the base case of $n=0$, we note that the inclusion $\Alg_R\hookrightarrow\Alg_R^{\ani}$ is exactly the inclusion of the full subcategory of $0$-truncated objects, and by \ref{thm:flat cohomology is Kan extended from polynomial algebras} the functor $A\mapsto\R\Gamma(A,\bG^{\ani})$ on ordinary $R$-algebras commutes with sifted colimits. The inductive step follows from a consideration of the fiber sequence \eqref{E:sift} and the fact that $\ell_{\bG/R}:=\rL e^*\rL\Omega _{\bG/R}$ is perfect.
\end{proof}

\begin{rem}\label{rem:derived hoobler sequence}
    Let $R$ be an $\mathbf{F}_p$-algebra and let $\bG$ be a commutative finite locally free height 1 group scheme over $R$. Theorem \ref{thm:coho is continuous} implies that flat cohomology $\R\Gamma(A,\bG^{\ani})$ for an animated $R$-algebra $A$ may be described exactly as in \ref{thm:derived hoobler}, with the appearing cotangent complexes of ordinary rings being replaced with the cotangent complexes for animated rings defined in Example \ref{ex:cotangent complex}.
\end{rem}

\begin{pg}{\bf Global formulation.}
We will also consider sheaves and cohomology over non-affine schemes. Let $X$ be an algebraic space. We will avoid defining the notion of an affine animated scheme over $X$, and instead give an ad-hoc definition of the categories of animated presheaves and sheaves on $X$. Let $\mls D$ be an $\infty$-category with all finite limits. We define
\[
    \mls P_{\mls D}(X^{\ani}):= \holim_{R\in\Alg_X}\mls P_{\mls D} (\Alg _{R}^{\ani, \op })\hspace{1cm}\text{and}\hspace{1cm}\Shv_{\mls D}(X^{\ani}):= \holim_{R\in\Alg_X}\Shv_{\mls D} (\Alg _{R}^{\ani, \op }).
\]
As before, in the case when $\mls D=\mls D(\Ab)$ we omit the subscript. We remark that this notation is somewhat abusive, as we have not assigned an independent meaning to $X^{\ani}$. Our intent is to suggest that we are animating the category of $X$-schemes, locally in the fppf topology. Note that if $R$ is a ring then the category $\Alg_{\Spec R}\simeq\Alg_R$ has an initial object and we have
\[
    \mls P_{\mls D}((\Spec R)^{\ani})\simeq\mls P_{\mls D}(\Alg_R^{\ani,\op})\hspace{1cm}\text{and}\hspace{1cm}\Shv_{\mls D}((\Spec R)^{\ani}) \simeq \Shv_{\mls D} (\Alg _R^{\ani, \op }).
\]
Given an affine $X$-scheme $\Spec R\to X$ we obtain maps
\[
    \mls P_{\mls D}(X^{\ani})\to\mls P_{\mls D}(\Alg_R^{\ani,\op})\hspace{1cm}\text{and}\hspace{1cm}\Shv_{\mls D}(X^{\ani})\to\Shv_{\mls D}(\Alg_R^{\ani,\op})
\]
which we denote by $\mls F\mapsto\mls F_R$. We will adopt the suggestive notation $\R\Gamma(R,\mls F):=\mls F_R(R)$.

\end{pg}
\begin{example}\label{ex:some global animated functors}
    The functors of Examples \ref{E:4.11}, \ref{E:C functor, animated version}, and \ref{ec:G ani} are all fppf sheaves, and so globalize as follows.
    \begin{enumerate}
        \item[(i)]\label{item:global animated complex} For a complex $\mls V^{}$ of quasicoherent sheaves on $X$ we let $\rH_{\mls V^{}}^{\ani}\in\Shv(X^{\ani})$ denote the animated sheaf whose image in $\Shv(\Alg_R^{\ani,\op})$ for $R\in\Alg_X$ is $\rH_{\mls V^{}_R}^{\ani}$~\eqref{eq:animated complex of modules}. This determines a functor
        \[
            \mls D(\text{Qcoh}(X))\to\Shv(X^{\ani}),\hspace{1cm}\mls V^{}\mapsto\rH_{\mls V^{}}^{\ani}.
        \]
        \item[(ii)]\label{item:global functor C} For a pair $(\mls V^{},\rho)$ where $\mls V^{}$ is a complex of quasicoherent sheaves on $X$ and $\rho \colon\rL\F_X^*\mls V^{}\to\mls V^{}$ is a map in $\mls D(X_{\et},\mls O_{\et})$, we let $\mls C_{(\mls V^{},\rho)}^{\ani}\in\Shv(X^{\ani})$ denote the animated sheaf whose image in $\Shv(\Alg_R^{\ani,\op})$ for $R\in\Alg_X$ is $\mls C_{(\mls V^{}_R,\rho_R)}^{\ani}$.
        \item[(iii)]\label{item:global animated H_G} For a commutative group scheme $\bG$ over $X$, we let $\rH_{\bG^{\ani}}\in\Shv(X^{\ani})$ denote the animated sheaf whose image in $\Shv(\Alg_R^{\ani,\op})$ is $\rH_{\bG_R^{\ani}}$. For a complex $\bG^{}$ of commutative group schemes over $X$ we similarly define $\rH_{\bG^{\ani}}\in\Shv(X^{\ani})$.
    \end{enumerate}
\end{example}

\begin{rem}\label{rem:str perf complex of vector bundles and animation}
    If $\mls V^{}$ is a strictly perfect complex of locally free sheaves on $X$, then $\rH_{\mls V^{}}^{\ani}\simeq\rH_{\bV^{\,\ani}}$, where $\bV^{}$ is the strictly perfect complex of vector bundles associated to $\mls V^{}$.
\end{rem}

\begin{rem}
	Note that we have \emph{not} defined an animated sheaf associated to a general perfect complex of commutative group schemes, only to a strictly perfect complex of commutative group schemes.
\end{rem}

\begin{pg}\label{P:4.6b}
For a morphism of algebraic spaces
\[
    f\colon X\rightarrow Y
\]
we have the functor
\[
    f^{\ani *}\colon\mls P_{\mls D} (Y^{\ani})\rightarrow \mls P_{\mls D} (X^{\ani})
\]
induced locally by the restriction maps $\mls P_{\mls D}(\Alg_{B}^{\ani,\op})\to\mls P_{\mls D}(\Alg_{A}^{\ani,\op})$, $\mls F\mapsto\mls F\circ i^{*,\ani}$ for affine $X$-schemes $\Spec A\to X$, affine $Y$-schemes $\Spec B\to Y$, and maps $i\colon B\to A$ of rings which induce a map $\Spec A\to\Spec B$ restricting $f$. This functor has a right adjoint 
\[
    f_*^\ani \colon \mls P_{\mls D}(X^{\ani})\rightarrow \mls P_{\mls D} (Y^{\ani})
\]
defined as follows. Let $\mls F\in \mls P_{\mls D}(X^{\ani})$ be a presheaf.
For an affine $Y$-scheme $\Spec(R)\rightarrow Y$ there is an induced functor
\begin{align*}
    \Alg_{X_R}\times \Alg_R^\ani &\rightarrow \mls D,\\
    (B, A)&\mapsto \mls F_B(B\otimes _R ^{\mathbf{L}}A)
\end{align*}
where $X_R:=X\times_Y\Spec R$. We recall that $\mls F_B\in\mls P_{\mls D}(\Alg_B^{\ani,\op})$ is the image of $\mls F$ under the natural restriction map, and so it makes sense to evaluate $\mls F_B$ on the animated $B$-algebra $B\otimes_R^{\mathbf{L}}A$. We set
\begin{equation}\label{E:4.6b}
    (f^\ani _*\mls F)_R(A):= \holim_{B\in \Alg_{X_R}}\mls F_B(B\otimes _R^{\mathbf{L}}A),
\end{equation}
Note that in this definition we could also restrict the homotopy limit to the category of $B\in \Alg _{X_R}$ with $\Sp (B)\rightarrow X_R$ \'etale. Varying $A$ we get a presheaf $(f^\ani _*\mls F)_R\in\mls P_{\mls D}(\Alg_R^{\ani,\op})$, and varying $R$ we obtain the object $f^{\ani}_*\mls F\in\mls P_{\mls D}(Y^{\ani})$.

These functors restrict to a pair of adjoint functors
\[
    f^{\ani *}\colon\Shv_{\mls D}(Y^{\ani})\rightarrow \Shv_{\mls D}(X^{\ani})
\]
and
\[
    f_*^\ani \colon\Shv_{\mls D}(X^{\ani})\rightarrow \Shv_{\mls D}(Y^{\ani})
\]
on the subcategories of sheaves.
\end{pg}

\begin{pg}
There are natural restriction maps
    \[
        \text{res}\colon\mls P_{\mls D}(X^{\ani})\to\mls P_{\mls D}(X)\hspace{1cm}\text{and}\hspace{1cm}\text{res}\colon\Shv_{\mls D}(X^{\ani})\to\Shv_{\mls D}(X)
    \]
induced by the restriction maps in the affine case.

\begin{lem}\label{lem:animated pushforward and ordinary pushforward}
If $f\colon X\rightarrow Y$ is a flat morphism of algebraic spaces, then the diagram
    \begin{equation}\label{eq:restriction inflation square}
        \begin{tikzcd}
            \Shv_{\mls D}(X^{\ani})\arrow{d}[swap]{f^{\ani}_*}\arrow{r}{\text{res}}&\Shv_{\mls D}(X)\arrow{d}{\R f_*}\\
            \Shv_{\mls D}(Y^{\ani})\arrow{r}{\text{res}}&\Shv_{\mls D}(Y)
        \end{tikzcd}
    \end{equation}
commutes up to homotopy (and similarly for the corresponding diagram of presheaf categories).
\end{lem}
\begin{proof}
    This follows from the definition of $f_*^\ani $ given in \ref{P:4.6b} and the observation that since $f$ is flat the derived tensor products in \eqref{E:4.6b} are equal to their ordinary tensor products if $\Sp (B)\rightarrow X_R$ is flat.
\end{proof}
\end{pg}

\begin{pg}
    For a presheaf $\mls F\in\mls P_{\mls D}(X)$ we define its \emph{animation} $\mls F^{\ani}\in\mls P_{\mls D}(X)$ by functoriality from the affine case. The resulting functor
    \[
        \mls P_{\mls D}(X)\to\mls P_{\mls D}(X^{\ani}),\hspace{1cm}\mls F\mapsto\mls F^{\ani}
    \]
    is left adjoint to restriction. We remark that in general the animation of a sheaf need not again be a sheaf.
\begin{defn}
    A presheaf $\mls F\in \mls P_{\mls D}(X^{\ani})$ is \emph{continuous} if its image in $\mls P_{\mls D}(\Alg_{R}^{\ani, \op })$ is continuous for all $R\in\Alg_X$.
\end{defn}
The following shows that continuity of a sheaf is a local property and is preserved under pushforward.

\begin{lem}\label{L:3.15}
Let $X$ be a quasi-compact separated scheme and let $f\colon Y\to Z$ be a quasi-compact separated morphism.
\begin{enumerate}
    \item Let $X = \cup _{i\in I}U_i$ be a finite open cover of $X$ by affine schemes $U_i = \Spec(R_i)$.  Then a sheaf $\mls F\in \Shv_{\mls D}(X^{\ani})$ is continuous if and only if its image in $\Shv_{\mls D}(\Alg_{R_i}^{\ani , \op })$ is continuous for each $i\in I$.
    \item If $\mls F\in\Shv_{\mls D}(Y^{\ani})$ is continuous, then the pushforward $f_*^\ani\mls F\in\Shv_{\mls D}(Z^{\ani})=\Shv_{\mls D}(Z^{\ani})$ is also continuous.
\end{enumerate}
\end{lem}
\begin{proof}
For (1) consider the associated hypercover $U_\bullet  \rightarrow X$.  Since $X$ is separated each $U_n$ is a disjoint union of affine schemes given by the $n$-fold intersections of the $U_i$.  Furthermore, since the cover is finite $U_\bullet $ is $m$-truncated for some $m$. If $u_n\colon U_n\rightarrow X$ is the projection then since $\mls F$ satisfies descent we have 
\[
    \mls F\simeq \holim_nu_{n*}^\ani\left(\mls F\big\lvert_{U_n}\right).
\]
The skeletal filtration therefore induces a finite filtration on $\mls F$ whose associated graded pieces are of the form $u_{\underline i*}^{\ani }\mls F|_{U_{\underline i}}$, where $\underline i = (i_1, \dots, i_s)$ is a multiindex and $u_{\underline i}\colon U_{\underline i} = U_{i_1}\cap \cdots \cap U_{i_s}\rightarrow X$ is the inclusion of the intersection.  Statement (1) then follows from noting that $u_{\underline i*}^{\ani }\mls F\big\lvert_{U_{\underline i}}$ is continuous since $U_{\underline i}$ includes into $U_i$ for some $i$.

For (2), we may reduce to the case when $Z=\Spec R$ is affine. Then $Y$ is quasi-compact and separated, and arguing as in (1) we can find a finite filtration on $f_*^\ani \mls F$ whose associated graded pieces are pushforwards of restrictions of $\mls F$ to affines.
\end{proof}

\end{pg}

\section{Formal cohomology}\label{sec:formal cohomology}

In this section we discuss the key technical ingredients for the proof of \ref{T:1.3}, which relate flat cohomology to a  formally completed version.  The main result used in the following sections is \ref{C:keystatement} below. The results of this section are naturally viewed in the context of a theory of compactly supported flat cohomology. We briefly discuss this in section \ref{S:compact}.

\begin{pg}\label{S:section6}{\bf The $\infty $-category of pro-objects.}
Let $\mls C$ be an $\infty $-category that admits finite limits and is idempotent complete.
Then by \cite[3.1.1 and 3.1.2]{LurieDAGXIII} the $\infty $-category of pro-objects in $\mls C$, denoted by $\Pro (\mls C)$,  can be defined as the full subcategory of the category $\mathrm{Fun}(\mls C, \Spc)^\op $   spanned by those functors which preserve finite limits.  As discussed in \cite[3.1.2]{LurieDAGXIII} this category $\Pro (\mls C)$ can also be described as
\[
    \Pro (\mls C)\simeq \Ind (\mls C^\op )^\op ,
\]
where $\Ind (\mls C^\op )$ is defined as in \cite[5.3.5.1]{LurieHTT}.  By the $\infty$-categorical Yoneda lemma \cite[5.1.3.1]{LurieHTT} the natural functor
\begin{equation}\label{eq:yoneda embedding in pro category}
    \mls C\hookrightarrow \Pro(\mls C)
\end{equation}
is fully faithful, so we often view $\mls C$ as a sub-category of $\Pro (\mls C)$. Given a full subcategory $\mls D\subset\mls C$ we obtain an embedding
\begin{equation}\label{eq:embedding of pro category 2}
    \Pro(\mls D)\hookrightarrow\Pro(\mls C).
\end{equation}

 
  \end{pg}

\begin{pg}{\bf Formal cohomology of animated sheaves.}
  
    \begin{notation}\label{not:subcategories, animated case}
        Let $X$ be an algebraic space $X$. We let $\Shvf(X^{\ani})\subset\Shv(X^{\ani})$ denote the full stable $\infty$-subcategory generated by the sheaves $w_{*}^\ani\rH$ for $w\colon W\rightarrow X$ a finite morphism and $\rH$ equal to either $\rH_{\mls V^{}}^{\ani}$ for a perfect complex $\mls V^{}$ over $W$ or to $\rH _{\bG ^\ani _W}$ for a commutative finite locally free group scheme $\bG_W$ over $W$.
\end{notation}
\end{pg}

\begin{pg}\label{P:5.2}
 Let $X$ be a noetherian algebraic space and let $i\colon Z\hookrightarrow X$ be a closed subspaces defined by an ideal $I\subset \mls O_X$.  Let
\[
    i_n\colon Z_n\hookrightarrow X
\]
be the closed subspace defined by $I^n$ so we have a sequence of closed immersions
\[
    Z = Z_1\hookrightarrow Z_2\hookrightarrow \cdots \hookrightarrow Z_n\hookrightarrow \cdots \hookrightarrow X.
\]
For an object $\mls F\in \Shv(X^{\ani})$ let $\mls F_n\in \Shv(Z_n^{\ani})$ be the restriction of $\mls F$ to $Z_n$. Define
\begin{equation}\label{E:prodef}
    \widehat {\mls F}:= \mathrm{holim}_n \ i_{n*}^\ani \mls F_n\in\Pro(\Shv (X^\ani)).
\end{equation}
We refer to $\widehat {\mls F}$ as the \emph{completion of $\mls F$ with respect to $Z$}.
\end{pg}

\begin{pg}\label{pg:affine case}
In the case when $X = \Sp (A)$ and $Z$ is defined by an ideal $I = (a_1, \dots, a_r)\subset A$ we can also consider the derived reduction
\[
    A_n^\ani := A\otimes ^{\mathbf{L}}_{a_i^n\mapsfrom X_i, \mathbf{Z}[X_1, \dots, X_r],  X_i\mapsto 0}\mathbf{Z}
\]
and the induced functor
\[
    B\mapsto \mls F(A_n^\ani \otimes _A^{\mathbf{L}}B)
\]
defining an object $\mls F_n'\in \Shv(X^{\ani})$.  Set
\[
    \widehat {\mls F}':= \mathrm{holim}_n \ \mls F_n'\in\Pro(\Shv(X^{\ani})).
\]
The following shows that these two approaches are equivalent in the main cases of interest.
\end{pg}
\begin{lem}\label{L:9.4b}
    In the situation of \ref{pg:affine case}, if $\mls F\in\Shvf(X^{\ani})$ then the natural map $\widehat {\mls F}'\rightarrow \widehat {\mls F}$ is an equivalence.
\end{lem}
\begin{proof} 
Replacing $A$ by its $I$-adic completion it suffices to prove the lemma in the case when $A$ is furthermore complete with respect to $I$. The ring $A_n^\ani $ can be described by a Koszul complex construction as in \cite[\S 3]{KST}, and therefore its homology groups are given by the homology of the Koszul complex $\mc K_n^\bullet $ associated to the sequence $(a_1^n, \dots, a_r^n)$.  The functoriality 
\[
    A_n^\ani \rightarrow A_{n-1}^\ani 
\]
is described by the functoriality of the Koszul complex using the diagram
\[
    \begin{tikzcd}
        A^r\arrow{rr}{(\cdot a_1,\ldots,\cdot a_r)}\arrow{dr}[swap]{(a_1^n,\ldots,a_r^n)}&&A^r\arrow{dl}{(a_1^{n-1},\ldots,a_n^{r-1})}\\
        &A.&
    \end{tikzcd}
\]
Now observe that by \cite[\href{https://stacks.math.columbia.edu/tag/0921}{Tag 0921}]{stacks-project} the projective systems $\{\mc K_n^\bullet \}$ and $\{A/(a_1^n, \dots, a_r^n)\}$ in $\mls D(\Mod(A))$ are isomorphic (as noted in the footnote to loc. cit. for every $n\geq 0$ there exists an integer $m\geq n$ such that the map $\mc K_m^\bullet \rightarrow \mc K_n^\bullet $ factors through $A/(a_1^m, \dots, a_r^m)$).

Now we prove the result. It suffices to consider the cases when $\mls F=w^\ani _*\rH$ with $w\colon W\rightarrow X$ a finite $X$-scheme and $\rH$ either $\rH _{\mls V }^{\ani}$ for a perfect complex $\mls V $ on $W$, or the case when $\rH =\rH_{\bG^{\ani}}$ for a commutative finite locally free $\bG$ on $W$.  Write $W = \Spec(C)$ with $C$ a finite $A$-algebra.  Chasing through the definition one finds that 
\begin{equation}\label{E:8.7.1}
\widehat {\mls F}'(B) = \holim _n\rH ((C\lotimes _AA_n^\ani )\lotimes _C(C\lotimes _AB))
\end{equation}
and
\[
    \widehat {\mls F}(B) = \holim _n\rH ((C\lotimes _AA_n)\lotimes _C(C\lotimes _AB)),
\]
with the map in question being induced by the natural map $C\lotimes _AA_n^\ani \rightarrow C\lotimes _AA_n$.  By the preceding discussion the projective system of cocones of these maps have eventually zero transition maps.

In the case when $\rH = \rH _{\mls V }^\ani $ for a perfect complex $\mls V $ on $W$ we conclude that the cone of the map of systems
\[
    \{\mls F_n'\}\rightarrow \{i_{n*}^\ani \mls F_n\}
\]
has eventually zero transition maps, and therefore is zero in the pro-category, proving our result in this case. In the case when $\rH = \rH _{\bG^\ani }$ for a commutative locally free group scheme $\bG $ on $W$ the result follows from the above discussion combined with \cite[5.1.10 (3) and 5.1.13]{CS}.
\end{proof}

\begin{lem}\label{L:8.7b} Let $f\colon Y\rightarrow X$ be an affine morphism of noetherian algebraic spaces, let $Z_X\hookrightarrow X$ be a closed subspace with preimage $Z_Y\hookrightarrow Y$. Let $\mls F\in \Shvf(Y^{\ani})$ be an object, let $\widehat {\mls F}$ be the completion of $\mls F$ with respect to $Z_Y\subset Y$, and let $\widehat {(f_*^\ani \mls F)}$ be the completion of the pushforward with respect to $Z_X$.  Then the natural map $\widehat {(f_*^\ani \mls F)}\rightarrow f_*^\ani \widehat {\mls F}$ is an isomorphism.
\end{lem}
\begin{proof}
The assertion is local on $X$ so it suffices to consider the affine case $X = \Spec(A)$ with $Z_X$ defined by an ideal $(a_1, \dots, a_n)\subset A$.  In this case it suffices to prove the analogous statement for $\widehat {\mls F}'$ instead of $\widehat {\mls F}$.  In this case the result follows from the formula \eqref{E:8.7.1} and the observation that $C\lotimes _AA_n^\ani \simeq C_n^\ani $.
\end{proof}

It will also be technically useful to have a third description of $\widehat {\mls F}$.  In the situation of \ref{pg:affine case} let $\widehat A$ be the derived $I$-adic completion of $A$ (in the sense of \cite[5.6.1]{CS}). 
\begin{lem}\label{L:8.8} For $\mls F \in \Shvf(\Spec (A)^{\ani})$
the natural map
    \[
        \R\Gamma (\widehat A, \mls F)\rightarrow \R\Gamma (A, \widehat {\mls F})
    \]
is an isomorphism.
\end{lem}
\begin{proof}
    It suffices to consider the case when $\mls F = w_*^\ani \rH$ for a finite morphism $\Sp (C)\rightarrow \Sp (A)$ and $\rH $ either $\rH _{\mls V }^\ani $ for a perfect complex $\mls V $ on $\Sp (C)$, or $\rH = \rH _{\bG ^\ani }$ for a commutative finite locally free $\bG $ on $\Sp (C)$, and $I = (a_1, \dots, a_n)$.  Now observe that since $C$ is finite over $A$ the Koszul description of $A_n$ discussed above gives that
    \[
        \widehat C = \holim _n C_n^\ani \simeq C\lotimes _A(\R \holim _nA_n) \simeq C\lotimes _A\widehat A,
    \]
    where $\widehat C$ is the derived completion of $C$ with respect to $IC$. Also using \ref{L:9.4b} we find that
    \[
        \R \Gamma (A, \widehat {\mls F})\simeq \R \Gamma (C, \widehat {\rH }).
    \]
    Replacing $C$ by $A$ we are then further reduced to the case when $C=A$.  In this case the result for $\rH _{\mls V }^\ani $ follows from \ref{L:9.4b}, and the result for $\rH _{\bG ^\ani }$ follows from \cite[5.6.6]{CS}.
\end{proof}

\begin{lem}\label{L:8.9} 
For $\mls F\in\Shvf(\Spec(A)^{\ani})$ the square
\begin{equation}\label{E:5.9.1}
    \xymatrix{
        \R\Gamma (A, \mls F)\ar[r]\ar[d]& \R\Gamma (\widehat A, {\mls F})\ar[d]\\
        \holim_{f\in I}\R\Gamma (A[1/f], \mls F)\ar[r]& \holim_{f\in I}\R\Gamma (\widehat A[1/f], \mls F)}
\end{equation}
is homotopy cartesian.
\end{lem}
\begin{proof}
As in the proof of \ref{L:8.8} if $\mls F = w_*^\ani \rH$ for a finite morphism $\Sp (C)\rightarrow \Sp (A)$ then the diagram \eqref{E:5.9.1} for $\mls F$ is canonically identified with the corresponding diagram for $\rH$ over $C$.  From this it follows that it suffices to prove the lemma for $\mls F = \rH _{\mls V }^\ani $ for a perfect complex $\mls V $ of $A$-modules and $\mls F = \rH _{\bG ^\ani }$ for a commutative finite locally free group scheme $\bG /A$. Furthermore, considering the homotopy fibers of the vertical maps, it suffices to show that the natural map on local cohomology
\begin{equation}\label{E:5.9.1b}
    \R\Gamma _I(A, \mls F)\rightarrow \R\Gamma _I(\widehat A, \mls F)
\end{equation}
is an equivalence. When $\mls F=\rH_{\mls V }^{\ani}$, the statement that \eqref{E:5.9.1b} is an equivalence is the statement that local cohomology is the same for a ring and its completion (see \cite[\href{https://stacks.math.columbia.edu/tag/0ALZ}{Tag 0ALZ}]{stacks-project}). In the case when $\mls F=\rH_{\bG^{\ani}}$, this follows from \cite[5.4.4]{CS}.
\end{proof}

\begin{pg}\label{P:Scomplexdef}{\bf The $\mls S$-complex.}
We continue with the notation of \ref{P:5.2}. For an object $\mls F\in \Shvf(X^{\ani})$ define 
\[
    \mls S_{(X,Z)}(\mls F):= \mathrm{Cocone}(\mls F\rightarrow \widehat {\mls F})\in\Pro(\Shv(X^{\ani})).
\]
If we think $Z$ is clear from context we may write simply $\mls S(\mls F)$ for $\mls S_{(X, Z)}(\mls F)$.
\end{pg}

\begin{thm}\label{T:5.8} If $g\colon \mls F\rightarrow \mls F'$ is a morphism in $\Shvf(X^{\ani})$ whose restriction to $X\setminus Z$ is an isomorphism, then the induced map
\[
    \mls S(g)\colon\mls S(\mls F)\rightarrow \mls S(\mls F')
\]
in $\Pro(\Shv(X^{\ani}))$ is an isomorphism.
\end{thm}
\begin{proof}

The result is local on $X$, so it suffices to show that, in the case when $X = \Spec(A)$ is the spectrum of a strictly henselian local ring and $Z$ is defined by
\[
    I = (a_1, \dots, a_r)\subset A,
\]
the natural map
\[
    \R\Gamma (A, \mls S(\mls F))\rightarrow \R\Gamma (A, \mls S(\mls F'))
\]
is an isomorphism.   This follows from \ref{L:8.8} and \ref{L:8.9} which give a description of both sides in terms of the restrictions to $X\setminus Z$.
\end{proof}

\begin{cor}\label{C:keystatement} Let $f\colon X\rightarrow Y$ be a finite morphism of locally noetherian algebraic spaces, and let $Z_Y\hookrightarrow Y$ be a closed subspace with preimage $Z_X\hookrightarrow X$.  Assume that the induced map $X\setminus Z_X\rightarrow Y\setminus Z_Y$ is an isomorphism.  Let $\mls F_Y\in \Shvf(Y^\ani )$ be an object with restriction $\mls F_X\in \Shvf(X^\ani )$.  Let $\widehat {\mls F}_Y$ (resp. $\widehat {\mls F}_X$) be the object defined in \eqref{E:prodef} using $Z_Y\hookrightarrow Y$ (resp. $Z_X\hookrightarrow X$).  Then the square
$$
\xymatrix{
\mls F_Y\ar[r]\ar[d]& f_*^\ani \mls F_X\ar[d]\\
\widehat {\mls F}_Y\ar[r]& f_*^\ani \widehat {\mls F}_X}
$$
is a pushout square.
\end{cor}
\begin{proof}
Since $f_*$ is a right adjoint, and therefore commutes with homotopy limits, we have from the definition of $\widehat {\mls F}_X$ that $f_*^\ani \widehat {\mls F}_X$ is isomorphic to the completion in the sense of \eqref{E:prodef} of the object $f_*\mls F_X$ on $Y$.  Therefore the cocone of the map $f_*^\ani \mls F_X\rightarrow f_*^\ani \widehat {\mls F}_X$ is identified with $\mls S_{(Y, Z)}(f_*^\ani \mls F_X)$.
The result therefore follows from \ref{T:5.8}.
\end{proof}

\section{Generic representability results for flat cohomology}\label{S:section8.5}

In this section we prove our main results on the generic representability of flat cohomology. In particular, we will prove Theorem \ref{T:1.3} from the introduction. We will deduce this from the following result for animated cohomology.

\begin{thm}\label{T:8.1}
    Let $f\colon X\to S$ be a projective morphism of noetherian schemes of characteristic $p$ with $S$ reduced. For an object $\mls F\in\Shvf(X^{\ani})$ there exists
    a dominant flat quasi-finite morphism $V\to S$ of schemes such that $(f^{\ani}_*\mls F)|_V$ is in $\Shvf(V^{\ani})$.
\end{thm}

Before giving the proof we explain how Theorem \ref{T:8.1} implies Theorem \ref{T:1.3}.

\begin{lem}\label{lem:deduction}
    If $S$ is a reduced noetherian scheme and $\mls F\in\Shvf(S^{\ani})$ is an object with associated ordinary sheaf $\mls F_{\ord}:=\res(\mls F)\in\Shv(S)$, then there exists a dense open $U\subset S$ such that the restriction $(\mls F_{\ord})|_U\in\Shv(U)$ is stably algebraic and finitely presented over $U$.
\end{lem}
\begin{proof}
    It suffices to consider the cases when $\mls F=w^{\ani}_*\rH$ for a finite morphism $w\colon W\rightarrow S$ and $\rH$ equal to either $\rH_{\mls V}^{\ani}$ for a perfect complex $\mls V$ over $W$ or to $\rH_{\bG^{\ani}}$ for a commutative finite locally free group scheme $\bG$ over $W$. Shrinking on $S$, and using that $S$ is reduced, we may further assume that $w$ is finite and flat. Now the case of a perfect complex $\mls V$ is immediate from \ref{C:6.24}. For the case of $\rH = \rH _{\bG ^\ani }$ consider the B\'{e}gueri resolution \ref{pg:begueri resolution} of $\bG$. Since $w$ is finite flat and the group schemes $\bB^0$ and $\bB^1$ are smooth, the derived pushforward $\R w_*\bG$ for the flat topology is represented by the complex $\left[w_*\bB ^0\rightarrow w_*\bB ^1\right]$, and so by \ref{lem:animated pushforward and ordinary pushforward} we have quasi-isomorphisms
    \[
        (w^{\ani}_*\rH_{\bG^{\ani}})_{\ord}\simeq\rH_{\R w_*\bG}\simeq\left[\rH_{w_*\bB^0}\to\rH_{w_*\bB^1}\right]
    \]
    in $\Shv(S)$. For $i=0,1$ the pushforward $w_*\bB^i$ is represented by a smooth affine group scheme of finite presentation over $S$ \cite[4.4]{MR1321819}, so this gives the result by \ref{C:6.24}.
\end{proof}

\begin{prop}\label{P:9.1implies}
    Theorem \ref{T:8.1} implies Theorem \ref{T:1.3}.
\end{prop}
\begin{proof}
    Let $\bG$ be a commutative finite locally free group scheme on $X$. We claim that there exists a dense open subscheme $U\subset S$ such that $(\R f_*\bG)|_U$ is stably algebraic and finitely presented over $U$. Theorem \ref{T:1.3} will then follow by applying \ref{cor:generic representability of stably alg complex}. By replacing $S$ with a dense open we may assume that $f$ is flat. By \ref{T:8.1} there exists a dominant flat quasi-finite morphism $V\to S$ of schemes such that the pullback $(f^{\ani}_*\rH_{\bG^{\ani}})|_V$ is in $\Shvf(V^{\ani})$. By \ref{lem:deduction} there exists a dense open subscheme $V'\subset V$ such that the associated ordinary sheaf
    \[
    ((f^{\ani}_*\rH_{\bG^{\ani}})|_{V'})_{\ord}\simeq((f^{\ani}_*\rH_{\bG^{\ani}})_{\ord})|_{V'}\simeq(\R f_*\bG)|_{V'}
    \]
    (using the flatness of $f$ and $V'\to S$ and \ref{lem:animated pushforward and ordinary pushforward}) is stably algebraic and finitely presented over $V'$. Let $U\subset S$ be the image of $V'$. Then $U$ is open and dense in $S$ and $V'\to U$ is an fppf cover, so by \ref{cor:stable infty category} we have that $(\R f_*\bG)|_U$ is stably algebraic and of finite presentation over $U$.
\end{proof}

\begin{pg}
In the remainder of this section we will prove Theorem \ref{T:8.1}. Our argument is a somewhat elaborate induction on the relative dimension of $f$, with inductive step involving the formation of suitably chosen alterations together with a simultaneous d\'evissage of the coefficient sheaf, with the end result of reducing to the special case when $f$ is smooth and projective and $\mls F=\rH_{\bG^{\ani}}$ for $\bG=\bmu _p$, $\balpha_p$, or $\mathbf{Z}/q$ for a prime $q$. Before starting the main part of the proof we collect together various auxiliary results that will be used.
\end{pg}

\begin{pg}\label{pg:derived blowups, local description}{\bf Derived blowups and flat cohomology.}
Let $A$ be a ring and let $a_1,\dots,a_r\in A$ be elements.
As in \cite{KST} (see also \cite{KR}) we can then define the \emph{derived blowup} 
\begin{equation}\label{E:derivedblowup}
    \pi^{\ani}\colon\Bl^\ani\rightarrow \Spec (A)
\end{equation}
by taking the fiber product, in the derived sense, of the diagram
\[
    \begin{tikzcd}
        &\Spec A\arrow{d}{(a_1,\dots,a_r)}\\
        \Bl_{0}\arrow{r}&\bA^r
    \end{tikzcd}
\]
where the bottom horizontal arrow is the ordinary blowup $\Bl_0$ of the ideal sheaf $(X_1,\dots,X_r)\subset\bZ[X_1,\dots,X_r]$ of the origin in $\bA^r=\bA^r_{\bZ}$. Rather than introducing an additional layer of derived schemes, we will work with this derived blowup in a somewhat ad hoc manner. Namely, fix a finite covering of $\Bl_0$ by affines (for example, the one induced by the restriction of the standard affine cover of projective space), so we get an fppf hypercover
\[
    U_\bullet \rightarrow \Bl_0
\]
with each $U_n$ affine.  We will, abusively, write simply $\Bl_0$ for the corresponding cosimplicial animated ring, and $\Bl^\ani$ for the  simplicial object of $\Alg _A^{\ani, \op }$ obtained by base change. Taking the coskeleton of the morphism 
\[
    \Bl^\ani\rightarrow \Spec(A)
\]
in the $\infty $-category of simplicial objects of $\Alg _A^{\ani , \op }$ we obtain an object 
\begin{equation}\label{eq:coskel blow up}
    B^{\ani}_\bullet \rightarrow \Spec (A),
\end{equation}
which formally is a bisimplicial object in $\Alg _A^{\ani, \op}$ though it most naturally should be viewed as a simplicial derived scheme.
\end{pg}

\begin{example}
Let $A$ be a ring with elements $a_1, \dots, a_r\in A$.   The derived blowup $\Bl ^\ani $ has an underlying scheme $\pi _0(\Bl ^\ani )$.  This is the scheme obtained by gluing the underlying schemes of $U_0$.  That is, if $\Bl _0 = \cup _i\Sp (B_i)$ is the open cover used in the construction then $\pi _0(\Bl ^\ani )$ is the scheme obtained by gluing from the derived tensor products $B_i\lotimes _{\mathbf{Z}[X_1, \dots, X_r]}A$.
In general it is not the case that the underlying scheme $\pi_0(\Bl^\ani)$ of the derived blowup $\Bl^\ani\rightarrow \Spec(A)$ is equal to the ordinary blowup $\Bl_I$ of $\Spec(A)$ at the ideal $I=(a_1,\dots,a_r)$. For an explicit example consider $A = k[t]$, $r=2$, and $a_1 = a_2 = t$. In this case the classical blowup of $I$ is just $\Spec(A)$ again, but $\pi _0(\Bl^\ani)$ has an extra component.  In general there is a closed immersion
\[
    \Bl_I\hookrightarrow \pi _0(\Bl^\ani)
\]
which is an isomorphism away from the closed subscheme of $\Spec(A)$ defined by $I$ \cite[4.1.11]{KR}.
\end{example}

\begin{pg}\label{P:7.7} We will also consider a more global version of the derived blowup construction.  Let $X$ be a scheme and let $\{a_i\colon \mls L_i\rightarrow \mls O_X\}_{i=1}^r$ be a collection of invertible sheaves with $\mls O_X$-linear maps to $\mls O_X$. Giving such a collection of invertible sheaves is equivalent to giving a morphism
\[
    X\rightarrow [\mathbf{A}^1/\mathbf{G}_m]^{\times r}\simeq [\mathbf{A}^r/\mathbf{G}_m^r].
\]
Since the origin in $\mathbf{A}^r$ is $\mathbf{G}_m^r$-equivariant the $\mathbf{G}_m^r$-action on $\mathbf{A}^r$ lifts to the blowup $\Bl_{0}\to\bA^r$ of the origin, and we can consider the induced diagram 
\begin{equation}\label{E:7.7.1}
    \begin{tikzcd}
        &X\arrow{d}\\
        \left[\Bl_0/\bG^r_m\right]\arrow{r}&\left[\bA^r/\bG^r_m\right].
    \end{tikzcd}
\end{equation}
The corresponding derived blowup
\[
    \pi^{\ani}\colon\Bl^{\ani}\to X
\]
is the fiber product of this diagram (in the sense of derived algebraic geometry). Locally, when the $\mls L_i$ are trivialized, this agrees with the preceding construction. Let $Z\subset X$ be the closed subscheme defined by the ideal generated by the images of the maps $a_i\colon \mls L_i\to\mls O_X$, let $U=X\setminus Z$ be the complement of $Z$ in $X$, and let
\[
    \pi\colon \Bl\to X
\]
be the ordinary blowup of $X$ along $Z$. The restriction of the map $\pi^{\ani}\colon\Bl^{\ani}\rightarrow X$ to $U$ is an isomorphism, and the closure of the resulting inclusion $U\hookrightarrow \pi _0(\Bl^{\ani})$ is the ordinary blowup $\Bl$ of $X$ along $Z$. Thus, the ordinary and derived blowups are related by a diagram
\begin{equation}\label{eq:derived blowup and ordinary blowup}
    \begin{tikzcd}
        \Bl\arrow{dr}[swap]{\pi}\arrow[hook]{r}{i}&\Bl^{\ani}\arrow{d}{\pi^{\ani}}\\
        &X.
    \end{tikzcd}
\end{equation}
\end{pg}

\begin{lem}\label{L:9.7} Suppose that $X$ is noetherian. The underlying scheme $\pi_0(\Bl^{\ani})$ of the derived blowup $\Bl^{\ani}$ can be written as a union
\begin{equation}\label{E:decomp}
    \pi _0(\Bl^{\ani}) = \Bl\cup \mathbf{P}_Z\cup T,
\end{equation}
where $\Bl$ is the ordinary blowup of $X$ along $Z$, $\mathbf{P}_Z$ is a projective bundle over $Z$, and $T$ has dimension $<\mathrm{dim}(X)$.
\end{lem}
\begin{proof}
We proceed by induction on $r$. For $r=1$ the result is immediate since the blowup map is an isomorphism. For the inductive step, assume the result holds for derived blowups associated to fewer than $r$ line bundles and maps.
As above we have an inclusion $\Bl\hookrightarrow\Bl^{\ani}$ which accounts for the first factor in \eqref{E:decomp}.  Furthermore, it is clear from the construction of $\Bl^{\ani}$ that the restriction of $\pi _0(\Bl^{\ani})$ to $Z$ is isomorphic to the projective bundle $\mathbf{P}_Z$ associated to the locally free sheaf $\mls L_1|_Z\oplus \cdots \oplus \mls L_r|_Z$, accounting for the second factor. For $i=1, \dots, r$ let $X_i\subset X$ be the closed subscheme defined by the vanishing of the image of $a_i\colon \mls L_i\rightarrow \mls O_X$.  In local coordinates when $X_i$ is affine and the $\mls L_i$ are trivial we can described the restriction of $\pi _0(\Bl^{\ani})$ to $X_i$ as follows.  Write $X_i = \Spec(R_i)$ and let $f_j\in R_i$ be a generator for the image of $a_j$.  Then 
\[
    \pi _0(\Bl^{\ani})|_{X_i}\simeq \mathrm{Proj}_{R_i}(R_i[u_1, \dots, u_r]/(u_if_j, u_jf_s-u_sf_j)_{j, s\neq i}).
\]
The closed subscheme of this scheme defined by $u_i = 0$ is equal to the underlying scheme of the derived blowup of $X_i$ associated to the collection $\{\mls L_j\rightarrow \mls O_{X_i}\}_{j\neq i}$, and this description shows that $\pi _0(\Bl^{\ani})$ is the union of this derived blowup and $\mathbf{P}_Z$. From this and induction we obtain the lemma.
\end{proof}


\begin{pg}\label{pg:derived blowups, again}
Define $Z^{\ani}$ and $E^{\ani}$ to be the derived fiber products in the diagrams
\[
    \begin{tikzcd}
        Z^{\ani}\arrow{d}\arrow{r}&X\arrow{d}\\
        \bB\bG_m^r\arrow[hook]{r}&\left[\bA^r/\bG^r_m\right]
    \end{tikzcd}
\]
and
\[
    \begin{tikzcd}
        E^{\ani}\arrow{d}\arrow{r}&\Bl^{\ani}\arrow{d}\\
        Z^{\ani}\arrow{r}&X.
    \end{tikzcd}
\]
These have the following local descriptions. Suppose that $X = \Spec(A)$ is affine and the $\mls L_i$ are trivialized, so we can view the $a_i$ as elements of $A$. In this case the morphism $X\rightarrow [\mathbf{A}^r/\mathbf{G}_m^r]$ factors through a morphism
\[
    \Spec(A)\rightarrow \mathbf{A}^r,
\]
and $Z^\ani $ is given by the derived tensor product (in the sense of animated rings)
\[
    A_Z^\ani := A\otimes ^{\mathbf{L}} _{k[X_1, \dots, X_r], X_i\mapsto 0}k.
\]
Furthermore, if
\[
    \Spec(B_\bullet )\rightarrow \mathrm{Bl}_0
\]
is an \'etale hypercover of the blowup of $\mathbf{A}^r$ at the origin, then $\Bl^{\ani}$ is described by the cosimplicial animated ring
\[
    A^{\prime }_\bullet := B_\bullet \otimes ^{\mathbf{L}} _{k[X_1, \dots, X_r]}A,
\]
and if $B_\bullet \rightarrow C_\bullet $ denotes the quotient given by the exceptional divisor then $E^\ani $ is given by
\[
    D_\bullet ^\ani := C_\bullet \otimes ^{\mathbf{L}} _{k[X_1, \dots, X_r]}A.
\]
\end{pg}

\begin{lem}\label{lem:derived blow up pushout}
    Let $\bG$ be a commutative group scheme over $X$ which is either smooth or is finite locally free and set $\mls F=\rH_{\bG^{\ani}}$. The induced diagram in $\mls D(\Ab)$
\begin{equation}\label{E:8.4.1}
    \begin{tikzcd}
        \R\Gamma(X,\mls F)\arrow{r}\arrow{d}&\R\Gamma(Z^{\ani},\mls F|_{Z^{\ani}})\arrow{d}\\
        \R\Gamma(\Bl^{\ani},\mls F|_{\Bl^{\ani}})\arrow{r}&\R\Gamma(E^{\ani},\mls F|_{E^{\ani}})
    \end{tikzcd}
\end{equation}
is a homotopy pushout diagram.
\end{lem}
\begin{proof}
    The assertion is local on $X$, so we may assume that $X=\Spec A$ is affine and the $\mls L_i$ are trivialized. Adopting the notation of \ref{pg:derived blowups, again}, we now observe that the diagram
\[
    \xymatrix{
    k[X_1, \dots, X_r]\ar[d]\ar[r]^-{X_i\mapsto 0}&k\ar[d]\\
    B_\bullet \ar[r]& C_\bullet }
\]
is homotopy cartesian and therefore the same is true for the diagram
\[
    \xymatrix{
    A\ar[d]\ar[r]& A_Z^\ani \ar[d]\\
    A'_\bullet \ar[r]& D_\bullet ^\ani .}
\]
Indeed the sequence of $k[X_1, \dots, X_r]$-modules
\[
    k[X_1, \dots, X_r]\rightarrow B_\bullet \oplus k\rightarrow C_\bullet \rightarrow k[X_1, \dots, X_r][1]
\]
is a distinguished triangle which, upon derived tensor product with $A$, induces a distinguished triangle
\[
    A\rightarrow A_Z^\ani \oplus A'_\bullet \rightarrow D_\bullet ^\ani \rightarrow A[1].
\]
It follows from this that for any affine flat group scheme $\bG$ over $A$ the induced diagram
\[
    \xymatrix{
        \tau _{\leq 0}\R\Gamma (A, \bG)\ar[r]\ar[d]& \tau _{\leq 0}\R\Gamma (A_Z^\ani, \bG)\ar[d]\\
        \tau _{\leq 0}\R\Gamma (A'_\bullet, \bG)\ar[r]& \tau _{\leq 0}\R\Gamma (D_\bullet ^\ani, \bG)}
\]
is homotopy cartesian. Taking the homotopy limit over \'etale hypercovers of $A$ and $D_\bullet ^\ani $ and  applying \cite[5.2.9]{CS} we obtain that for a smooth $\bG$ the diagram \eqref{E:8.4.1} is a homotopy pushout diagram. The case of a finite locally free group scheme then follows from this using the B\'egueri resolution.
\end{proof}


\begin{pg}{\bf Preparations for the proof of \ref{T:8.1}.} 
    \end{pg}

    We prove a series of reduction steps and special cases to be used in the proof of \ref{T:8.1}. Let $f\colon X\to S$ be a proper morphism of noetherian schemes with $S$ reduced. For a sheaf $\mls F\in\Shv(X^{\ani})$, we will say that \emph{the conclusion of \ref{T:8.1} holds for $(f\colon X\to S,\mls F)$} if there exists a dominant flat quasi-finite morphism $V\to S$ of schemes such that $(f^{\ani}_*\mls F)|_V$ is in $\Shvf(V^{\ani})$. 
    If we think that the morphism $f$ is clear from context, we may write simply $(X,\mls F)$ for $(f\colon X\to S,\mls F)$. Given a morphism $Y\to X$, we write $\mls F_Y:=\mls F|_Y$ for the restriction of $\mls F$ to $Y$.

    \begin{lem}\label{lem:observation in reductions}
        Let $f\colon X\to S$ be a proper morphism of noetherian schemes with $S$ reduced. The full subcategory of $\Shv(X^{\ani})$ spanned by those objects $\mls F$ for which the conclusion of \ref{T:8.1} holds for $(X, \mls F)$ is a full stable $\infty$-subcategory of $\Shv(X^\ani )$.
    \end{lem}
    \begin{proof}
        Note that the statement refers to a subcategory of $\Shv(X^{\ani})$, not $\Shvf(X^{\ani})$. The claim follows by noting that for a scheme $Y$ the subcategory $\Shvf(Y^{\ani})\subset\Shv(Y^{\ani})$ is a full stable $\infty$-subcategory, and that any two dominant flat quasi-finite morphisms to $S$ may be refined by a third such morphism.
    \end{proof}

    \begin{lem}\label{lem:animated coho of complex}
        Let $f\colon X\to S$ be a proper morphism of noetherian schemes with $S$ reduced. If $\mls V^{}$ is a bounded complex of coherent sheaves on $X$, then the conclusion of \ref{T:8.1} holds for $\rH_{\mls V^{}}^{\ani}$.
    \end{lem}
    \begin{proof}
    Shrinking on $S$ we may assume that $f$ is flat and that $\mls V$ is an $S$-perfect complex \cite[\href{https://stacks.math.columbia.edu/tag/0DI0}{Tag 0DI0}]{stacks-project}, in which case the cohomology $\R f_*\mls V$ is a perfect complex on $S$ whose formation commutes with arbitrary base change \cite[\href{https://stacks.math.columbia.edu/tag/0CTM}{Tag 0CTM}]{stacks-project}. By \ref{lem:animated pushforward and ordinary pushforward} we have isomorphisms
     \[
        \res(f^{\ani}_*\rH^{\ani}_{\mls V})\simeq\R f_*(\res(\rH_{\mls V}))\simeq\R f_*\rH_{\mls V}\simeq\rH_{\R f_*\mls V}.
     \]
    By \ref{E:4.11} the sheaf $\rH_{\mls V}^{\ani}$ is continuous, so by \ref{L:3.15} the cohomology $f^{\ani}_*\rH_{\mls V}^{\ani}$ is also continuous, and we obtain isomorphisms
     \[
        f^{\ani}_*\rH^{\ani}_{\mls V}\simeq\left(\res(f^{\ani}_*\rH^{\ani}_{\mls V})\right)^{\ani}\simeq\left(\rH_{\R f_*\mls V}\right)^{\ani}\simeq\rH^{\ani}_{\R f_*\mls V},
     \]
    which implies in particular that $f^{\ani}_*\rH^{\ani}_{\mls V}\in \Shvf(S^{\ani}).$
    \end{proof}

    \begin{prop}\label{prop:main rep thm, animated version}
    Let $f\colon X\to S$ be a proper morphism of noetherian schemes of characteristic $p$ with $S$ reduced and let $\bG$ be a commutative finite flat group scheme over $X$. The conclusion of \ref{T:8.1} holds for $(X,\rH_{\bG^{\ani}})$ under any of the following sets of assumptions.
    \begin{enumerate}
        \item\label{item:main rep thm animated, etale ell} $f$ is smooth and $\bG$ is \'{e}tale.
        \item\label{item:main rep thm animated, coheight 1} $\bG$ has coheight 1.
        \item\label{item:main rep thm animated, height 1} $f$ is lci and $\bG$ has height 1.
\end{enumerate}
\end{prop}
\begin{proof}
    We will show that under any of these assumptions there exists a dense open subscheme $U\subset S$ such that $(f^{\ani}_*\rH_{\bG^{\ani}})|_U$ is in $\Shvf(U^{\ani})$ (note that this is in fact slightly stronger than the conclusion of \ref{T:8.1}). By shrinking on $S$ we may assume that $f$ is flat.
    
    We argue by reducing to the non-animated results of Theorems \ref{T:strongform} and \ref{T:A.2}. The result in case~\eqref{item:main rep thm animated, coheight 1} follows from \ref{lem:animated coho of complex} and the Artin--Mazur resolution. We consider case~\eqref{item:main rep thm animated, etale ell}. Suppose $\bG$ is \'{e}tale. For a ring $R$ we have the functor
    \[
        \pi_{0}^*\colon \mls P(\Alg_R^{\op})\to\mls P(\Alg_R^{\ani,\op})
    \]
    which sends a presheaf $\mls F$ to the presheaf $A\mapsto\mls F(\pi_0(A))$. This restricts to a functor on the subcategories of sheaves, and then globalizes to give a functor
    \[
        \pi_{0}^*\colon \Shv(X)\to\Shv(X^{\ani}).
    \]
    It follows from the fiber sequences~\eqref{E:sift} and Postnikov completeness~\eqref{eq:postnikov completeness} that this has the property that
    \[
        \rH_{\bG^{\ani}}\simeq\pi_0^*\left(\text{res}\left(\rH_{\bG^{\ani}}\right)\right)\simeq\pi^*_0\left(\rH_{\bG}\right)
    \]
    via the canonical map. Arguing as in the proof of Lemma \ref{L:3.15} the same is true for the cohomology $f^{\ani}_*\rH_{\bG^{\ani}}$. Thus using \ref{lem:animated pushforward and ordinary pushforward} we obtain isomorphisms
    \[
        f^{\ani}_*\rH_{\bG^{\ani}}\simeq\pi_0^*\left(\text{res}\left(f^{\ani}_*\rH_{\bG^{\ani}}\right)\right)\simeq\pi_0^*\left(\R f_*\rH_{\bG}\right)\simeq\pi^*_0\left(\rH_{\R f_*\bG}\right).
    \]
    Using case~\eqref{item:main rep thm animated, coheight 1} and arguing as in the proof of \ref{T:strongform}~\eqref{item:coho etale ell}, we may assume that $\bG$ has constant order coprime to $p$. Arguing as before we deduce that the cohomology $\R f_*\bG$ is in the stable $\infty$-subcategory of $\mls D(S)$ generated by finite \'{e}tale group schemes. Thus $f^{\ani}_*\rH_{\bG^{\ani}}$ is in the stable $\infty$-subcategory of $\Shv(S^{\ani})$ generated by the sheaves $\pi^*_0\left(\rH_{\bA}\right)$ for finite \'{e}tale group schemes $\bA$ on $S$. But as shown above, for any such $\bA$ we have $\pi^*_0\left(\rH_{\bA}\right)\simeq\rH_{\bA^{\ani}}$, which gives the result in case~\eqref{item:main rep thm animated, etale ell}.
    
    Finally, suppose that~\eqref{item:main rep thm animated, height 1} holds, so $\bG$ has height 1. By \ref{thm:coho is continuous} the functor $\rH_{\bG^{\ani}}$ is continuous, so by \ref{L:3.15} the cohomology $f^{\ani}_*\rH_{\bG^{\ani}}$ is also continuous. Applying \ref{lem:animated pushforward and ordinary pushforward} we obtain isomorphisms
    \[
        f^{\ani}_*\rH_{\bG^{\ani}}\simeq\left(\res(f^{\ani}_*\rH_{\bG^{\ani}})\right)^{\ani}\simeq\left(\R f_*\rH_{\bG}\right)^{\ani}\simeq\left(\rH_{\R f_*\bG}\right)^{\ani}.
    \]
    Consider the functor
    \[
        \sha\colon\Shv(S)\to\Shv(S^{\ani}),\hspace{1cm}\mls G\mapsto\sh\left(\mls G^{\ani}\right)
    \]
    that first animates and then sheafifies. As $f^{\ani}_*\rH_{\bG^{\ani}}$ is a sheaf, we have from the above that
    \[
        f^{\ani}_{*}\rH_{\bG^{\ani}}\simeq\sha\left(\rH_{\R f_*\bG}\right).
    \]
    By \ref{T:A.2} the ordinary flat cohomology $\rH_{\R f_*\bG}$ is contained in the full stable $\infty$-subcategory of $\Shv(S)$ generated by the sheaves $\rH_{\mls W}$ for perfect complexes $\mls W$ on $S$ and $\mls C_{(\mls W,\sigma)}$ for pairs $(\mls W,\sigma)$ on $S$. For a perfect complex $\mls W$ we have that $\rH_{\mls W}^{\ani}$ is already a sheaf, and thus we have
    \[
        \sha(\rH_{\mls W})=\sh\left((\rH_{\mls W})^{\ani}\right)\simeq\sh(\rH_{\mls W}^{\ani})\simeq\rH^{\ani}_{\mls W}.
    \]
    Similarly, for any pair $(\mls W,\sigma)$ we have by \ref{E:C functor, animated version} that 
    \[
        \sha\left(\mls C_{(\mls W,\sigma)}\right)\simeq\sh\left((\mls C_{(\mls W,\sigma)})^{\ani}\right)\simeq\sh\left(\mls C^{\ani}_{(\mls W,\sigma)}\right)\simeq\mls C^{\ani}_{(\mls W,\sigma)}.
    \]
    As $\sha$ is exact, this implies that $f^{\ani}_*\rH_{\bG^{\ani}}\simeq\sha\left(\rH_{\R f_*\bG}\right)$ is contained in the full stable $\infty$-subcategory of $\Shv(S^{\ani})$ generated by the sheaves $\rH_{\mls W}^{\ani}$ for perfect complexes $\mls W$ on $S$ and the sheaves $\mls C_{(\mls W,\sigma)}^{\ani}$ for pairs $(\mls W,\sigma)$ on $S$. As before, for any such $\mls W$ we may find a Zariski cover $S'\to S$ over which $\rH_{\mls W}^{\ani}$ is in the subcategory generated by the sheaves $\rH_{\bV^{\ani}}$ for vector groups $\bV$ on $S'$. Similarly, for a pair $(\mls W,\sigma)$ we may find a Zariski cover $S'\to S$ over which $\mls W$ is strictly perfect and $\sigma$ is induced by a map of strict complexes. Arguing as in \ref{ssec:X=S} and using Theorem \ref{thm:main kan theorem, over a stack}, we obtain that the restriction of $\mls C_{(\mls W,\sigma)}^{\ani}$ to $S'$ is in the subcategory generated by the sheaves $\rH_{\bA^{\ani}}$ for finite locally free height 1 group schemes $\bA$ on $S'$. Moreover, we may find a Zariski cover which has these properties simultaneously for any finite collection of such objects, so this gives the result in case~\eqref{item:main rep thm animated, height 1}.
  \end{proof}

Following \cite[5.1.11]{CS} let $S = \Sp (R)$ be an affine noetherian scheme and let $\mls F\in \Shv (S^\ani )$ be a sheaf. We can then consider the functor $T_{\mls F}$
on the $\infty $-category of animated $R$-modules taking values in animated abelian groups given by 
\[
    M\mapsto \text{Fib}(\mls F(R[M])\rightarrow \mls F(R)),
\]
where $R[M]$ denotes the animated ring of dual numbers on $M$ (see \cite[5.1.8]{CS}).  If $\mls F$ commutes with limits then by \cite[5.5.2.7]{LurieHTT} $T_{\mls F}$ is corepresented by an animated $R$-module $\rL \Omega ^1_{\mls F}$ (this is a generalization of the construction of the cotangent complex - hence the notation). Consider the following conditions on $\mls F$:
\begin{enumerate}
    \item[(a)] For every animated $R$-algebra $A$ with associated truncations $A_n := \tau _{\leq n}A$ the natural map
    \[
        \mls F(A)\rightarrow \holim_n\mls F(A_n)
    \]
    is an equivalence.
    \item[(b)] The functor $T_{\mls F}$ commutes with limits and $\rL \Omega ^1_{\mls F}$ is a perfect complex.
    \item[(c)] For every square-zero extension of animated $R$-algebras
    \[
        \xymatrix{
        A'\ar[r]\ar[d]& A\ar[d]^-i\\
        A\ar[r]^-s& A[M[1]]}
    \]
    in the sense of \cite[5.1.9]{CS} the induced diagram
    \[
        \xymatrix{
        \mls F(A')\ar[r]\ar[d]& \mls F(A)\ar[d]^-i\\
        \mls F (A)\ar[r]^-s& \mls F (A[M[1]])}
    \]
    is a fiber diagram.
\end{enumerate}

\begin{lem}\label{lem:001}
Let $S$ be a reduced noetherian scheme. For any object $\mls F\in \Shv ^\circ (S^\ani )$, there exists a 
dense affine open $\Sp (R)\subset S$ such that $\mls F|_{R}\in \Shv (\Sp (R)^\ani)$ satisfies conditions (a), (b), and (c).

\end{lem}
\begin{proof} 
    It suffices to consider the case when $\mls F = w^\ani _*\rH $ with $w\colon W\rightarrow S$ a finite morphism and $\rH$ equal to either $\rH _{\bG ^\ani _W}$ for a finite locally free group scheme $\bG _W$ over $W$ or $\rH ^\ani _{\mls V}$ for a perfect complex $\mls V$ on $W$. Shrinking on $S$ we may assume that $w$ is finite and flat, and hence reduce to the case when $W=S$. The result follows from \cite[5.2.6 and 5.2.8]{CS}.
\end{proof}

\begin{lem}\label{L:9.20b} Let $f\colon X\rightarrow S$ be a proper morphism of noetherian schemes with $S$ reduced. For any object $\mls F\in \Shv ^\circ (X^\ani )$, there exists a dense affine open $\Sp (R)\subset S$ such that the restriction $(f_*^\ani \mls F)_R\in \Shv (\Sp (R)^\ani )$ satisfies conditions (a), (b), and (c).
\end{lem}
\begin{proof}
    As in the previous proof it suffices to consider the case when $\mls F = \rH _{\bG ^\ani }$ for a finite locally free group scheme $\bG /X$ or $\rH ^\ani _{\mls V}$ for a perfect complex $\mls V$ on $X$.  Furthermore, by shrinking on $S$ we may assume that $f$ is flat and that $S = \Sp (R)$ is affine. Let $A$ be an animated $R$-algebra and let $\Sp (B_\bullet )$ be an fppf hypercover of $X$ by affine schemes.  Then (note that each $B_n$ is flat over $R$)
    \[
        (f_*^\ani \mls F)(A) = \holim _n\mls F(A\otimes _RB_n),
    \]
    and similarly for the truncations of $A$.  Condition (a) for $\mls F = \rH _{\bG ^\ani }$ therefore follows from \cite[5.2.6]{CS} and (a) for a perfect complex is immediate.  

    From this and the corresponding result in the affine case, which follows from \cite[5.2.8]{CS} in the case of a finite locally free group scheme and is immediate in the case of a perfect complex, we get (c).  Furthermore, this result in the affine case shows that for $\mls F = \rH _{\bG ^\ani }$ we have $T_{f_*^\ani \mls F}$ corepresented by the dual of the perfect complex $\R f_*\ell_{\bG/X}$, and that in the case of $\mls F = \rH ^\ani _{\mls V}$ the functor $T_{f_*^\ani \mls F}$ is corepresented by the dual of the perfect complex $\R f_*\mls V$.  It follows that condition (b) holds in both cases.
\end{proof}

\begin{lem}\label{L:9.19b}
    Let $S=\Spec R$ be an affine noetherian scheme, let $\mls F,\mls G\in\Shv(S^{\ani})$ be two sheaves each of which satisfies conditions (a), (b), and (c), and let $u\colon\mls F\to\mls G$ be a morphism. Then $u$ is an isomorphism if and only if $u$ restricts to an isomorphism on $\Alg_R$.
\end{lem}
\begin{proof}
    We first consider the map $T_u\colon T_{\mls F}\rightarrow T_{\mls G}$ induced by $u$, which by the Yoneda lemma is induced by a map $u^*\colon\rL\Omega ^1_{\mls G}\rightarrow \rL \Omega ^1_{\mls F}$. The dual of this map is the value of $T_u$ on the ring of dual numbers $R[\varepsilon]:=R[\varepsilon]/\varepsilon^2$.  Since this is an ordinary algebra our assumption that $u$ is an isomorphism on $\Alg_R$ implies that the dual of $u^*$ is an equivalence. By (b) the source and target of $u^*$ are both perfect complexes, so $u^*$ is also an equivalence.
    
    We now prove the result. To show that the value of $u$ on an animated $R$-algebra $A$ is an equivalence it suffices by (a) to consider the case when $A$ is $n$-truncated for some $n\geq 0$. We proceed by induction on $n$. The base case of $n=0$ follows from (b). For the induction step we will show that if $u$ is an equivalence on the truncation $A_n:=\tau_{\leq n}A$ then $u$ is also an equivalence on $A_{n+1}$. For this we note that given a square-zero extension as in (c) we have a morphism of fiber sequences
    \[
    \begin{tikzcd}
        \mls F(A')\arrow{d}{u_{A'}}\arrow{r}&\mls F(A)\arrow{d}{u_A}\arrow{r}&\R\Hom_R(\rL\Omega^1_{\mls F},M)\arrow{d}{u^*}\\
        \mls G(A')\arrow{r}&\mls G(A)\arrow{r}&\R\Hom_R(\rL\Omega^1_{\mls G},M).
    \end{tikzcd}
    \]
    Since $u^*$ is an equivalence we conclude that $u_A$ is an equivalence if and only if $u_{A'}$ is an equivalence. Applying this in particular to the square-zero extension $A_{n+1}\rightarrow A_n$ we conclude that if $\mls F(A_n)\rightarrow \mls G(A_n)$ is an equivalence then so is $\mls F(A_{n+1})\rightarrow \mls G(A_{n+1})$.
\end{proof}

    \begin{cor}\label{C:9.21}
        Let $f\colon X\rightarrow S=\Sp (R)$ be a proper morphism of noetherian schemes with $S$ reduced and affine, and let $\mls F\in \Shv ^\circ (X^\ani )$ be an object. Suppose that $f_*^\ani \mls F$ admits a morphism $f_*^\ani \mls F\rightarrow \mls U$ to an object $\mls U\in \Shv ^\circ (S^\ani )$ which induces an isomorphism when restricted to $\Alg _R$.  Then the conclusion of \eqref{T:8.1} holds for $(X, \mls F)$.
    \end{cor}
    \begin{proof}
        This follows from combining Lemmas \ref{L:9.20b} and \ref{L:9.19b}.
    \end{proof}

\begin{lem}\label{lem:reduction reduction}
    Let $f\colon X\to S$ be a proper morphism of noetherian schemes with $S$ reduced. Let $\bG$ be a commutative finite flat group scheme over $X$ and set $\mls G:=\rH_{\bG^{\ani}}$. The conclusion of \ref{T:8.1} holds for $(X,\mls G)$ if and only if it holds for $(X_{\red},\mls G|_{X_{\red}})$, where $X_{\red}\subset X$ is the reduction.  
\end{lem}
\begin{proof}
    Let $\mls I\subset \mls O_X$ be the ideal sheaf of nilpotent elements. For $n\geq 1$ set $X_n:=V(\mls I^{n})$, so that we have a sequence of closed subschemes
    \[
        X_{\red} = X_1\subset X_2\subset \cdots \subset X_N = X,
    \]
    and let $a_n\colon X_{n}\hookrightarrow X$ denote the inclusion. Set $\bG_n:=\bG|_{X_n}$ and $\mls G_n:=\rH_{\bG_n^{\ani}}$. To prove the lemma it suffices by \ref{lem:observation in reductions} to show that the conclusion of \ref{T:8.1} holds for the cocone of the natural map $\mls G\to a^{\ani}_{1*}\mls G_1$.
   For this we proceed by induction on $n\geq 1$ such that $\mls I^n=0$. The base case of $n=1$ is trivial, and the induction step follows from \cite[5.2.8]{CS}, which gives a fiber sequence
    \[
       a_{n+1*}^{\ani}\mls G_{n+1}\rightarrow a_{n*}^\ani\mls G_n\rightarrow \rH^{\ani}_{\mls E}
    \]
    where $\mls E:=\RHom (\ell_{\bG/X}, \mls I^n/\mls I^{n+1})$, and noting that by \ref{lem:animated coho of complex} the conclusion of \ref{T:8.1} holds for $\rH^{\ani}_{\mls E}$.
\end{proof}

\begin{prop}\label{prop:reduced subscheme lemma}
Fix a reduced noetherian scheme $S$. Then the conclusion of \ref{T:8.1} holds for all pairs $(f\colon X\to S,\mls F)$ with $f\colon X\to S$ proper and $\mls F\in\Shvf(X^{\ani})$ if and only if the conclusion of \ref{T:8.1} holds for all such pairs with $X$ furthermore assumed reduced.
\end{prop}
\begin{proof}
    The ``only if'' direction is immediate. For the ``if'' direction assume that the conclusion of \ref{T:8.1} holds for all pairs $(Y, \mls G)$ with $Y$ reduced and proper over $S$ and $\mls G\in\Shvf(Y^{\ani})$. Let $X$ be a noetherian scheme (possibly non reduced) and let $f\colon X\to S$ be a proper morphism. By \ref{lem:observation in reductions} to prove that the conclusion of \ref{T:8.1} holds for $(X,\mls F)$ for all $\mls F\in \Shvf(X^{\ani})$ it suffices to consider the case when $\mls F = w_*^\ani \rH$, with $w\colon W\rightarrow X$ a finite morphism and $\rH$ equal to either $\rH_{\mls V}^\ani $ for a perfect complex $\mls V$ on $W$ or to $\rH _{\bG^\ani}$ for a finite flat group scheme $\bG$ over $W$. Shrinking on $S$ if necessary so that $W$ is flat over $S$ and replacing $X$ in the statement by $W$ we are further reduced to the case when $W = X$. In this case the statement in the case of a perfect complex follows from \ref{lem:animated coho of complex}, while in the latter case the statement follows from \ref{lem:reduction reduction}.
\end{proof}

\begin{prop}\label{P:9.5b} Let $f\colon X\rightarrow S$ be a proper morphism of noetherian schemes with $X$ and $S$ reduced and let $\bG_X$ be a finite flat group scheme over $X$. Let 
\[
    X = X_1\cup X_2
\]
be a decomposition of $X$ into the union of two reduced closed subschemes such that there exists an open subset $U\subset X$ for which the intersections $U\cap X_i\subset X_i$, $i=1,2$, are dense. Let $Z$ denote the intersection $X_1\cap X_2$ with the reduced structure. Let
\[
    a_Z\colon Z\hookrightarrow X,\hspace{1cm}a_{X_i}\colon X_i\hookrightarrow X \ \ (i=1,2)
\]
be the inclusions, write $\mls G_X:=\rH_{\bG^{\ani}}$, and for a closed subscheme $Y\subset X$ write $\mls G_Y:=\mls G_X|_Y$. If the conclusion of \ref{T:8.1} holds for $(X_1, \mls G_{X_1})$ and $(Z, \mls G_Z)$, then the conclusion of \ref{T:8.1} holds for $(X, \mls G_X)$ if and only if it holds for $(X_2, \mls G_{X_2})$.
\end{prop}
\begin{proof}
For $i=1,2$ let $r_i\colon\mls G_X\to a^{\ani}_{X_i*}\mls G_{X_i}$ be the maps induced by adjunction and set
\[
    \varphi:=(r_1,r_2)\colon\mls G_X\to a^{\ani}_{X_1*}\mls G_{X_1}\oplus  a^{\ani}_{X_2*}\mls G_{X_2}.
\]
We claim that the conclusion of \ref{T:8.1} holds for $\cone(\varphi)$. This will imply the result by \ref{lem:observation in reductions}.

Let $\mls I\subset \mls O_X$ be the ideal defining $Z$. For $n\geq 1$ let $Z_n$ denote the closed subscheme of $X$ defined by $\mls I^n$ and let $a_{Z_n}\colon Z_n\hookrightarrow X$ denote the inclusion. For $n\geq 1$ and $i=1,2$ let  $Z_{i,n}$ denote the closed subscheme of $X_i$ defined by the image of $\mls I^n$ in $\mls O_{X_i}$ and let $a_{Z_{i,n}}\colon Z_{i,n}\hookrightarrow X$ denote the inclusion. For each $n\geq 1$ we obtain a diagram
\begin{equation}\label{E:9.21.1}
    \begin{tikzcd}
        \mls G_X\arrow{d}[swap]{\varphi}\arrow{r}&\widehat{\mls G}_X\arrow{d}[swap]{\widehat{\varphi}}\arrow{r}&a^{\ani}_{Z_n*}\mls G_{Z_n}\arrow{d}[swap]{\varphi_n}\\
        a_{X_{1}*}^\ani \mls G_{X_1}\oplus a_{X_{2}*}^\ani \mls G_{X_2}\arrow{r}&a_{X_{1}*}^\ani \widehat{\mls G}_{X_1}\oplus a_{X_{2}*}^\ani \widehat{\mls G}_{X_2}\arrow{r}&a_{Z_{1,n}*}^\ani \mls G_{Z_{1, n}}\oplus  a_{Z_{2,n}*}^\ani \mls G_{Z_{2, n}}
    \end{tikzcd}
\end{equation}
where the completions are with respect to $Z$. This diagram induces maps
\begin{equation}\label{eq:two maps on cones}
    \cone(\varphi)\iso\cone(\widehat{\varphi})\to\cone(\varphi_n),
\end{equation}
where the first map is an isomorphism because \ref{C:keystatement} applied to the finite morphism $X_1\coprod X_2\rightarrow X$ implies that the left square is a homotopy pushout square. Define $Q$ to be the quotient in the short exact sequence
\[
    0\to\mls O_X\to a_{X_1*}\mls O_{X_1}\oplus a_{X_2*}\mls O_{X_2}\to Q\to 0.
\]
Note that the left map is injective because $X$ is reduced. Since $Q$ is coherent and supported on $Z$ there exists an integer $N>0$ such that $\mls I^NQ = 0$.

\begin{claim}\label{claim:9.21}
For any $n\geq N$ there exists a dense open affine subset $\Sp (R)\subset S$ such that 
    the restriction of the square
    \begin{equation}\label{E:9.21.2}
    \begin{tikzcd}
        f_*^\ani \widehat{\mls G}_X\arrow{d}[swap]{\widehat{\varphi}}\arrow{r}&f_*^\ani(a^{\ani}_{Z_n*}\mls G_{Z_n})\arrow{d}[swap]{\varphi_n}\\
      f_*^\ani( a_{X_{1}*}^\ani \widehat{\mls G}_{X_1})\oplus f_*^\ani (a_{X_{2}*}^\ani \widehat{\mls G}_{X_2})\arrow{r}&f_*^\ani (a_{Z_{1,n}*}^\ani \mls G_{Z_{1, n}})\oplus f^{\ani}_*(a_{Z_{2,n}*}^\ani \mls G_{Z_{2, n}})
    \end{tikzcd}
    \end{equation}
        (obtained by applying $f^{\ani}_*$ to the right square of~\eqref{E:9.21.1}) to $\Alg _R$ is homotopy cartesian.
\end{claim}

Assuming the claim we obtain the proposition as follows. Applying $f^{\ani}_*$ to~\eqref{eq:two maps on cones} we obtain morphisms
\[
    f^{\ani}_*\cone(\varphi)\iso f^{\ani}_*\cone(\widehat{\varphi})\to f^{\ani}_*\cone(\varphi_n).
\]
By \ref{lem:reduction reduction} and our assumption that the conclusion of \ref{T:8.1} holds for $(Z,\mls G_Z)$, we have that the conclusion of \ref{T:8.1} also holds for $f^{\ani}_*\cone(\varphi_n)$ for all $n$. By the claim, if $n\geq N$ then after replacing $S$ by a dense affine open subscheme $\Spec R\subset S$ the right map restricts to an isomorphism on $\Alg_R$. Thus by \ref{C:9.21} the conclusion of \ref{T:8.1} holds for $\cone(\varphi)$, and the lemma follows.

We now prove Claim \ref{claim:9.21}. Shrinking on $S$ we may assume that $X$, $Z_n$, $Z_{i, n}$, and $Q$ are all flat over $S$ and that $S=\Spec R$ is affine. By the definition of $f^{\ani}_*$~\eqref{E:4.6b}, for a sheaf $\mls F\in\Shv(X^{\ani})$ the value of $f^{\ani}_*\mls F$ on an ordinary $R$-algebra $A$ is given by
\[
    (f^{\ani}_*\mls F)(A)=\holim_{B\in\Alg_{X_{A}}}\mls F_B(B),
\]
where $X_A=X\otimes_RA$ and $\mls F_B\in\Shv(\Alg_B^{\ani,\op})$ is the image of $\mls F$ under the natural restriction map. Moreover, we may take this homotopy limit only over those $B\in\Alg_{X_A}$ that are flat over $X_A$, hence flat over $A$. So, let $A$ be an ordinary $R$-algebra, fix $B\in\Alg_{X_A}$ flat over $R$, write $\Spec(B)\times_XX_i\simeq\Spec (B_i)$, let $B_n$ (resp. $B_{i,n}$) denote the reduction of $B$ (resp. $B_i$) modulo $\mls I^n$, and let $\widehat{B}$ (resp. $\widehat{B}_i$) denote the $\mls I$-adic completion of $B$ (resp. $B_i$). Using \ref{L:8.8}, to show that~\eqref{E:9.21.2} is homotopy cartesian it will suffice to show that the square
\[
    \begin{tikzcd}
        \R\Gamma(\widehat{B},\bG_X)\arrow{r}\arrow{d}&\R\Gamma(B_n,\bG_X)\arrow{d}\\
        \R\Gamma(\widehat{B}_1,\bG_{X})\oplus\R\Gamma(\widehat{B}_2,\bG_{X})\arrow{r}&\R\Gamma(B_{1,n},\bG_X)\oplus\R\Gamma(B_{2,n},\bG_X)
    \end{tikzcd}
\]
is homotopy cartesian. This is a statement about flat cohomology. However, we can reduce it to a question about \'etale cohomology using the B\'egueri resolution \ref{pg:begueri resolution}
\[
    0\rightarrow \bG_X\rightarrow \bB^0_X\rightarrow \bB^1_X\rightarrow 0
\]
for $\bG_X$. To prove the claim it suffices to show that for $j=0, 1$ the above square with $\bG$ replaced by $\bB^j$ is homotopy cartesian.
Since pushforward along a closed immersion is exact for the \'etale topology \cite[\href{https://stacks.math.columbia.edu/tag/04CA}{Tag 04CA}]{stacks-project}, for this in turn it suffices to show that for $j=0,1$ the 
 diagram
\begin{equation}\label{eq:square for B}
    \begin{tikzcd}
        \bB^j_X(\widehat{B})\arrow[two heads]{r}\arrow{d}&\bB^j_X(B_n)\arrow{d}\\
        \bB^j_X(\widehat{B}_1)\oplus\bB^j_X(\widehat{B}_2)\arrow[two heads]{r}&\bB^j_X(B_{1,n})\oplus\bB^j_X(B_{2,n})
    \end{tikzcd}
\end{equation}
is cartesian, where the horizontal morphisms are surjective since $\bB^j_X$ is smooth. To see this note that considering the exact sequence (note that the exactness follows from the flatness of $Q$ over $R$)
\[
    0\rightarrow B\rightarrow B_1\oplus  B_2\rightarrow Q\otimes _{\mls O_X}B\rightarrow 0
\]
we see that a pair 
\[
(\alpha _1, \alpha _2)\in \bB^j_X(\widehat B_1)\oplus \bB^j_X(\widehat B_2)
\]
lies in the image of $\bB^j_X(\widehat B)$ if and only if the composition
\[
    \xymatrix{
        \mls O_{\bB^j_{X}}\ar[r]^-{(\alpha _1, \alpha _2)}& \widehat {B}_1\times \widehat {B}_2\ar[r] & Q\otimes _{\mls O_X}B}
\]
is zero.  Now since $\mls I^NQ = 0$ this map factors through $B_{1, n}\times B_{2, n}$, and therefore vanishes if the image of $(\alpha _1, \alpha _2)$ in $\bB^j_X(B_{1, n})\oplus \bB^j_X(B_{2, n})$ is in the image of $\bB^j_X(B_n)$. Thus for $n\geq N$ the square~\eqref{eq:square for B} is cartesian.
\end{proof}

\begin{lem}\label{lem:schematically dense} Let $f\colon X\rightarrow \Sp (R)$ be a finite type morphism of noetherian schemes with $R$ reduced, and let $j\colon U\hookrightarrow X$ be a schematically dense open subset of $X$.  Then there exists a dense affine open subscheme $\Sp (R')\subset \Sp (R)$ such that for any $R'$-algebra $A$ the base change $j_A\colon U_A\hookrightarrow X_A$ is schematically dense.
\end{lem}
\begin{proof}
    This is essentially contained in the proof of \cite[\href{https://stacks.math.columbia.edu/tag/0573}{Tag 0573}]{stacks-project}.  To prove the lemma it suffices to consider the case when $X = \Sp (B)$ is affine.  Since $B$ is assumed noetherian the ideal $I\subset B$ defining $Z$ is finitely generated, say $I = (g_1, \dots, g_m)$.  By \cite[\href{https://stacks.math.columbia.edu/tag/0565}{Tag 0565}]{stacks-project} if $A$ is an $R$-algebra then $U_A\hookrightarrow \Sp (B\otimes _RA)$ is schematically dense if and only if the base change of the map
    \[
        \begin{tikzcd}
            B\arrow{r}{(g_1,\ldots,g_m)}&B^{\oplus m}
        \end{tikzcd}
    \]
    to $A$ is injective. Let $M$ denote the cokernel of this map. By \cite[\href{https://stacks.math.columbia.edu/tag/051T}{Tag 051T}]{stacks-project} we can, after shrinking on $\Sp (R)$, arrange that $M$ is flat over $R$, in which case the base change of the sequence
    \[
        \begin{tikzcd}
            0\arrow{r}&B\arrow{r}{(g_1,\ldots,g_m)}&B^{\oplus m}\arrow{r}& M\arrow{r}& 0
        \end{tikzcd}
    \]
    along any map $R\rightarrow A$ remains exact.
\end{proof}

\begin{lem}\label{L:9.26} Let $R\rightarrow B$ be a finite type morphism of noetherian rings.  Let $J\subset B$ be an ideal, let $Q$ be a $B$-module annihilated by some power of $J$, and for  $n\geq 1$ let $B_n$ denote $B/J^n$.
\begin{enumerate}
    \item There exists an integer $N\geq 1$ such that for all $n\geq N$ the map
    \[
        \Ext^1_B(Q, B)\rightarrow \Ext^1_B(Q, B_n)
    \]
    is injective.
    \item If $R$ is furthermore assumed reduced then for every $n\geq N$ there exists a dense open affine subscheme $\Sp (R')\subset \Sp (R)$ (depending on $n$) such that for any $R'$-algebra $A$ the maps
    \[
        \Ext^1_B(Q, B)\otimes _RA\rightarrow \Ext^1_B(Q, B\otimes _RA)
    \]
    and
    \[
        \Ext^1_B(Q, B_n)\otimes _RA\rightarrow \Ext^1_B(Q, B_n\otimes _RA)
    \]
    are isomorphisms, and the restriction map
    \[
        \Ext^1_B(Q, B\otimes _RA)\rightarrow \Ext^1_B(Q, B_n\otimes _RA)
    \]
    is injective.
\end{enumerate}
\end{lem}
\begin{proof}
Let $m$ be a constant such that $J^mQ = 0$, and let $F_\bullet \rightarrow Q$ be a free resolution of $Q$ so that for a $B$-module $M$ we have $\Ext^i_B(Q, M)= \rH ^i(F_\bullet ^\vee \otimes _BM).$  Let $Z^1$ denote the kernel of the map $F_1^\vee \rightarrow F_2^\vee $.  By the Artin--Rees lemma \cite[\href{https://stacks.math.columbia.edu/tag/00IN}{Tag 00IN}]{stacks-project} there exists a constant $c$ such that for $n\geq c$ 
\[
    \ker(J^nF_1^\vee \rightarrow J^nF_2^\vee ) = Z^1\cap J^nF_1^\vee \subset J^{n-c}(Z^1\cap J^cF_1^\vee ).
\]
This implies that the image of the map
\[
    \Ext^1_B(Q, J^n)\rightarrow \Ext^1_B(Q, A)
\] 
is contained in $J^{n-c}\Ext^1_B(Q, A)$. Therefore if $n-c\geq m$ then this image is $0$, since $Q$ is annihilated by $J^m$, and then considering the long exact sequence of Ext-groups associated to the sequence
\[
    0\rightarrow J^n\rightarrow A\rightarrow A_n\rightarrow 0
\]
we conclude that (1) holds with $N = c+m$.

To get statement (2) choose $R'$ such that the cohomology groups of the complexes 
\[
    F_0^\vee \rightarrow F_1^\vee \rightarrow F_2^\vee, \ \ F_0^\vee \otimes _BB_n\rightarrow F_1^\vee \otimes _BB_n\rightarrow F_2^\vee \otimes _BB_n
\]
restrict to flat $R'$-modules over $R'$, which implies that the cohomology groups of these complex commutes with arbitrary base change $R'\rightarrow A$, and furthermore the cokernel of the map of Ext-groups in (1) is flat over $R'$.
\end{proof}

\begin{prop}\label{P:9.8}
Let $f\colon X\rightarrow S$ be a proper morphism of noetherian schemes with $S$ reduced. Let $\varphi\colon\bG\rightarrow\bH$ be a morphism of commutative finite locally group schemes over $X$, and assume that $\varphi$ is an isomorphism over the complement of a nowhere dense closed subscheme $Z\subset X$. Set $\mls G:=\rH_{\bG^{\ani}}$, $\mls H:=\rH_{\bH^{\ani}}$, write also $\varphi\colon\mls G\to\mls H$ for the induced map, and set
\[
    \mls K:=\cone(\mls G\xrightarrow{\varphi}\mls H)\in\Shvf(X^{\ani}).
\]
If the conclusion of \ref{T:8.1} holds for $(Z,\mls K_Z)$, then the conclusion of \ref{T:8.1} holds for $(X,\mls G)$ if and only if it holds for $(X,\mls H)$.
\end{prop}
\begin{proof}
This is very similar to the proof of \ref{P:9.5b}.
Using \ref{prop:reduced subscheme lemma} we may without loss of generality assume that $X$ is reduced, $f$ is flat, and that $S = \Sp (R)$ is affine.
Let $U$ denote $X\setminus Z$ and let $j\colon U\hookrightarrow X$ be the inclusion.  After shrinking further on $\Sp (R)$ we may also assume by \ref{lem:schematically dense} that for every $R$-algebra $A$ the base change $j_A\colon U_A\hookrightarrow X_A$ of $j$ to $A$ is schematically dense.

Let $\mls I\subset\mls O_X$ be the ideal defining $Z$. For $n\geq 1$ let $Z_n\subset X$ be the subscheme defined by $\mls I^n$ and let $a_{n}\colon Z_n\hookrightarrow X$ denote the inclusion. Consider the diagram
\begin{equation}\label{eq:two squares 9.24}
    \begin{tikzcd}
        \mls G\arrow{d}[swap]{\varphi}\arrow{r}&\widehat{\mls G}\arrow{d}[swap]{\widehat{\varphi}}\arrow{r}&a_{Z_n*}^{\ani}\mls G_{Z_n}\arrow{d}[swap]{\varphi_n}\\
        \mls H\arrow{r}&\widehat{\mls H}\arrow{r}&a_{Z_n*}^{\ani}\mls H_{Z_n}
    \end{tikzcd}
\end{equation}
where the completions are with respect to $Z$, and the induced maps
\[
    \mls K=\cone(\varphi)\iso\cone(\widehat{\varphi})\to\cone(\varphi_n).
\]
Here the left map is an isomorphism because by \ref{T:5.8} the left square is a homotopy pushout square. We make the following claim.

\begin{claim}\label{C:9.25}
    There exists an integer $N\geq 1$ such that for any $n\geq N$, after replacing $\Spec R$ with a dense affine open subscheme we have that the restriction of the square
    \begin{equation}\label{eq:square 9.25}
        \begin{tikzcd}
            f^{\ani}_*\widehat{\mls G}\arrow{d}{\widehat{\varphi}}\arrow{r}&f^{\ani}_*(a^{\ani}_{Z_n*}\mls G_{Z_n})\arrow{d}{\varphi_n}\\
            f^{\ani}_*\widehat{\mls H}\arrow{r}&f^{\ani}_*(a^{\ani}_{Z_n*}\mls H_{Z_n})
        \end{tikzcd}
    \end{equation}
    (obtained by applying $f^{\ani}_*$ to the right square of~\eqref{eq:two squares 9.24}) to $\Alg_R$ is homotopy cartesian.
\end{claim}

As in the proof of \ref{P:9.5b}, this will imply the result. Indeed, by \ref{lem:reduction reduction} and our assumption that the conclusion of \ref{T:8.1} holds for $(Z, \mls K_Z)$ we know that the conclusion of \ref{T:8.1} also holds for $\cone(\varphi_n)$ for all $n$. Given the claim, we will have that the composition
\[
    f^{\ani}_*\cone(\varphi)\iso f^{\ani}_*\cone(\widehat\varphi)\to f^{\ani}_*\cone(\varphi_n)
\]
restricts to an isomorphism over a dense open, and so the proposition will follow from \ref{C:9.21}.

We prove \ref{C:9.25}. Consider the B\'egueri resolutions for $\bG$ and $\bH$. To distinguish between the two, we denote them by the top and bottom row of the diagram
\[
    \begin{tikzcd}
        0\arrow{r}&\bG\arrow{r}\arrow{d}{\varphi}&\bB^0\arrow{r}\arrow{d}{\varphi^0}&\bB^1\arrow{r}\arrow{d}{\varphi^1}&0\\
        0\arrow{r}&\bH\arrow{r}&\bC^0\arrow{r}&\bC^1\arrow{r}&0,
    \end{tikzcd}
\]
where the vertical morphisms are isomorphisms over the complement of $Z$. Consider the maps on coordinate rings 
\[
    (\varphi^{j})^*\colon \mls O_{\bC^j}\rightarrow \mls O_{\bB^j}
\]
induced by $\varphi^j$. Since $X$ is reduced, this map is injective. Fix a coherent subsheaf $\mls M_j\subset \mls O_{\bB^j}$ which generates $\mls O_{\bB^j}$ as an $\mls O_X$-algebra and set $\mls N_j=\mls O_{\bC^j}\cap \mls M_j$. After possibly enlarging $\mls M_j$ we may assume that $\mls N_j$ also generates $\mls O_{\bC^j}$ as an $\mls O_X$-algebra. Let $\mc Q_j$ denote the quotient $\mls M_j/\mls N_j$.
Fix a finite cover $X = \cup _{i=1}^s\Sp (B_i)$. By \ref{L:9.26} we can, after replacing $\Sp (R)$ by an affine dense open subset, find an integer $N\geq 1$ such that for every $R$-algebra $A$ and flat morphism $\Sp (C)\rightarrow \Sp (B_i\otimes _RA)$ the following hold for $j=0,1$:
\begin{enumerate}
    \item The map $\text{Ext}^1_C(\mc Q_j|_{\Sp (C)}, C)\rightarrow \text{Ext}^1_C(\mc Q_j|_{\Sp (C)}, C_n)$ is injective, where $C_n$ denotes the quotient of $C$ by the ideal defining the preimage of $Z_n$.
    \item The preimage of $X\setminus Z$ in $\Sp (C)$ is schematically dense in $\Sp (C)$.
\end{enumerate}
Take $n\geq N$. As in the proof of \ref{P:9.5b}, to show that the restriction of the square~\eqref{eq:square 9.25} to $\Alg_R$ is homotopy cartesian it suffices to show that the square
\begin{equation}\label{eq:diagramBC}
    \begin{tikzcd}
        \bB^j(\widehat{C})\arrow{d}\arrow[two heads]{r}&\bB^j(C_n)\arrow{d}\\
        \bC^j(\widehat{C})\arrow[two heads]{r}&\bC^j(C_n)
    \end{tikzcd}
\end{equation}
is cartesian for $j=0,1$, where $\widehat C$ is the completion of $C$. Note that the horizontal morphisms are surjective since $\bB^j$ and $\bC^j$ are smooth. Let $M_j$ (resp. $N_j$) denote the $C$-module corresponding to $\mls M_j|_{\Sp (C)}$ (resp. $\mls N_j|_{\Sp (C)}$), and fix a homomorphism $c\colon \mls O_{\bC ^j, C}\rightarrow \widehat C$.  Also let $W\subset \Sp (\widehat {C})$ denote the preimage of $Z$, which by (2) is schematically dense.  We then have a commutative diagram of solid arrows
$$
\xymatrix{
\text{Sym}^\bullet N_j\ar@{->>}[d]\ar[r]& \text{Sym}^\bullet M_j\ar@{->>}[d]\\
\mls O_{\bC ^j, C}\ar@{^{(}->}[r]\ar[d]^-c& \mls O_{\bB ^j, C}\ar@{-->}[ld]\ar[ddl]^-{b_W}\\
\widehat C\ar@{^{(}->}[d]& \\
\Gamma (W, \mls O_W),}
$$
from which it follows that $c$ extends to an algebra homomorphism $b\colon\mls O_{\bB ^j, C}\rightarrow \widehat C$ if and only if the restriction $c|_{N_j}\colon N_j\rightarrow \widehat C$ extends to a $C$-module homomorphism $M_j\rightarrow \widehat C$.  Equivalently, if and only if the pushout of the short exact sequence
$$
0\rightarrow N_j\rightarrow M_j\rightarrow \mc Q_j|_{\Sp (C)}\rightarrow 0
$$
along $c|_{N_j}\colon N_j\rightarrow \widehat C$ splits.  This follows from (1) and the fact that the image of $c$ in $\bC^j(C_n)$ lifts to $\bB^j(C_n)$. We conclude that the square~\eqref{eq:diagramBC} is cartesian for $j=0,1$. This completes the proof of Claim \ref{C:9.25} and hence Proposition \ref{P:9.8}.
\end{proof}

\begin{prop}[Invariance under blowing up]\label{prop:blow up inductive step v2}
    Let $f\colon X\to S$ be a projective morphism of noetherian schemes of characteristic $p$ with $S$ reduced. Let $\pi\colon X'\rightarrow X$ be a blowup (in the ordinary sense) of $X$ along a nowhere dense closed subscheme $Z\subset X$ and let $f'\colon X'\rightarrow S$ be the composition of $f$ with $\pi$. Suppose that $f$ has generic relative dimension $d$ and that the conclusion of \ref{T:8.1} holds for all pairs $(Y,\mls G)$ where $Y\to S$ is a projective morphism of relative dimension $\leq d-1$ and $\mls G\in\Shvf(Y^{\ani})$. Then for $\mls F\in\Shvf(X^{\ani})$, the conclusion of \ref{T:8.1} holds for $(X,\mls F)$ if and only if it holds for $(X',\mls F|_{X'})$.
\end{prop}
\begin{proof}
    By \ref{lem:animated coho of complex} and \ref{lem:observation in reductions} it suffices to consider the case when $\mls F=\rH_{\bG^{\ani}}$ for a finite locally free group scheme $\bG$ over $X$. We may assume that $S$ is affine. Let $\mls I\subset\mls O_X$ be the ideal defining $Z$. As $X$ is projective over $S$ we may find an ample invertible sheaf $\mls L$ on $X$ and a surjection 
    \[
        (\mls L^{\otimes n})^{\oplus r}\twoheadrightarrow\mls I
    \]
    for some $r$ and $n$, and so extend $\pi$ to a derived blowup say $\pi^{\ani}\colon X^{\prime\,\ani}\to X$. Let $Y:=\pi_0(X^{\prime\,\ani})$ be the underlying scheme of the derived blowup. We obtain a diagram
    \[
        \begin{tikzcd}
            Y\arrow[hook]{r}&X^{\prime\,\ani}\arrow{d}{\pi^{\ani}}\\
            X'\arrow{r}{\pi}\arrow[hook]{u}&X.
        \end{tikzcd}
    \]
    Define $Z^{\ani}$ as in~\eqref{pg:derived blowups, again}. For $i\geq 0$ let $Z_i$ denote the derived scheme $\tau_{\leq i}Z^{\ani}$. We have a sequence of maps
    \[
        Z = Z_0\hookrightarrow Z_1\hookrightarrow \cdots \hookrightarrow Z_n = Z^\ani.
    \]
    Each map $Z_i\hookrightarrow Z_{i+1}$ is a square-zero thickening in the sense of \cite[5.1.9]{CS}, with ideal $\mls J_n$ a bounded complex of coherent sheaves on $Z$. As $Z$ has relative dimension $\leq d-1$ over $S$, our assumptions imply that the conclusion of \ref{T:8.1} holds for $(Z,\mls F|_Z)$. Arguing as in the proof of \ref{prop:reduced subscheme lemma} using \cite[5.2.8]{CS} we find that there exists a dense open subscheme $U\subset S$ over which the cocone of the natural map 
    \[
        f_{Z^\ani *}^\ani \mls F|_{Z^\ani }\to f_{Z*}^{\ani}\mls F|_Z
    \]
    restricts to an object of $\Shvf(U^{\ani})$.
     Applying \ref{lem:derived blow up pushout} and sheafifying, and using \ref{lem:observation in reductions}, we obtain that the conclusion of \ref{T:8.1} holds for $(X,\mls F)$ if and only the same is true for the cohomology of $\mls F|_{X^{\prime\,\ani}}$. Arguing as above, this is the case if and only if the conclusion of \ref{T:8.1} holds for $(Y,\mls F|_Y)$. Now, by
    \ref{L:9.7} the scheme $Y=\pi _0(X^{\prime\,\ani})$ underlying $X^{\prime\,\ani}$ can be decomposed as a union
    \[
        Y=X'\cup\mathbf{P}^r_Z\cup T
    \]
    for some integer $r$ and closed subscheme $T$ of dimension $<\mathrm{dim}(X)$. By our assumptions the conclusion of \ref{T:8.1} holds for $(T,\mls F|_{T})$. Applying \ref{prop:main rep thm, animated version} to $\mls F|_{\mathbf{P}^r_Z}$ and the morphism $\mathbf{P}^r_Z\to Z$ and using our assumption again, we see that the same is true for $(\mathbf{P}_Z,\mls F|_{\mathbf{P}^r_Z})$. By \ref{P:9.5b} (and using \ref{prop:reduced subscheme lemma} if needed to pass to reduced subschemes) we get that the conclusion of \ref{T:8.1} holds for $(Y,\mls F|_Y)$ if and only if it holds for $(X',\mls F|_{X'})$.
\end{proof}

The following result will allow us to reduce the case of a general finite flat group scheme to the cases of $\bmu_p$, $\balpha_p$, and $\bZ/q$.

\begin{prop}\label{P:8.13}
Let $f\colon X\rightarrow S$ be a finite type morphism of noetherian separated integral schemes of characteristic $p$ and let $\bG$ be a commutative finite flat group scheme over $X$. There exists a commutative diagram 
\begin{equation}\label{E:6.6.1}
\begin{tikzcd}
    X'\arrow{r}\arrow{d}[swap]{f'}&X\arrow{d}{f}\\
    S'\arrow{r}{g}&S
\end{tikzcd}
\end{equation}
with the following properties:
\begin{enumerate}
    \item [(i)] The morphism $g\colon S'\rightarrow S$ is a dominant quasi-finite morphism of schemes with $S'$ regular.
    \item [(ii)] The morphism $f'\colon X'\rightarrow S'$ admits a factorization
    \[
        \begin{tikzcd}
            X'\arrow{r}{d}&X^{\prime\prime}\arrow{r}&S',
        \end{tikzcd}
    \]
    where $d$ is a blowup and $X^{\prime\prime}\to S'$ is smooth and quasi-projective.
    \item [(iii)] The induced morphism $X'\rightarrow X\times _SS'$ admits a factorization
    \[
        \begin{tikzcd}
            X'\arrow{r}{a}&W\arrow{r}{b}&Y\arrow{r}{c}&X\times_SS',
        \end{tikzcd}
    \]
    where $a$ is a finite flat surjection which identifies $W$ with the quotient $X'/H$ of $X'$ by an action of a finite group $H$, $b$ is a finite flat surjection inducing a purely inseparable field extension $k(W)/k(Y)$, and $c$ is the blowup of a nowhere dense closed subscheme $Z\subset X\times _SS'$.
    \item [(iv)] The pullback $\bG'$ of $\bG$  to $X'$ admits a filtration
    \[
        0=\bG_{-1}\subset\bG_0\subset \bG_1\subset \cdots \subset \bG_s=\bG'
    \]
    where each $\bG_i$ is a commutative finite flat group scheme and for each of the successive quotients $\bG_{i}/\bG_{i-1}$ there exists a group scheme $\bH_i$ over $X'$ which is isomorphic to either $\bmu_p$, $\balpha_p$, or $\mathbf{Z}/q$ for a prime $q$ (possibly equal to $p$) and a morphism $\bH_i\to\bG_{i}/\bG_{i-1}$ which is an isomorphism over a dense open subscheme of $X'$.
\end{enumerate}
\end{prop}
\begin{proof}
Let $K$ denote the function field of $X$, let $K\hookrightarrow \overline K$ be an algebraic closure, and let $\bG_{\overline K}$ be the pullback of $\bG$ to $\overline K$. The simple objects in the category of commutative finite flat group schemes over $\overline K$ are of the form $\bmu_p$, $\balpha_p$, and $\mathbf{Z}/q$ for a prime $q$. We therefore have a filtration
\[
    0=\bG_{-1,\overline{K}}\subset\bG_{0, \overline K}\subset \bG_{1, \overline K}\subset \cdots \subset \bG_{s, \overline K}\subset \bG_{\overline K}
\]
such that for each $i$ the quotient $\bG_{i, \overline K}/\bG_{i-1, \overline K}$ is isomorphic to one of $\bmu_p$, $\balpha_p$, or $\mathbf{Z}/q$ for a prime $q$. By a standard limit argument this filtration descends to $\bG_L$ for a finite extension $L/K$.  Let $V\rightarrow X$ be the normalization of $X$ in $L$.  Let $\bG_{i, V}\subset \bG_V$ be the scheme-theoretic closure of $\bG_{i, L}$ in $\bG_V$. By Raynaud--Gruson \cite[I, 5.2.2]{RG} there exists a blowup $V'\rightarrow V$ such that the strict transform $\bG_{i, V'}$ of $\bG_{i, V}$ is a subgroup scheme of $\bG_{V'}$ flat over $V'$. After a further blowup up we may arrange for the existence of a morphism of group schemes $\bH_i\to\bG_i/\bG_{i-1}$ for each $i$ as in the conclusion of (iv).

Let $\kappa _S$ denote the function field of $S$ and let $X_\eta $, $V_\eta $, etc. denote the generic fibers. Arguing as above we can by  \cite[7.3]{deJong} find a finite extension $E/\kappa _S$ and a generically finite proper morphism $X^{\prime \prime }_\eta \rightarrow V_\eta '\otimes _{\kappa _S}E$ with $X^{\prime \prime }_\eta $ smooth and projective over $E$ equipped with an action of a finite group $H$ such that the induced map $X^{\prime \prime }_\eta /H\rightarrow V_\eta '\otimes _{\kappa _S}E$ is generically purely inseparable.  
Applying another blowup $Y\rightarrow X$ and setting $X'$ equal to the strict transform of $X^{\prime \prime }$ we then obtain the desired diagram \eqref{E:6.6.1} over the generic point of $S$.  Spreading out, we then obtain the conclusion of the proposition.

\end{proof}

\begin{pg}{\bf Proof of \ref{T:8.1}.} 
We now prove \ref{T:8.1}. Let $f\colon X\to S$ be a projective morphism of noetherian schemes of characteristic $p$ with $S$ reduced and $\mls F\in\Shvf(X^{\ani})$.
    We begin with some reductions. Let $d$ be the generic relative dimension of $f$ (the maximum of the relative dimensions of $f$ over the generic points of $S$). Note that if $d=0$ then the result is tautological owing to how we have defined the categories $\Shvf(X^{\ani})$. By induction on $d$ we may therefore assume that \ref{T:8.1} holds for morphisms of generic relative dimension $\leq d-1$. By \ref{prop:reduced subscheme lemma} we may assume that $X$ is reduced, and applying \ref{P:9.5b}, our induction hypothesis, and using \ref{prop:reduced subscheme lemma}, we may assume that $X$ is integral. By shrinking on $S$ we may assume that $S$ is integral and separated, hence $X$ is also separated, and that $f$ is flat. Finally, by \ref{lem:animated coho of complex} and \ref{lem:observation in reductions} it suffices to consider the case when $\mls F=\rH_{\bG^{\ani}}$ for a finite locally free group scheme $\bG$ on $X$.
\end{pg}

\begin{pg}\label{pg:last step in proof} 
It will suffice to prove the result after base change by a dominant quasi-finite morphism $S'\to S$, as such a map is generically flat because $S$ is reduced. Applying \ref{P:8.13} and replacing $S$ by such an $S'$, we may find a smooth projective morphism $X^{\prime\prime}\to S$ and a commutative diagram
\begin{equation}\label{E:7.18.1}
\begin{tikzcd}
    X'\arrow{r}{a}\arrow{d}[swap]{d}&W\arrow{r}{b}&Y\arrow{r}{c}&X\arrow{d}{f}\\
    X^{\prime \prime}\arrow{rrr}&&&S
\end{tikzcd}
\end{equation}
where $a$ is a finite flat surjection which identifies $W$ with the quotient $X'/H$ of $X'$ by an action of a finite group $H$, $b$ is a finite flat surjection which induces a purely inseparable extension of function fields, $c$ and $d$ are blowups, and the pullback $\bG|_{X'}$ of $\bG$ to $X'$ admits a filtration as in \ref{P:8.13} (iv). We will show that the conclusion of \ref{T:8.1} holds for $(X,\mls F)$.
By \ref{prop:blow up inductive step v2}, the conclusion of \ref{T:8.1} holds for $(X,\mls F)$ if and only if it holds for $(Y,\mls F|_Y)$, so we may further assume that $Y = X$. As $X^{\prime\prime}$ is smooth over $S$, \ref{prop:main rep thm, animated version} show that the conclusion of \ref{T:8.1} holds for $(X^{\prime\prime},\rH_{\bG^{\ani}})$ for $\bG$ any of the $X^{\prime\prime}$-group schemes $\bmu_p$, $\balpha_p$, or $\bZ/q$ for a prime $q$. By the induction hypothesis and \ref{prop:blow up inductive step v2} the conclusion of \ref{T:8.1} holds also for the animated sheaves associated to any of these group schemes on $X'$ and the morphism $f'\colon X'\to S$. By \ref{P:9.8} the conclusion of \ref{T:8.1} holds for the animated sheaves associated to each of the successive quotients $(\bG_i/\bG_{i-1})|_{X'}$, and thus the conclusion \ref{T:8.1} holds for $(X',\mls F|_{X'})$. By construction, the map $f'$ factors as the composition of the finite flat surjection $X'\to X$ followed by $f\colon X\to S$. For $n\geq 0$ set
\[
    X'_n:= \underbrace{X'\times_{X}X'\times_{X}\cdots \times _{X}X'}_{n+1}
\]
and let $f'_n\colon X'_n\to S$ be the structure morphism. There is a spectral sequence relating the cohomologies $f^{\prime\,\ani}_{n*}\mls F|_{X'_n}$ to $f^{\ani}_*\mls F$. From this we see that to prove that the conclusion of \ref{T:8.1} holds for $(X,\mls F)$ it will suffice to prove the same for $(X'_n,\mls F|_{X'_n})$ for each $n\geq 0$.

For this we proceed as follows. For $n\geq 0$ set        
\[
    \widetilde X_n:= \underbrace{X'\times _{W}X'\times _{W}\cdots \times _{W}X'}_{n+1}\quad\text{and}\quad W_n:= \underbrace{W\times _{X}W\times _{X}\cdots \times _{X}W}_{n+1}
\]
    and consider the morphism
    \[
        z_n:\coprod _{(g_1, \dots, g_n)\in H^n}X'\rightarrow X'_n
    \]
    given by taking the coproduct of the maps
    \[
        X'\rightarrow X'_n, \ \ w\mapsto (x, g_1x, \dots, g_nx).
    \]
    Since $H$ acts on $X'$ over $W$ the map $z_n$ admits a factorization of the form
    \[
        \begin{tikzcd}
            \coprod_{H^n}X'\arrow{r}{\widetilde{z}_n}\arrow[bend right=25]{rr}[swap]{z_n}&\widetilde{X}_n\arrow[hook]{r}&X'_n
        \end{tikzcd}
    \]
    where $\widetilde{z}_n$ is an isomorphism over a dense open subscheme of the target and $\widetilde{X}_n\hookrightarrow X'_n$ is a closed immersion defined by a nilpotent ideal. We have already shown that the conclusion of \ref{T:8.1} holds for $(X',\mls F|_{X'})$, which implies the same for $(\coprod_{H^n}X',\mls F|_{\coprod_{H^n}X'})$ for each $n\geq 0$. By \ref{prop:blow up inductive step v2} this implies that the conclusion of \ref{T:8.1} holds for $(\widetilde{X}_n,\mls F|_{\widetilde{X}_n})$. Using that $\widetilde{X}_n\hookrightarrow X'_n$ is in infinitesimal thickening and \ref{prop:reduced subscheme lemma} we conclude that the conclusion of \ref{T:8.1} holds for $(X'_n,\mls F|_{X'_n})$ for each $n\geq 0$. This completes the proof of \ref{T:8.1}.
    
\qed
\end{pg}

\section{Compactly supported cohomology}\label{S:compact}

While not strictly necessary for the proofs of the main results,  it is natural to consider a notion of compactly supported flat cohomology, which we now define, and the generalization \ref{T:5.12} of \ref{T:8.1}.  This notion of compactly supported cohomology is the natural extension of the theory for \'etale sheaves and the theory for coherent sheaves considered in \cite{Hartshornecompact}.


\begin{defn} Let $X$ be a noetherian algebraic space and let $Z\subset X$ be a closed subspace. For $\mls F\in \Shvf(X^{\ani})$ we define the \emph{compactly supported cohomology of $\mls F$ with respect to $Z$}, denoted $\R\Gamma _c((X, Z), \mls F)$, by
\[
    \R\Gamma _c((X, Z), \mls F):= \R\Gamma (X, \mls S_{(X,Z)}(\mls F)),
\]
where $\mls S_{(X, Z)}(\mls F)$ is defined as in \ref{P:Scomplexdef}.
If $f\colon X\rightarrow S$ is a morphism of algebraic spaces we also consider the relative version, denoted $f_!^{(X, Z)}\mls F$, defined by
\[
    f_!^{(X, Z)}\mls F:=f_*^\ani \mls S_{(X,Z)}(\mls F)\in \Pro(\Shv(S^{\ani})).
\]
\end{defn}
 
It follows from the definition that there is a distinguished triangle
\[
    \R\Gamma _c((X, Z), \mls F)\rightarrow \R\Gamma (X, \mls F)\rightarrow \R\Gamma (X, \widehat {\mls F})\rightarrow \R\Gamma _c(X, \mls F)[1].
\]

\begin{rem} Observe that in the case when $\mls F$ is the functor associated to a coherent sheaf our definition recovers the definition of compactly supported cohomology in \cite{Hartshornecompact}.
\end{rem}

\begin{notation}
    Let $\Shv_{\dag}(X^{\ani})\subset\Pro(\Shv(X^{\ani}))$ denote the smallest full stable $\infty$-subcategory containing the image of $\Shvf(X^{\ani})$ under~\eqref{eq:yoneda embedding in pro category} and the image of $\Pro(\Shv^{\pf}(X^{\ani}))$ under~\eqref{eq:embedding of pro category 2}, where $\Shv^{\pf}(X^{\ani})\subset\Shv(X^{\ani})$ is the full stable $\infty$-category generated by the animated sheaves associated to perfect complexes on $X$.
\end{notation}

\begin{thm}\label{T:5.12} Let $f\colon X\rightarrow S$ be a projective morphism of noetherian schemes of characteristic $p$ with $S$ reduced and let $Z\subset X$ be a closed subscheme. 
    \begin{enumerate}
        \item For any $\mls F\in \Shvf(X^{\ani})$ there exists a dense open subscheme $U\subset S$ such that $f_*^\ani \widehat {\mls F}$ restricts to an object of $\Shv_\dag (U^\ani )$.
        \item For any $\mls F\in \Shvf(X^{\ani})$ there exists a dominant flat quasi-finite morphism $U\rightarrow  S$ such that the restriction of $f_!^{(X, Z)}(\mls F)$ to $U$ lies in $\Shv_\dag (U^{\ani})$.
    \end{enumerate}
\end{thm}
\begin{proof}
Statement (2) follows from (1) by considering the
 fiber sequence
\[
    f^{(X,Z)}_{!}\mls F\to f^{\ani}_*\mls F\to f^{\ani}_*\widehat{\mls F}\to f^{(X,Z)}_{!}\mls F[1]
\]
in $\Pro (\Shv(S^{\ani}))$ and \ref{T:8.1}.

To prove statement (1) note that the collection of objects $\mls F\in \Shvf(X^{\ani})$ for which the conclusion of (1) holds is a full stable $\infty $-subcategory so it suffices to verify the statement for $\mls F = w^\ani _*\rH $ for a finite morphism $w\colon W\rightarrow X$ and $\rH = \rH _{\mls V }^\ani $ for a perfect complex $\mls V $ on $W$ and $\rH = \rH _{\bG ^\ani }$ for a finite locally free group scheme $\bG /W$.  Replacing $X$ by $W$ and using \ref{L:8.7b} we are further reduced to the case when $W= X$.

Shrinking on $S$ if necessary we may assume that the following hold:
\begin{enumerate}
    \item [(i)] $f\colon X\rightarrow S$ and $f_Z\colon Z\rightarrow S$ are flat. 
    \item [(ii)] $S$ is regular.  
    \item [(iii)] If $\mls J\subset \mls O_X$ denotes the ideal sheaf of $Z$ in $X$ then $\mls J^n/\mls J^{n+1}$ is flat over $S$ for all $n\geq 0$.  Indeed consider the scheme $\widetilde Z:=\underline {\Sp }_Z(\oplus _{n\geq 0}\mls J^n/\mls J^{n+1})$.  Then to arrange that $\mls J^n/\mls J^{n+1}$ is flat over $S$ for all $S$ it suffices to shrink on $S$ so that the finite type $S$-scheme $\widetilde Z$ is flat over $S$.
    \item [(iv)] $f_{Z*}^\ani \rH|_Z\in \Shvf(S^{\ani})$ (using \ref{T:8.1}).
    \end{enumerate}
    Let $Z_n\subset X$ denote the closed subscheme defined by $\mls J^n$ and let $f_{Z_n}\colon Z_n\rightarrow S$ be the restriction of $f$.  To prove the theorem it suffices to show that under the above conditions the fiber of the restriction map $f_{Z_n*}^\ani \rH |_{Z_n}\rightarrow f_{Z*}^\ani \rH |_Z$ is in $\Shv^{\pf}(S^\ani)$, and for this in turn it suffices to show that the fiber $\mc K_n$ of the map $f_{Z_n*}^\ani \rH |_{Z_n}\rightarrow f_{Z_{n-1}*}^\ani \rH |_{Z_{n-1}}$ is the sheaf associated to a perfect complex.  Using cohomology and base change \cite[\href{https://stacks.math.columbia.edu/tag/07VK}{Tag 07VK}]{stacks-project}, in the case when $\rH = \rH _{\mls V }^\ani $ the sheaf $\mls K_n$ is the animated sheaf associated to $\R f_{*} (\mls V \lotimes \mls J^{n-1}/\mls J^n)$, and in the case when $\rH  = \rH _{\bG ^\ani }$ for a finite locally free group scheme $\bG /X$ we have by \cite[5.2.8]{CS} that $\mls K_n$ is the animated sheaf associated to $\R f_{*}(\RHom (\ell_{\bG /X}, \mls J^{n-1}/\mls J^n)$. As $S$ is assumed regular, any bounded complex of coherent sheaves is perfect. Thus, in both cases we conclude that $\mls K_n$ is the animated sheaf associated to a perfect complex on $S$.
\end{proof}

\section{Examples}\label{sec:examples}
In this section we record some examples of flat cohomology groups, illustrating our main results.

\begin{example}[Flat cohomology of $\balpha_p$]\label{ex:cohomology of alpha_p}
    Let $k$ be a field of characteristic $p$ and let $f\colon X\to\Spec k$ be a proper $k$-scheme. As described in \ref{pg:methods of proof}, the flat cohomology of $\balpha_p$ may be computed in terms of the coherent cohomology of $\mls O_X$ using the long exact sequence obtained by applying $\R f_*$ to the short exact sequence~\eqref{eq:alpha_p sequence, first appearence}, combined with the canonical isomorphisms $\bV(\rH^n(X,\mls O_X))\simeq\R^nf_*\bG_a$. For example, if $X=E$ is an elliptic curve, then $\rH^n(E,\mls O_E)=0$ for $n\geq 2$, hence $\R^nf_*\balpha_p=0$ for $n\geq 3$, and in low degrees we have $f_*\balpha_p\simeq\balpha_p$ and an exact sequence
    \[
        0\to\R^1f_*\balpha_p\to\R^1f_*\bG_a\xrightarrow{\F}\R^1f_*\bG_a\to\R^2f_*\balpha_p\to 0.
    \]
    As $\rH^1(E,\mls O_E)$ is one dimensional, we have $\R^1f_*\bG_a\simeq\bG_a$. If $E$ is ordinary then $\F$ is an isomorphism, so $\R^1f_*\balpha_p=\R^2f_*\balpha_p=0$, and if $E$ is supersingular then $\F=0$, so we have $\R^1f_*\balpha_p\simeq\R^2f_*\balpha_p\simeq\bG_a$ as group schemes.
\end{example}

\begin{example}[Cohomology of curves]
    Let $C$ be a curve over an algebraically closed field $k$ of characteristic $p$ with structural morphism $f\colon C\to\Spec k$ and let $q$ be a nonzero integer (possibly divisible by $p$). The flat cohomology groups $\rH^n(C,\bmu_q)$ are well known, and may be computed using the Kummer sequence and Tsen's theorem. We give a similar computation of the flat cohomology sheaves $\R^nf_*\bmu_q$. Applying $\R f_*$ to the Kummer sequence on $C$, we obtain $\R ^0f_*\bmu_q\simeq\bmu_q$ and an exact sequence
    \[
        0\to\R^1f_*\bmu_q\to\Pic_C\xrightarrow{\cdot q}\Pic_C\to \R^2f_*\bmu_q\to\R^2f_*\bG_m\to\ldots.
    \]
    This gives $\R^1f_*\bmu_q\simeq\Pic_C[q]$. We claim that the above induces an isomorphism $\Pic_C/q\simeq\R ^2f_*\bmu_q$ of group schemes, and hence $\R ^2f_*\bmu_q\simeq\bZ/q$ as group schemes, and that $\R ^nf_*\bmu_q=0$ for $n\geq 3$. For this we note that by \ref{C:1} the sheaves $\R ^nf_*\bmu_q$ are representable by finite type $k$-group schemes for all $n\geq 0$. Moreover, to show that a map $\bH\to\bH'$ of finite type $k$-group schemes is an isomorphism, it suffices to check that it induces a bijection $\bH(A)\iso\bH'(A)$ on $A$-points for all artinian local $k$-algebras $A$. Thus, to complete our computation it will suffice to show that $(\R ^nf_*\bG_m)(A)=0$ for all $n\geq 2$ and all artinian local $k$-algebras $A$. For this we use Grothendieck's theorem \cite[Th\'{e}or\`{e}me 11.7]{MR0244271} that the flat cohomology of $\bG_m$ agrees with its \'{e}tale cohomology and the vanishing $\rH^n_{\et}(C,\bG_m)=0$ for $n\geq 2$ \cite[03RM]{stacks-project} to obtain $(\R ^nf_*\bG_m)(k)=\rH^n(C,\bG_m)=0$ for all $n\geq 2$. Furthermore, if $A$ is an artinian local $k$-algebra, then this combined this with the vanishing of the coherent cohomology of $C_A$ in degrees $\geq 2$ we obtain $\rH^n(C_A,\bG_m)=0$ for $n\geq 2$. We now use that any fppf cover $U\to\Spec A$ where $A$ is an artinian local $k$-algebra may be refined by an fppf cover $\Spec B\to\Spec A$ where $B$ is also an artinian local $k$-algebra \cite[0DET]{stacks-project}. It follows that $(\R ^nf_*\bG_m)(A)=0$ for any such $A$, as claimed.
\end{example}

\begin{example}[Cohomology of K3 surfaces]\label{ex:k3 surfaces}
    Let $X$ be a K3 surface over an algebraically closed field $k$ of characteristic $p$. The structure of the flat cohomology groups $\rH^n(X,\bmu_p)$ is well known (see e.g. \cite[\S2.7]{MR4585353}). We describe the flat cohomology sheaves $\R ^nf_*\bmu_p$ in terms of invariants of $X$, focusing on the case $n=2$. Let $\bH$ be a group scheme representing $\R ^2f_*\bmu_p$, let $\bU\subset\bH$ be the connected component of the identity, and let $\bD:=\bH/\bU$ be the component group, so that we have a short exact sequence
    \begin{equation}\label{eq:UHL}
        0\to\bU\to\bH\to\bD\to 0.
    \end{equation}
    We first consider the connected component $\bU$. We recall that the formal Brauer group $\widehat{\Br}(X)$ of $X$ is a 1-dimensional smooth formal group, hence isomorphic as a formal scheme to $\Spf k[[t]]$, and that the group structure on $\widehat{\Br}(X)$ is determined by a discrete invariant $h$, known as the \emph{height} $h$ of $X$, which is either an integer $1\leq h\leq 10$ or is $\infty$ (see \cite[18, \S3]{MR3586372}). Furthermore, taking cohomology of the Kummer sequence on $X$ yields the exact sequence
    \[
        0\to\Pic_{X}/p\to\R ^2f_*\bmu_p\to(\R ^2f_*\bG_m)[p]\to 0.
    \]
    The Picard scheme of a K3 surface is discrete, so taking completions at the identity yields isomorphisms
    \[
        \widehat{\bU}\simeq\widehat{\R^2f_*\bmu_p}\simeq\widehat{\Br}(X)[p].
    \]
    From this we can determine the isomorphism type of $\bU$ in terms of the height of $X$. Indeed, if $h<\infty$, then the group $\rH^2(X,\bmu_p)$ is finite. Thus, $\bU$ is a finite purely infinitesimal $k$-group scheme, and so by the above is isomorphic to the $p$-torsion of the formal group law of height $h$. For example, when $h=1$ we have $\widehat{\Br}(X)\simeq\widehat{\bG}_m$, so $\bU\simeq\bmu_p$, and when $h=2$ we have $\widehat{\Br}(X)\simeq\widehat{E}$ for a supersingular elliptic curve $E$, and therefore $\bU\simeq E[p]$. When $h=\infty$, we have $\widehat{\Br}(X)\simeq\widehat{\bG}_a$, which implies that $\widehat{\bU}\simeq\widehat{\bG}_a$ and therefore $\bU\simeq\bG_a$.
    
    \begin{rem}
        If $k$ is not algebraically closed then $\bU$ for a K3 surface with $h=\infty$ need only be a $k$-form of $\bG_a$. An example where this form is nontrivial is when $k$ is the generic point of the moduli space of lattice polarized supersingular K3 surfaces and $X$ is the universal family (see \cite[3.1.16, 3.5.10]{BL}). In particular, over non closed fields the isomorphism type of $\bU$ is not determined by the height.
    \end{rem}

    To complete our calculation we consider the component group $\bD$. If $h<\infty$ then $\bU(k)$ is trivial, so we have that $\bD(k)$ is isomorphic to $\bH(k)=(\R^2f_*\bmu_p)(k)\simeq\rH^2(X,\bmu_p)\simeq(\bZ/p)^{\oplus 22-2h}$. As $\bD$ is a finite \'{e}tale $k$-group this determines its structure as a group scheme. Noting that the reduction map gives a canonical splitting of~\eqref{eq:UHL}, we conclude that when $h<\infty$ we have a canonical isomorphism
    \[
        \bH\simeq\bU\times\left(\bZ/p\right)^{\oplus 22-2h}
    \]
    of group schemes, where $\bU$ is isomorphic to the $p$-torsion in the 1-dimensional formal group law of height $h$. When $h=\infty$, it follows from computations of Artin \cite[\S4]{ArtinSSK3} that $\bD(k)\simeq(\bZ/p)^{\oplus 22-2\sigma_0}$, where $\sigma_0$ is the Artin invariant of $X$ (a certain discrete invariant of $X$ satisfying $1\leq\sigma_0\leq 10$), and we conclude that $\bH$ sits in a short exact sequence of group schemes
    \[
        0\to\bG_a\to\bH\to\left(\bZ/p\right)^{\oplus 22-2\sigma_0}\to 0.
    \]
\end{example}

\begin{example}\label{ex:r^2 f_*}
    Let $f\colon X\to S$ be a smooth proper morphism of schemes of characteristic $p$ with geometrically connected fibers, let $q$ be a nonzero integer (possibly divisible by $p$), and assume that $\Pic_{X/S}[q]$ is flat over $S$. These hypothesis hold, for example, if $f$ is a family of K3 surfaces, abelian varieties, or Enriques surfaces. We then have $\R^0f_*\bmu _{q} = \bmu _{q}$, by our assumptions on the fibers of $f$, and by Kummer theory we have $\R^1f_*\bmu _{q}=\Pic_{X/S}[q]$. Applying \ref{cor:global rep corollary} with $\bG=\bmu_{q}$, we deduce that $\R^2f_*\bmu_{q}$ is representable.
\end{example}

    \begin{example}[Non-representable cohomology in families]\label{ex:non representable cohomology in families}
        Without the flatness assumption in Corollary \ref{cor:global rep corollary} the conclusion can fail. For an example with $\bG=\balpha_p$, consider a family $f\colon E\to S$ of elliptic curves. Suppose that the fiber $E_s$ is supersingular for some $s\in S$ whose complement is dense in $S$ and is ordinary otherwise. Taking cohomology of the defining exact sequence for $\balpha_p$ gives an exact sequence
            \[
                0\to\R^1f_*\balpha_p\to\R^1f_*\bG_a\xrightarrow{\F}\R^1f_*\bG_a\to\R^2f_*\balpha_p\to 0.
            \]
        The pushforward $\R^1f_*\bG_a$ is a one dimensional vector group over $S$, and the map $\F$ vanishes at $s$ and is invertible otherwise. Thus, $\R^1f_*\balpha_p$ is non-flat and supported only at $s$, and $\R^2f_*\balpha_p$ is not representable.
            
        For an example with $\bG=\bmu_p$, consider a family $f\colon X\to S$ of K3 surfaces. Then $\R^2f_*\bmu_p$ is representable (by Example \ref{ex:r^2 f_*}), but need not be flat over $S$. For instance, $\R^2f_*\bmu_p$ has relative dimension 0 over the finite height locus, and has relative dimension 1 over the supersingular locus. If the family $f$ involves a specialization from the finite height locus to the supersingular locus, then one can check that $\R^3f_*\bmu_p$ is indeed not representable.
    \end{example}


\providecommand{\bysame}{\leavevmode\hbox to3em{\hrulefill}\thinspace}
\providecommand{\MR}{\relax\ifhmode\unskip\space\fi MR }
\providecommand{\MRhref}[2]{%
  \href{http://www.ams.org/mathscinet-getitem?mr=#1}{#2}
}
\providecommand{\href}[2]{#2}

\end{document}